\documentclass[11pt]{article}
\usepackage{amsfonts}
\usepackage{amsmath}
\usepackage{amsthm}
\usepackage{amssymb}
\usepackage{array}
\usepackage{enumerate}
\usepackage{microtype}
\usepackage{tikz}
\usepackage{tkz-berge}
\usepackage{mathtools}
\usepackage[left=2in, right=1.5in, top=1.75in, bottom=1.75in]{geometry}

\theoremstyle{plain}
\newtheorem{theorem}{Theorem}[section]
\newtheorem{corollary}[theorem]{Corollary}
\newtheorem{lemma}[theorem]{Lemma}

\newtheorem{question}{Question}
\newtheorem{conjecture}[theorem]{Conjecture}
\theoremstyle{definition}
\newtheorem{definition}[theorem]{Definition}

\begin{document}

\pagenumbering{roman}

\thispagestyle{empty}
\newgeometry{left=1.5in,right=1in}
\linespread{1.5}
\textsc{
\vspace*{0in}
\begin{center}
\LARGE
The Graceful Tree Conjecture: \\
A Class of Graceful Diameter-6 Trees
\end{center}
\vspace{.6in}
\begin{center}
Matthew C.\ Superdock
\end{center}
\vspace{.6in}
\begin{center}
A Senior Thesis \\
Submitted to the Department of Mathematics of Princeton University \\
in Partial Fulfillment of the Requirements for the Degree of \\
Bachelor of Arts.
\end{center}
\vspace{.3in}
\begin{center}
Adviser: Shiva Kintali
\end{center}
\vspace{.3in}
\begin{center}
May 2013
\end{center}
}
\linespread{1}
\clearpage

\newpage
\vspace*{\fill}
\begin{center}
This thesis represents my own work in accordance with University regulations.
\end{center}

\begin{center}
/s/ Matthew C.\ Superdock
\end{center}
\vspace*{\fill}

\clearpage
\restoregeometry

\newpage
\addcontentsline{toc}{section}{Abstract}
\begin{center}
\Large \textbf{Abstract}
\end{center}
We survey the current state of progress on the Graceful Tree Conjecture, 
and then we present several new results toward the conjecture, driven by three new ideas:

{\raggedright
\begin{enumerate}
\item It has been proven that \emph{generalized banana trees} are graceful by rearranging the branches at the root.

\hspace{0.5in} ---Consider rearranging branches at all internal vertices.\\
\item The method of \emph{transfers} has typically involved \emph{type-1 transfers} and \emph{type-2 transfers}.

\hspace{0.5in} ---All type-2 transfers are type-1 transfers in disguise, and\\
\hspace{1in}hence can be removed from the discussion.\\
\item The method of \emph{transfers} has typically used the sequence of transfers
$$0\rightarrow n\rightarrow 1\rightarrow n - 1\rightarrow\cdots$$
\hspace{0.5in} ---Transfer backwards to manipulate the resulting labels.\\
\end{enumerate}
}

\noindent Using these ideas, we prove that several classes of diameter-6 trees are graceful, and we generalize some of these classes to larger trees.  We also introduce a class of graceful spiders, prove that attaching sufficiently many leaves to any tree gives a graceful tree, and extend known results on trees with perfect matchings.

\clearpage

\newpage
\addcontentsline{toc}{section}{Acknowledgements}
\begin{center}
\Large \textbf{Acknowledgements}
\end{center}
This thesis represents a year of God's faithfulness; he called me to work on this particular conjecture against conventional wisdom, brought about timely conversations that helped me approach the conjecture more effectively, and guided me to my key insight in a remarkable way.

I sincerely thank my adviser, Shiva Kintali, for introducing me to the conjecture and for guiding me through the research process; his patience and his generosity with his time have been invaluable.

I thank Robert Gunning for his sage advice throughout my undergraduate career, but especially this year; one conversation particularly altered my approach to the problem and helped spark my first bit of progress.

I thank Richard Zhang for helping me edit this thesis, and I thank the many people who have been encouraged me and prayed for me throughout this process, especially my family, the Bromans, the '13, and Sara Marie.

I give an especially heartfelt thanks to Allie, who has infused a dangerous amount of joy into my life these past few months, and has already changed me in surprising ways.

\clearpage



\tableofcontents
\clearpage

\pagenumbering{arabic}

\parindent=1.5em

\section{Introduction}


The Graceful Tree Conjecture, conjectured by Kotzig (see Bermond \cite{survey0}) and made famous by Rosa's 1967 paper \cite{rosa}, has drawn much attention for its remarkable simplicity.  We can introduce the conjecture informally as follows:

Consider a tree with $n$ edges and $n + 1$ vertices.  Label the vertices with the distinct integers $0, 1, \ldots , n$, and then label each edge with the positive difference between the labels of its ends.  If the edge labels are the distinct integers $1, \ldots , n$, the vertex labeling is called a \emph{graceful labeling}, and the Graceful Tree Conjecture states that every tree has such a labeling.\newline

Let $G = (V(G), E(G))$ be a simple, undirected graph, where $V(G)$ is the vertex set of $G$, and $E(G)$ is the edge set of $G$.  We define a \emph{labeling} of $G$ in general, and then we define four particular types of labelings of $G$, before narrowing our focus to trees.

\begin{definition}
A \emph{labeling} $f$ of $G$ is an injective mapping of the vertices of $G$ to the non-negative integers.  Under the labeling $f$, the \emph{label} of a vertex $v$ is $f(v)$, and the \emph{induced label} of an edge $uv$ is the absolute difference of the labels of its ends, $|f(u) - f(v)|$.
\end{definition}

Let $G$ be a graph with $n$ edges, and let $f$ be a labeling of $G$.  Let $V_{f}$ be the set of vertex labels under $f$, and let $E_{f}$ be the set of induced edge labels under $f$.  We now define our four particular types of labelings, first introduced by Rosa \cite{rosa}.  (Note that each of these types of labelings implicitly requires the induced labels of the edges to be distinct.)

\begin{definition}
The labeling $f$ is a \emph{$\rho$-labeling} if $V_{f}\subseteq\{0, 1, \ldots , 2n\}$ and $E_{f} = \{x_{1}, \ldots , x_{n}\}$, where $x_{i} = i$ or $x_{i} = (2n + 1) - i$ for each $i$.
\end{definition}

\begin{definition}
The labeling $f$ is a \emph{$\sigma$-labeling} if $V_{f}\subseteq\{0, 1, \ldots , 2n\}$ and $E_{f} = \{1, \ldots , n\}$.
\end{definition}

\begin{definition}
The labeling $f$ is a \emph{graceful labeling} if $V_{f}\subseteq\{0, 1, \ldots , n\}$ and $E_{f} = \{1, \ldots , n\}$.  The graph $G$ is \emph{graceful} if it has a graceful labeling $f$.
\end{definition}

\begin{definition}
The labeling $f$ is an \emph{$\alpha$-labeling} with \emph{index} $k$ if $f$ is graceful, and $k$ is an integer such that for each edge $uv$, either $f(u)\le k< f(v)$ or $f(v)\le k < f(u)$.
\end{definition}

We have the following immediate consequences (Rosa \cite{rosa}):

\begin{itemize}
\item The labelings form a hierarchy.  Specifically, all $\alpha$-labelings are graceful labelings, all graceful labelings are $\sigma$-labelings, and all $\sigma$-labelings are $\rho$-labelings.
\item Only graphs with at most $n + 1$ vertices can have graceful labelings.
\item Only bipartite graphs can have $\alpha$-labelings.
\end{itemize}

All trees $T$ with $n$ edges have exactly $n + 1$ vertices, so any graceful labeling $f$ of $T$ must have $V_{f} = \{0, 1, \ldots , n\}$.  In other words, any graceful labeling of a tree $T$ uses exactly the integers $\{0, 1, \ldots , n\}$ as vertex labels.

We may now state the Graceful Tree Conjecture, conjectured by Kotzig (see Bermond \cite{survey0}) and referenced by Rosa \cite{rosa} in 1967.  The conjecture has also been referred to as Rosa's conjecture, the Ringel-Kotzig conjecture, or the Ringel-Kotzig-Rosa conjecture (see Alfalayleh et.\ al.\ \cite{survey1}).

\begin{conjecture}
All trees are graceful.
\end{conjecture}

This conjecture was originally motivated by problems in graph decompositions.  To state these problems, we need to introduce a few definitions.  In the following definitions, let $G^{*}$ be a graph with vertices $v_{0}, \ldots , v_{n^{*}}$.  (We use $G^{*}$ in order to reserve $G$ for subgraphs of $G^{*}$, since we will be considering labelings of subgraphs of $G^{*}$.)

\begin{definition}
A \emph{decomposition} of $G^{*}$ is a set $S$ of subgraphs of $G^{*}$ such that each edge of $G^{*}$ belongs to exactly one subgraph in $S$.
\end{definition}

\begin{definition}
To \emph{turn} an edge $v_{i}v_{j}$ of $G^{*}$ is to increase both indices by one, resulting in the edge $v_{i + 1}v_{j + 1}$, where we take indices modulo $n^{*} + 1$.  To \emph{turn} a subgraph $G$ of $G^{*}$ is to simultaneously turn all edges of $G$.
\end{definition}

\begin{definition}
A decomposition $S$ of $G^{*}$ is \emph{cyclic} if turning any subgraph in $S$ gives another subgraph in $S$.
\end{definition}

Two natural questions arise from these definitions, for general $G^{*}, G$:

\begin{itemize}
\item Does $G^{*}$ have a decomposition into subgraphs isomorphic to $G$?
\item Does $G^{*}$ have a cyclic decomposition into subgraphs isomorphic to $G$?
\end{itemize}

\noindent In the case where $G$ is a tree $T$ with $n$ edges, and $G^{*}$ is the complete graph $K_{2n + 1}$, the following conjectures have been made, the second a stronger version of the first.

\begin{conjecture} \label{ringel}
{\normalfont (Ringel \cite{ringel1963problem})} Let $T$ be a tree with $n$ edges.  Then the complete graph $K_{2n + 1}$ has a decomposition into subgraphs isomorphic to $T$.
\end{conjecture}

\begin{conjecture} \label{kotzig}
{\normalfont (Kotzig, see Rosa \cite{rosa})} Let $T$ be a tree with $n$ edges.  Then the complete graph $K_{2n + 1}$ has a cyclic decomposition into subgraphs isomorphic to $T$.
\end{conjecture}

The following theorems, which we state without proof, establish the connection between graph decompositions and the labelings defined above.

\begin{theorem}
{\normalfont (Rosa \cite{rosa})} Let $G$ be a graph with $n$ edges.  Then the complete graph $K_{2n + 1}$ has a cyclic decomposition into subgraphs isomorphic to $G$ if and only if $G$ has a $\rho$-labeling.
\end{theorem}

\begin{theorem}
{\normalfont (Rosa \cite{rosa})} Let $G$ be a graph with $n$ edges.  If $G$ has an $\alpha$-labeling, then the complete graph $K_{2kn + 1}$ has a cyclic decomposition into subgraphs isomorphic to $G$ for each $k\in\mathbb{N}$.
\end{theorem}


Therefore, if every tree $T$ has a $\rho$-labeling, then Conjecture \ref{kotzig} holds, and hence Conjecture \ref{ringel} also holds.  Since any graceful labeling is also a $\rho$-labeling,
the Graceful Tree Conjecture also implies both of these conjectures.

Though originally motivated by these conjectures in graph decompositions, the Graceful Tree Conjecture has since taken on a life of its own.  There is now general interest in solving it for its own sake, and the conjecture has inspired the study of other graph labelings, most notably harmonious labelings (Graham \& Sloane \cite{graham1980additive}) and cordial labelings (Cahit \cite{cahit1987cordial}).\newline

%

Now we show several standard ways of transforming graceful labelings and $\alpha$-labelings into new graceful labelings and $\alpha$-labelings.  We are mainly concerned with graceful labelings, but also with $\alpha$-labelings, since a later theorem will allow us to obtain a graceful tree by combining two trees, one with a graceful labeling and one with an $\alpha$-labeling.

\begin{lemma}
Let $f$ be a graceful labeling of $G$.  Then the labeling $f'$ defined by $f'(v) = n - f(v)$ is also a graceful labeling of $G$.
\end{lemma}

\begin{proof}
Since $f$ is injective and has range $\{0, 1, \ldots , n\}$, $f'$ is also injective and has range $\{0, 1, \ldots , n\}$.  The induced edge labels are the same under $f, f'$, so $f'$ is graceful.
\end{proof}

\begin{definition}
The labeling $f'$ produced by the lemma above is called the \emph{complementary labeling} of $f$.
\end{definition}

\begin{lemma}
Let $f$ be an $\alpha$-labeling of $G$ with index $k$.  Then the complementary labeling $f'$ is also an $\alpha$-labeling of $G$ with index $n - k - 1$.
\end{lemma}

\begin{proof}
By the previous lemma, $f'$ is graceful.  Suppose $uv$ is an edge of $G$.  Since $f$ is an $\alpha$-labeling, we have, without loss of generality, $f(u)\le k < f(v)$.  Then $f'(v) < n - k\le f'(u)$, so that $f'(v)\le n - k - 1 < f'(u)$.  Hence $f'$ is an $\alpha$-labeling of $G$ with index $n - k - 1$.
\end{proof}

\begin{lemma} \label{inverse}
{\normalfont (Koh, Rogers \& Tan \cite{koh2})} Let $f$ be an $\alpha$-labeling of $G$.  Then the labeling
$$f'(v) = \left\{
     \begin{array}{lr}
       k - f(v) & \text{if } f(v)\le k\\
       n + k + 1 - f(v) & \text{if } f(v) > k
     \end{array}
   \right.
$$
is also an $\alpha$-labeling with index $k$.
\end{lemma}

\begin{proof}
Suppose $uv$ is an edge of $G$.  Since $f$ is an $\alpha$-labeling, we have, without loss of generality, $f(u)\le k < f(v)$.  Then $f'(u)\le k < f'(v)$, and
\begin{align*}
f'(v) - f'(u) &= (n + k + 1 - f(v)) - (k - f(u))\\
&= (n + 1) - (f(v) - f(u))
\end{align*}
Therefore, the induced labels of the edges under $f'$ are also $\{1, \ldots , n\}$, so $f'$ is a graceful labeling and hence an $\alpha$-labeling with index $k$.
\end{proof}

\begin{definition}
The labeling $f'$ produced by the lemma above is called the \emph{inverse labeling} of $f$.
\end{definition}


\section{Known Results}

\subsection{Paths}

Let $P_{n}$ be the path with $n$ edges.

\begin{theorem} \label{path}
{\normalfont (Rosa \cite{rosa})} All paths are graceful.
\end{theorem}

\begin{proof}
Label $P_{n}$ by starting at one end of the path and alternating between the least and greatest remaining label along the path, so that the labels are
$$0,\; n,\; 1,\; n - 1,\; \ldots$$
\end{proof}

It is natural to ask, which vertices of $P_{n}$ can receive which labels in a graceful labeling?  This is an interesting problem for general trees $T$, and an especially important problem for the label 0 (or equivalently, by taking complementary labelings, the label $n$), since a later theorem will allow us to combine any tree with an $\alpha$-labeling and any tree with a graceful labeling by identifying the vertices labeled 0, and obtain a larger graceful tree.

\begin{definition}
(Chung \& Hwang \cite{rotatable}, Bloom \cite{bloom1979chronology}) A tree $T$ is \emph{$0$-rotatable} if, for each vertex $v$ of $T$, there exists a graceful labeling $f$ of $T$ with $f(v) = 0$.
\end{definition}

For paths, in most cases we can find a graceful labeling such that a chosen vertex receives a chosen label, as the following theorems show.

\begin{theorem} \label{paths_0-rotatable}
{\normalfont (Rosa \cite{rosa2})} The path $P_{n}$ is $0$-rotatable for all $n$.
\end{theorem}

\begin{theorem}
{\normalfont (Flandrin, Fournier \& Germa \cite{path-labelings2})} Let $n\ge 8$.  Then for all vertices $v$ of $P_{n}$, and for all $k$ with $0\le k\le n$, there exists a graceful labeling $f$ of $P_{n}$ with $f(v) = k$.
\end{theorem}

It has been shown that the number of graceful labelings of $P_{n}$ grows asymptotically at least as fast as $(5/3)^{n}$ (Aldred, Siran \& Siran \cite{path-asymptotics}), and the bound has since been improved to $(2.37)^{n}$ (Adamaszek \cite{path-asymptotics2}).


\begin{theorem}
{\normalfont (Rosa \cite{rosa})} All paths have an $\alpha$-labeling.
\end{theorem}

\begin{proof}
The graceful labeling of $P_{n}$ in the proof of Theorem \ref{path} is an $\alpha$-labeling with index $\lfloor n/2\rfloor$.
\end{proof}

As with graceful labelings, it is natural to ask, which vertices of $P_{n}$ can receive which labels in an $\alpha$-labeling?  It is useful to formulate this problem in terms of bipartitions of $T$, since the possible labels for a vertex depend on the size of its part.

Let $T$ be a tree, and let $V_{1}, V_{2}$ be vertex sets of a bipartition of $T$.  If $T$ has an $\alpha$-labeling, then either
\begin{itemize}
\item The vertices in $V_{1}$ are labeled $\{0, \ldots , |V_{1}| - 1\}$, and the vertices in $V_{2}$ are labeled $\{n - |V_{2}| + 1, \ldots , n\}$, or
\item The vertices in $V_{1}$ are labeled $\{n - |V_{2}| + 1, \ldots , n\}$, and the vertices in $V_{1}$ are labeled $\{0, \ldots , |V_{2}| - 1\}$.
\end{itemize}
If the $\alpha$-labeling is of the first type, then the complementary $\alpha$-labeling is of the second type, and vice versa, so that there is a one-to-one correspondence between the $\alpha$-labelings of each type.

Consider a vertex $v$ of $T$, and suppose $v\in V_{i}$.  The vertices in $V_{i}$ are labeled either $\{0, \ldots , |V_{i}| - 1\}$ or $\{n - |V_{i}| + 1, \ldots , n\}$.  We would like to know which labels $v$ can take in each of these cases.  But it suffices to answer this question in just one case, since our one-to-one correspondence will then give the corresponding answer in the other case.  So we may assume that the vertices in $V_{i}$ are labeled $\{0, \ldots , |V_{i}| - 1\}$.

These considerations motivate the following definition:

\begin{definition}
(Kotzig \cite{depth}) Let $T$ be a tree, and let $V_{1}, V_{2}$ be vertex sets of a bipartition of $T$.  Let the \emph{potential depth set} $D_{p}(v)$ of a vertex $v$ of $T$ be the set $\{0, \ldots , |V_{i}| - 1\}$, where $v\in V_{i}$.  Let the \emph{depth set} $D(v)$ of $T$ be the subset of elements of $D_{p}(v)$ which are the label of $v$ in some $\alpha$-labeling of $T$.
\end{definition}

With this framework in place, we now present a complete answer to the question of which vertices of $P_{n}$ can receive which labels in an $\alpha$-labeling.
As with graceful labelings, we give special attention to the label 0.

\begin{theorem} \label{zero_alpha}
{\normalfont (Rosa \cite{rosa2})}  Let $v$ be a vertex of the path $P_{n}$.  There exists an $\alpha$-labeling $f$ of $P_{n}$ with $f(v) = 0$, unless $v$ is the central vertex of $P_{4}$.
\end{theorem}

\begin{theorem}
{\normalfont (Cattell \cite{path-labelings})} Let $v$ be a vertex of the path $P_{n}$.  Then $D(v) = D_{p}(v)$ except in the following cases:
\begin{itemize}
\item If $v$ is a leaf of $P_{4k}$, then $D(v) = D_{p}(v)\backslash k$.
\item If $v$ is the central vertex of $P_{4}$, then $D(v) = D_{p}(v)\backslash 0$.
\item If $v$ is adjacent to a leaf of $P_{6}$, then $D(v) = D_{p}(v)\backslash 1$.
\end{itemize}
\end{theorem}


\subsection{Caterpillars}

\begin{definition}
(see Morgan \cite{lobsters-perfect}) The tree $T$ is \emph{$m$-distant} if there exists a path $P$, such that all vertices of $T$ are a distance at most $m$ from $P$.
\end{definition}

\begin{definition}
A \emph{caterpillar} is a 1-distant tree.
\end{definition}




\begin{theorem}
{\normalfont (Rosa \cite{rosa})} All caterpillars have $\alpha$-labelings.
\end{theorem}

\begin{proof}
Let $T$ be a caterpillar, and let $P$ be a path of $T$ such that each vertex of $T$ is a distance at most 1 from $P$.  Let $v_{0}, \ldots , v_{k}$ be the vertices of $P$, in order.  We start at one end of the path and alternate between small and large labels along the path, also labeling the non-path vertices as we go.  The first few steps are as follows:
\begin{itemize}
\item Give $v_{0}$ the label 0.
\item Give the neighbors of $v_{0}$ not in $P$ the largest remaining labels.
\item Give $v_{1}$ the largest remaining label.
\item Give the neighbors of $v_{1}$ not in $P$ the smallest remaining labels.
\item Give $v_{2}$ the smallest remaining label.
\end{itemize}
Continuing in this way gives the desired $\alpha$-labeling.
\end{proof}

As with paths, it is important to determine which vertices of a caterpillar can be labeled 0 in a graceful labeling or $\alpha$-labeling.  Caterpillars are not $0$-rotatable in general; consider for example the tree obtained by replacing an edge of the star $K_{1, 3}$ with a path of length 3 (Duke \cite{duke1969can}).  However, the following results, particularly the first theorem below, will later allow us to attach caterpillars to arbitrary graceful graphs.

\begin{definition}
Let $v$ be a vertex of a graph $G$.  The \emph{eccentricity} of $v$ is the maximum distance, over all vertices $u$ of $G$, between $u$ and $v$.
\end{definition}

\begin{theorem} \label{caterpillar}
{\normalfont (Hrn\u{c}iar \& Haviar \cite{diameter5})} Let $v$ be a vertex of a caterpillar $T$, such that either $v$ has maximum eccentricity or $v$ is adjacent to a vertex of maximum eccentricity.  Then there exists an $\alpha$-labeling $f$ of $T$ with $f(v) = 0$.
\end{theorem}

\begin{proof}
There exists a path $P$ of $T$, starting with $v$, such that all vertices of $T$ are a distance at most 1 from $P$.  Then the construction from the proof above produces an $\alpha$-labeling $f$ with $f(v) = 0$, as desired.
\end{proof}


We end this section by presenting one particular class of 0-rotatable caterpillars, without proof.

\begin{definition}
(Chung \& Hwang \cite{rotatable}) A \emph{$t$-toe caterpillar} is a caterpillar whose internal vertices all have degree exactly $t + 2$.
\end{definition}

\begin{theorem}
{\normalfont (Chung \& Hwang \cite{rotatable})} All $t$-toe caterpillars are 0-rotatable.
\end{theorem}

\subsection{Symmetrical trees}

In this section, we present several important methods of forming larger graceful trees from smaller ones.  All of these methods are based on a single method called the \emph{$\Delta$-construction}, which 
allows us to combine several copies of a graceful tree $T$ to obtain a larger graceful tree.

In this section, we use $n$ to refer to the number of vertices, rather than the number of edges, of a tree, departing from our usual convention.  This makes the theorem statements and proofs slightly cleaner.

\begin{definition}
(Stanton \& Zarnke \cite{stanton1973labeling}, see Koh, Rogers \& Tan \cite{koh3} and Burzio \& Ferrarese \cite{generalized-delta}) Let $S, T$ be two trees, and let $v$ be a vertex of $T$.  Consider attaching one copy of $T$ at each vertex of $S$, by identifying each vertex of $S$ with the vertex corresponding to $v$ in a distinct copy of $T$.  We denote the resulting tree by $S\Delta T$, and we call the construction the \emph{$\Delta$-construction}.

\end{definition}

\begin{theorem} \label{delta}
{\normalfont (Stanton \& Zarnke \cite{stanton1973labeling})} If $S, T$ are graceful, then $S\Delta T$ is graceful.
\end{theorem}

\begin{proof}
Let $f, g$ be graceful labelings of $S, T$, and suppose $S, T$ have $n_{S}, n_{T}$ vertices, so that $S\Delta T$ contains $n_{S}$ copies of $T$.  We construct $n_{S}$ auxiliary labelings $g_{0}, g_{1} \ldots , g_{n_{S} - 1}$ of $T$ by choosing a bipartition $(A, B)$ of $T$, and taking
$$g_{i}(x) = \left\{
     \begin{array}{lr}
       in_{T} + g(x) & \text{if } x\in A\\
       (n_{S} - i - 1)n_{T} + g(x) & \text{if } x\in B
     \end{array}
   \right.
$$
Consider labeling $n_{S}$ copies of $T$ with the labelings $g_{i}$.  It is straightforward to show that the vertex labels are all distinct, that the edge labels are all distinct, and that the sets $\cup V_{g_{i}}, \cup E_{g_{i}}$ of all vertex and edge labels used are
\begin{align*}
\cup V_{g_{i}} &= \{0, 1, \ldots , n_{S}n_{T} - 1\}\\
\cup E_{g_{i}} &= \{1, \ldots , n_{S}n_{T} - 1\}\backslash\{n_{T}, 2n_{T}, \ldots , (n_{S} - 1)n_{T}\}
\end{align*}

Now label $S\Delta T$ by labeling the copy of $T$ at each vertex $x$ of $S$ with the labeling $g_{f(x)}$.  The resulting labeling uses the correct vertex labels by the above.

Consider the edge labels.  For an edge $x_{1}x_{2}$ of $S$, both vertices correspond to the vertex $v$ of $T$, and the induced edge label is
$$|g_{f(x_{1})}(v) - g_{f(x_{2})}(v)| = |f(x_{1}) - f(x_{2})|\cdot n_{T}$$
Therefore, our labeling of $S\Delta T$ assigns the labels $\{n_{T}, 2n_{T}, \ldots , (n_{S} - 1)n_{T}\}$ to the edges in $S$.  By the above, the labeling assigns the other labels in $\{1, \ldots , n_{S}n_{T} - 1\}$ to the edges in the $n_{S}$ copies of $T$.  Therefore, the labeling is graceful, and hence $S\Delta T$ is graceful.
\end{proof}

\begin{definition}
(Stanton \& Zarnke \cite{stanton1973labeling}, see Burzio \& Ferrarese \cite{generalized-delta}) Let $S, T$ be trees, and let $u, v$ be vertices of $S, T$, respectively.  Consider attaching one copy of $T$ at each vertex of $S$ other than $u$, by identifying each vertex of $S$ other than $u$ with the vertex corresponding to $v$ in a distinct copy of $T$.  We denote the resulting tree by $S\Delta_{+1} T$, and we call the construction the \emph{$\Delta_{+1}$-construction}.
\end{definition}

\begin{theorem}
{\normalfont (Stanton \& Zarnke \cite{stanton1973labeling})}  If $S, T$ have graceful labelings $f, g$, with $f(u) = n_{S} - 1$ and $g(v) = 0$, where $n_{S}$ is the number of vertices of $S$, then $S\Delta_{+1} T$ is graceful.
\end{theorem}

\begin{proof}
We construct a labeling analogous to the one in the theorem above.  Suppose $S, T$ have $n_{S}, n_{T}$ vertices, so that $S\Delta T$ contains $n_{S} - 1$ copies of $T$.  Let $(A, B)$ be the bipartition of $T$ with $v\in A$.  We construct $n_{S} - 1$ auxiliary labelings $g_{0}, g_{1}, \ldots , g_{n_{S} - 2}$ of $T$ by taking
$$g_{i}(x) = \left\{
     \begin{array}{lr}
       in_{T} + g(x) & \text{if } x\in A\\
       (n_{S} - i - 2)n_{T} + g(x) & \text{if } x\in B
     \end{array}
   \right.
$$
Consider labeling $n_{S} - 1$ copies of $T$ with the labelings $g_{i}$.  It is straightforward to show that the vertex labels are all distinct, that the edge labels are all distinct, and that the sets $\cup V_{g_{i}}, \cup E_{g_{i}}$ of all vertex and edge labels used are
\begin{align*}
\cup V_{g_{i}} &= \{0, 1, \ldots , (n_{S} - 1)n_{T} - 1\}\\
\cup E_{g_{i}} &= \{1, \ldots , (n_{S} - 1)n_{T} - 1\}\backslash\{n_{T}, 2n_{T}, \ldots , (n_{S} - 2)n_{T}\}\\
&= \{1, \ldots , (n_{S} - 1)n_{T}\}\backslash\{n_{T}, 2n_{T}, \ldots , (n_{S} - 1)n_{T}\}
\end{align*}

Now label $S\Delta_{+1}T$ by labeling the copy of $T$ at each vertex $x$ of $S$ other than $u$ with the labeling $g_{f(x)}$, and by labeling the vertex $u$ with the label $(n_{S} - 1)n_{T}$.  The resulting labeling uses the correct vertex labels by the above.

Consider the edge labels.  For edges in $S$, note that each vertex $x$ of $S$ receives the label $f(x)n_{T}$, so our labeling of $S\Delta_{+1}T$ assigns the labels $\{n_{T}, 2n_{T}, \ldots , (n_{S} - 1)n_{T}\}$ to the edges in $S$.  By the above, the labeling assigns the other labels in $\{1, \ldots ,  (n_{S} - 1)n_{T}\}$ to the edges in the $n_{S} - 1$ copies of $T$.  Therefore, the labeling is graceful, and hence $S\Delta_{+1}T$ is graceful.
\end{proof}

Burzio \& Ferrarese \cite{generalized-delta} generalize each of these constructions one step further.  These generalizations require a shift in thinking.  Below we give the required shift of thinking and the resulting generalization for each construction.

\begin{center}
\emph{The generalized $\Delta$-construction.}
\end{center}

Above, we obtained $S\Delta T$ by attaching $n_{S}$ copies of $T$ to the vertices of $S$.  However, we can also obtain $S\Delta T$ by attaching $n_{S}$ copies of $T$ to each other, according to the structure of $S$.  More precisely, we establish a correspondence between the copies of $T$ and the vertices of $S$, and we connect two copies of $T$ by an edge if the corresponding vertices of $S$ are adjacent, where the ends of this edge are the vertices corresponding to the fixed vertex $v$ of $T$.

In the proof of Theorem \ref{delta}, it is essential that the ends of this edge correspond to the same vertex of $T$, but it is not essential that the ends correspond to a fixed vertex $v$ of $T$ across all such edges.  Therefore, the $\Delta$-construction can be generalized by allowing the ends of each edge connecting two copies of $T$ to be any pair of corresponding vertices in the two copies.

\begin{center}
\emph{The generalized $\Delta_{+1}$-construction.}
\end{center}

Analogously, we can obtain $S\Delta_{+1} T$ by attaching $n_{S} - 1$ copies of $T$, and one single vertex $u$, all to each other, according to the structure of $S$.  We establish the same correspondence as above, in some cases connecting two copies of $T$ by an edge, and in some cases connecting a copy of $T$ with $u$ by an edge.  The ends of these edges that are in copies of $T$ all correspond to the fixed vertex $v$ of $T$.

The $\Delta_{+1}$-construction can be generalized in a similar way to the $\Delta$-construction, but only for the edges connecting two copies of $T$.  For the edges connecting a copy of $T$ with $u$, the end of the edge in the copy of $T$ must be the vertex corresponding to the fixed vertex $v$ of $T$.  But as with the $\Delta$-construction, we can allow the ends of each edge connecting two copies of $T$ to be any pair of corresponding vertices in the two copies.\newline

Before applying some of these constructions, we present a theorem about the $\Delta$-construction.  Chung \& Hwang \cite{rotatable} prove the following theorem for the standard $\Delta$-construction, but their proof also directly extends to the generalized $\Delta$-construction.

\begin{theorem}
{\normalfont (Chung \& Hwang \cite{rotatable})} Let $S, T$ be 0-rotatable trees.  Then $S\Delta T$ is 0-rotatable.
\end{theorem}

\begin{proof}
Let $f, g$ be graceful labelings of $S, T$, and let $x$ be a vertex of $S\Delta T$.  Then $x$ is in one of the $n_{S}$ copies of $T$.  Suppose $x$ is in the copy of $T$ attached at the vertex $x_{S}$ of $S$, and suppose $x$ corresponds to the vertex $x_{T}$ of $T$.  Then the label of $x$ in $S\Delta T$ is $g_{f(x_{S})}(x_{T})$.

Since $S, T$ are 0-rotatable, we can choose $f, g$ such that $f(x_{S}), g(x_{T}) = 0$, and we can choose the bipartition $(A, B)$ such that $x_{T}\in A$.  Then
$$g_{f(x_{S})}(x_{T}) = f(x_{S})\cdot n_{T} + g(x_{T}) = 0$$
Therefore, the label of $x$ in $S\Delta T$ is 0.  Since $x$ is arbitrary, $S\Delta T$ is 0-rotatable.
\end{proof}

With these constructions, we can prove that certain classes of symmetrical trees are graceful.

\begin{definition}
A rooted tree $T$ with root $v$ is \emph{symmetrical} if any two vertices at the same level have the same number of children
\end{definition}

\begin{theorem} \label{symmetrical_graceful}
{\normalfont (Stanton \& Zarnke \cite{stanton1973labeling})} Let $T$ be a symmetrical tree with root $v$.  Then $T$ has a graceful labeling $f$ with $f(v) = 0$.
\end{theorem}

\begin{proof}
Apply the $\Delta_{+1}$-construction repeatedly, with $S$ a star. 
\end{proof}



\begin{corollary}
All complete binary and $k$-ary trees are graceful.
\end{corollary}

Until now, we have worked with the concept of symmetry with respect to a vertex.  In contrast, Poljak \& Sura \cite{Poljak82} introduce the notion of symmetry with respect to a caterpillar.  In the following definitions, let $T$ be a tree with root $v$.

\begin{definition}
(Poljak \& Sura \cite{Poljak82}) Let $C$ be a caterpillar with internal vertices $v_{0}, \ldots , v_{k}$, and let $T_{0}, \ldots , T_{k}$ be rooted trees.  Then consider the tree obtained by identifying the roots of two copies of each $T_{i}$ with $v_{i}$ for each $i$.  If $T$ can be obtained by this construction, with $v_{0}$ corresponding to $v$, then we say that $T$ is \emph{c-symmetrical}.
\end{definition}

If a tree is symmetrical, then it is also c-symmetrical, but the converse is not necessarily true.  Therefore, c-symmetry is a weaker concept than symmetry.

\begin{definition}
For vertices $u, w$ of $T$, the vertex $w$ is a \emph{descendant} of $u$ if $u$ is between $v$ and $w$.  The tree $T$ is \emph{equidescendant} if any two non-leaf vertices equidistant from $v$ have the same number of descendants.
\end{definition}

\begin{theorem}
{\normalfont (Poljak \& Sura \cite{Poljak82})} All equidescendant, c-symmetrical trees are graceful.
\end{theorem}

\subsection{Subdividing edges}

Using the $\Delta_{+1}$-construction, we can show that under certain conditions, subdividing edges of a graceful tree produces another graceful tree.  The following lemma will be useful.

\begin{lemma} \label{subdivide_lemma}
Let $T$ be a tree, and let $M$ be a perfect matching of $T$.  It is possible to direct the edges of $M$ such that the ends of each edge not in $M$ are both heads or both tails of edges in $M$.
\end{lemma}

\begin{proof}
Suppose $T$ has $2n$ vertices.  There is a sequence of subtrees $T_{1}, \ldots , T_{n}$ of $T$ such that
\begin{itemize}
\item $T_{i}$ has $2i$ vertices, and has a perfect matching that is a subset of $M$.
\item $T_{1}$ is a single edge in $M$, and $T_{n}$ is $T$.
\item $T_{i + 1}$ is obtained from $T_{i}$ by attaching a vertex of $T$ adjacent to $T_{i}$, and then attaching the other end of the edge in $M$ incident with that vertex.
\end{itemize}

It is possible to direct the edges of $T_{1}$ in the desired way.  Moreover, if it is possible to direct the edges of $T_{i}$ in the desired way, then it is also possible to direct the edges of $T_{i + 1}$ in the desired way, since we can always direct the new edge in $M$ so that the property holds for the new edge not in $M$.  Therefore, by induction it is possible to direct the edges of all $T_{i}$ in the desired way, in particular $T$ itself, proving the lemma.
\end{proof}

\begin{theorem} \label{subdivide}
{\normalfont (Burzio \& Ferrarese \cite{generalized-delta})} The tree obtained from a graceful tree by subdividing each edge once is graceful.
\end{theorem}

\begin{proof}
The resulting tree can be formed by the generalized $\Delta_{+1}$-construction, where
\begin{itemize}
\item $S$ is the original tree, and $u$ is a vertex of $S$, such that $f(u) = n_{S} - 1$ for some graceful labeling $f$ of $S$.
\item $T$ is the path $P_{1}$, and $v$ is a leaf of $T$.
\end{itemize}

To show that the resulting tree can be decomposed in this way, we prove the following claim, which is different from the lemma above.  Then we will apply the lemma above to show that the copies of $P_{1}$ can be oriented in the appropriate way.\\

\noindent\emph{Claim:  It is possible to direct the edges of $S$ such that no two edges have the same head, and no edge has head $v$.}\\

\noindent\emph{Proof of Claim.}  There is a sequence of subtrees $S_{1}, \ldots , S_{n_{S}}$ of $S$ such that
\begin{itemize}
\item $S_{1}$ is the isolated vertex $u$, and $S_{n_{S}}$ is $S$.
\item $S_{i + 1}$ is obtained from $S_{i}$ by attaching a leaf to $S_{i}$.
\end{itemize}

It is possible to direct the edges of $S_{1}$ in the desired way.  Moreover, if it is possible to direct the edges of $S_{i}$ in the desired way, then it is also possible to direct the edges of $S_{i + 1}$ in the desired way, since we can direct the new edge toward the new leaf.  Therefore, by induction it is possible to direct the edges of all $S_{i}$ in the desired way, in particular $S$ itself, proving the claim.\\

Each edge of the original tree $S$ corresponds to two edges of the resulting tree, one incident with its head and one incident with its tail.  For each edge of $S$, consider the corresponding edge incident with its head.  These edges of the resulting tree are our copies of $P_{1}$ in the $\Delta_{+1}$-construction.

It remains to show that we can orient these copies such that edges between different copies connect corresponding vertices.  To do this, we attach a leaf to the resulting tree at the vertex $u$, resulting in a tree with a perfect matching, and then apply the lemma above.  We orient the copies of $P_{1}$ according to the directions of the edges given by the lemma, as follows:
\begin{itemize}
\item If $u$ is the head of an edge in the matching, then for each copy of $P_{1}$, put the vertex corresponding to $v$ at the head of its edge.
\item If $u$ is the tail of an edge in the matching, then for each copy of $P_{1}$, put the vertex corresponding to $v$ at the tail of its edge.
\end{itemize}
This ensures that the copies of $P_{1}$ are oriented in the appropriate way, so the resulting tree is graceful by the generalized $\Delta_{+1}$-construction.
\end{proof}

\begin{theorem}
{\normalfont (Burzio \& Ferrarese \cite{generalized-delta})} The tree obtained from a graceful tree by replacing each edge with a path of fixed length $k$ is graceful.
\end{theorem}

\begin{proof}
As above, but where $T$ is the path $P_{k}$.  Just as each copy of $P_{1}$ above can be oriented in two ways, each copy of $P_{k}$ can be oriented in two ways.
\end{proof}



\subsection{Trees with perfect matchings} \label{perfect_matching_prelim_section}

Trees with perfect matchings have similar structure to the subdivision graph considered in the section above, leading to the following result.



%
%

\begin{definition} \label{contree}
(Broersma \& Hoede \cite{broersma1999another})  Let $T$ be a tree, and let $M$ be a perfect matching of $T$.  The \emph{contree} of $T$ is the tree obtained from $T$ by contracting the edges in $M$.
\end{definition}

\begin{theorem} \label{broersma's_mistake}
{\normalfont (Broersma \& Hoede \cite{broersma1999another})} If $T$ has a perfect matching, and the contree of $T$ is graceful, then $T$ is graceful.
\end{theorem}

\begin{proof}
Let $S$ be the contree of $T$, and let $n_{S}$ be the number of vertices of $S$.  The tree $T$ can be formed by the generalized $\Delta$-construction, by attaching $n_{S}$ copies of $P_{1}$ to each other, according to the structure of $S$.

As in the proof of Theorem \ref{subdivide}, we must show that we can orient these copies such that edges between different copies connect corresponding vertices.  By Lemma \ref{subdivide_lemma}, it is possible to direct the edges of the perfect matching such that the ends of each edge not in the matching are both heads or both tails of edges in $M$.  We orient the copies of $P_{1}$ according to the directions of the edges.

This ensures that the copies of $P_{1}$ are oriented in the appropriate way, so the resulting tree is graceful by the generalized $\Delta$-construction.
\end{proof}

\begin{corollary} \label{broersma's_redemption}
{\normalfont (Broersma \& Hoede \cite{broersma1999another})} If $T$ has a perfect matching, and the contree of $T$ is a caterpillar, then $T$ is graceful.
\end{corollary}
\noindent We will later apply this corollary to prove that all lobsters with perfect matchings are graceful.

Before continuing, we note a few strange features of Broersma \& Hoede's paper \cite{broersma1999another}:
\begin{itemize}
\item They prove the theorem above within a proof of another result, but never explicitly state the theorem above as a result.
\item They do explicitly state the corollary above as a result, but prove it by a different method.
\end{itemize}
Because of these issues, the theorem above has often been overlooked in subsequent work (see Morgan \cite{lobsters-perfect}).


%

%
%
%
%

\subsection{Trees formed from caterpillars}

The main limitation of the two constructions introduced in the sections above, the $\Delta$-construction and $\Delta_{+1}$-construction, is that the combined trees must be identical.  Here we introduce a construction that allows us to combine two non-identical trees, provided that one has an $\alpha$-labeling.  Since all caterpillars have $\alpha$-labelings, this construction is particularly applicable to trees formed from caterpillars.


\begin{theorem} \label{combine_alpha_alpha}
{\normalfont (Huang, Kotzig \& Rosa \cite{rosa3})}  Let $T_{1}, T_{2}$ be disjoint trees, and let $v_{1}, v_{2}$ be vertices of $T_{1}, T_{2}$, respectively.  Let $f_{1}$ be an $\alpha$-labeling of $T_{1}$ with $f_{1}(v_{1}) = 0$, and let $f_{2}$ be an $\alpha$-labeling of $T_{2}$ with $f_{2}(v_{2}) = 0$.  Then the tree $T$ obtained by identifying $v_{1}$ with $v_{2}$ has an $\alpha$-labeling.
\end{theorem}

\begin{proof}
Suppose $T_{1}, T_{2}$ have $n_{1}, n_{2}$ edges, and suppose $f_{1}$ has index $k_{1}$.  By Lemma \ref{inverse}, the inverse labeling $f_{1}'$ is an $\alpha$-labeling with index $k_{1}$, and we have $f_{1}'(v_{1}) = k_{1}$.  Then consider the labeling
$$g(v) = \left\{
     \begin{array}{lr}
       f_{1}'(v) & \text{if $v\in V(T_{1})$ and $f_{1}'(v)\le k_{1}$}\\
       f_{1}'(v) + n_{2} & \text{if $v\in V(T_{1})$ and $f_{1}'(v) > k_{1}$}\\
       k_{1} + f_{2}(v) & \text{if $v\in V(T_{2})$}
     \end{array}
   \right.
$$
Since $f_{1}'(v_{1}) = k_{1} + f_{2}(v_{2})$, the labeling $g$ is well-defined, and the set $V_{g}$ of vertex labels under $g$ is $\{0, \ldots , n_{1} + n_{2}\}$.  The edges in $T_{2}$ have labels $\{0, \ldots , n_{2} - 1\}$ under $g$, and the edges in $T_{1}$ have labels $\{n_{2}, \ldots , n_{1} + n_{2} - 1\}$ under $g$, since each edge label in $T_{1}$ under $g$ is $n_{2}$ more than the corresponding edge label under $f_{1}'$.  Therefore, $g$ is graceful.

Since $f_{1}, f_{2}$ are $\alpha$-labelings, $g$ is an $\alpha$-labeling with index $k_{1} + k_{2}$.
\end{proof}

\begin{theorem} \label{combine_graceful}
{\normalfont (Huang, Kotzig \& Rosa \cite{rosa3})}  Let $T_{1}, T_{2}$ be disjoint trees, and let $v_{1}, v_{2}$ be vertices of $T_{1}, T_{2}$, respectively.  Let $f_{1}$ be an $\alpha$-labeling of $T_{1}$ with $f_{1}(v_{1}) = 0$, and let $f_{2}$ be a graceful labeling of $T_{2}$ with $f_{2}(v_{2}) = 0$.  Then the tree $T$ obtained by identifying $v_{1}$ with $v_{2}$ is graceful.
\end{theorem}

\begin{proof}
By the same argument as above.
\end{proof}

\begin{corollary} \label{attach_caterpillar}
{\normalfont (See Hrn\u{c}iar \& Haviar \cite{diameter5})}  Let a tree $T$ have a graceful labeling $f$, and let $u$ be a vertex of $T$ with $f(u) = 0$.  Let $H$ be a caterpillar disjoint from $T$, and let $v$ be a vertex of $H$ either of maximum eccentricity or adjacent to a vertex of maximum eccentricity.  Then the tree obtained from $T, H$ by identifying $u$ and $v$ is graceful.
\end{corollary}

\begin{proof}
By Theorem \ref{caterpillar}, there exists an $\alpha$-labeling of $H$ with $v$ labeled 0, so the tree obtained from $T, H$ by identifying $u$ and $v$ is graceful.
\end{proof}

Now we apply these theorems to several classes of trees formed from caterpillars.

\begin{definition}
Let $T$ be a tree, let $v$ be a vertex of $T$, and let $T_{1}, \ldots , T_{k}$ be the connected components of $T\backslash v$.  Then a \emph{branch} of $T$ at $v$ is a subtree of $T$ induced by $V(T_{i})\cup\{v\}$ for some $i$.
\end{definition}

\begin{definition}
(Rosa \cite{rosa}) Let $\mathcal{F}$ be the class of trees $T$ satisfying the following properties, for some vertex $v$ of $T$:
\begin{itemize}
\item All branches of $T$ at $v$ are caterpillars, with $v$ a vertex of maximum eccentricity in each caterpillar.
\item All branches of $T$ at $v$ are isomorphic, except possibly one, where each isomorphism maps $v$ to itself.
\end{itemize}
\end{definition}

\begin{theorem}
{\normalfont (Rosa \cite{rosa})} All trees in $\mathcal{F}$ are graceful.
\end{theorem}

\begin{proof}
Let $T\in\mathcal{F}$, and suppose all branches of $T$ at $v$ are isomorphic to the caterpillar $H$, except possibly one.  The vertex of $H$ corresponding to $v$ is a leaf, so it is adjacent to a unique vertex $u$ of $H$.  We now define two auxiliary trees:
\begin{itemize}
\item Let $H'$ be the tree obtained from $H$ by deleting the vertex corresponding to $v$.  Then either $u$ has maximum eccentricity in $H'$ or $u$ is adjacent to a vertex of maximum eccentricity in $H'$.
\item Let $T'$ be the subtree of $T$ consisting of the branches of $T$ isomorphic to $H$.
\end{itemize}
By Lemma \ref{caterpillar}, $H'$ has a graceful labeling with $u$ labeled 0.  Since we can obtain $T'$ by the $\Delta_{+1}$-construction, by attaching copies of $H'$ to the leaves of a star, the tree $T'$ has a graceful labeling with $v$ labeled with the maximum label.  By taking the complementary labeling, the tree $T'$ has a graceful labeling with $v$ labeled 0.  Then by Corollary \ref{attach_caterpillar}, $T$ is graceful.
\end{proof}


\begin{definition}
(Rosa \cite{rosa}) Let $\mathcal{U}_{3}$ be the class of trees $T$ such that for some vertex $v$ of $T$, two branches of $T$ at $v$ are paths, and the third is a caterpillar.
\end{definition}

\begin{theorem}
{\normalfont (Rosa \cite{rosa})} All trees in $\mathcal{U}_{3}$ are graceful.
\end{theorem}

\begin{proof}
(see Huang, Kotzig \& Rosa \cite{rosa3})  Let $T'$ be the subtree of $T$ consisting of the two branches of $T$ at $v$ which are paths, and let $H$ be the other branch of $T$.  By Theorem \ref{paths_0-rotatable}, $T'$ is 0-rotatable, so $T'$ has a graceful labeling with $v$ labeled 0.  Then by Corollary \ref{attach_caterpillar}, $T$ is graceful.
\end{proof}

\begin{theorem}
{\normalfont (Sethuraman \& Jesintha \cite{sethuraman1})} Let $T$ be a tree obtained from a collection of caterpillars by identifying a chosen vertex of maximum eccentricity of each caterpillar all together, where the caterpillars satisfy the following conditions:
\begin{itemize}
\item All caterpillars have the same diameter $N\ge 3$.
\item All internal vertices of all caterpillars have odd degree.
\end{itemize}
Then $T$ is graceful.
\end{theorem}

\begin{proof}
Using transfers, introduced in the next section.
\end{proof}



\subsection{Banana trees and their generalizations} \label{banana_tree_section}

\begin{definition}
(L\"{u}, Chen \& Yeh \cite{firecrackers}) A \emph{banana tree} is a tree obtained from a collection of stars $K_{1, m_{i}}$ with $m_{i}\ge 1$ by joining a leaf of each star by an edge to a new vertex.  
\end{definition}

\begin{definition}
(Hrn\u{c}iar \& Monoszova \cite{banana}) A \emph{generalized banana tree} is a tree obtained from a collection of stars $K_{1, m_{i}}$ with $m_{i}\ge 1$ by joining a leaf of each star by a path of fixed length $h\ge 0$ to a new vertex, called the \emph{apex}. 
\end{definition}
\noindent We make several observations about this definition:
\begin{itemize}
\item If we consider a generalized banana tree as a rooted tree with apex as root, each leaf is at level $h + 1$ or $h + 2$, depending on whether $m_{i} = 1$ or $m_{i} > 1$.
\item The class of generalized banana trees with $h = 1$ is exactly the class of banana trees.
\end{itemize}

%

In order to prove that generalized banana trees are graceful, we introduce an important 
lemma, which gives rise to the concept of \emph{transfers}.

\begin{lemma} \label{the_classic_lemma}
{\normalfont (Hrn\u{c}iar \& Haviar \cite{diameter5})} Let $T$ be a tree with a graceful labeling $f$, and let $u$ be a vertex adjacent to leaves $u_{1}, u_{2}$.  Let $v$ be a vertex of $T$ other than $u, u_{1}, u_{2}$.
\begin{itemize}
\item Suppose $u_{1}\ne u_{2}$, and let $T'$ be the tree obtained from $T$ by deleting the leaves $u_{1}, u_{2}$ from $u$ and reattaching them at $v$.  If
$$f(u_{1}) + f(u_{2}) = f(u) + f(v)$$
then $f$ is a graceful labeling of $T''$.
\item Suppose $u_{1} = u_{2}$, and let $T'$ be the tree obtained from $T$ by deleting the leaf $u_{1}$ from $u$ and reattaching it at $v$.  If
$$2f(u_{1}) = f(u) + f(v)$$
then $f$ is a graceful labeling of $T'$.
\end{itemize}
\end{lemma}

\noindent We will most often use this lemma as stated, to transfer leaves.  However, the lemma can be generalized as follows, to transfer branches:


\begin{definition}
Suppose $T$ is a tree, and suppose $u, v$ are adjacent vertices of $T$.  Then denote the branch of $T$ at $u$ containing $v$ by $T_{u, v}$.
\end{definition}

\begin{lemma}
{\normalfont (Hrn\u{c}iar \& Haviar \cite{diameter5})} Let $T$ be a tree with a graceful labeling $f$, and let $u$ be a vertex adjacent to vertices $u_{1}, u_{2}$.  Let $T'$ be the subtree of $T$ containing $u$ and the vertices of $T$ not in $T_{u, u_{1}}$ or $T_{u, u_{2}}$.  Let $v$ be a vertex of $T'$ other than $u$.
\begin{itemize}
\item Suppose $u_{1}\ne u_{2}$, and let $T''$ be the tree obtained by identifying $v$ with the vertex $u$ of $T_{u, u_{1}} ,T_{u, u_{2}}$.  If $f(u_{1}) + f(u_{2}) = f(u) + f(v)$, then $f$ is a graceful labeling of $T''$.
\item Suppose $u_{1} = u_{2}$, and let $T''$ be the tree obtained by identifying $v$ with the vertex $u$ of $T_{u, u_{1}}$.  If $2f(u_{1}) = f(u) + f(v)$, then $f$ is a graceful labeling of $T''$.
\end{itemize}
\end{lemma}

We call any transformation corresponding to the lemmas above a \emph{transfer}.
It is often possible to transfer many leaves $\{u_{i}\}$ from $u$ to $v$ at the same time, and we call this a $u\rightarrow v$ \emph{transfer}.
Two important examples are as follows:



\begin{definition}
(Hrn\u{c}iar \& Haviar \cite{diameter5}) A \emph{transfer of the first type} is a $u\rightarrow v$ transfer with
$$\{u_{i}\} = \{k, k + 1, \ldots , k + m\}$$
where $u + v = k + (k + m)$.
\end{definition}

\begin{definition}
(Hrn\u{c}iar \& Haviar \cite{diameter5}) A \emph{transfer of the second type} is a $u\rightarrow v$ transfer with
$$\{u_{i}\} = \{k, k + 1, \ldots , k + m\}\cup\{l, l + 1, \ldots , l + m\}$$
where $u + v = k + l + m$.
\end{definition}

In general, we will often perform the following sequence of transfers, starting with transfers of the first type, and then switching to transfers of the second type.
$$0\rightarrow n\rightarrow 1\rightarrow n - 1\rightarrow \cdots$$
We now make several observations, which we will further explore in section \ref{preliminary}.
\begin{itemize}
\item A transfer of the first type after another transfer of the first type leaves behind an odd number of leaves.
\item A transfer of the second type after another transfer of the second type leaves behind an even number of leaves.
\item A transfer of the second type after a transfer of the first type can leave behind an even or odd number of leaves.
\end{itemize}
With these observations in place, we are ready to prove that all generalized banana trees are graceful.

\begin{theorem} \label{generalized_banana}
{\normalfont (Hrn\u{c}iar \& Monoszova \cite{banana})} All generalized banana trees are graceful.
\end{theorem}

\begin{proof}
Let $T$ be a generalized banana tree with $m$ branches at the apex, and fixed path length $h$.  We have two cases, based on the parity of $m$.\newline

\noindent\emph{Case 1: $m$ is odd.}\newline

We define the following auxiliary trees:
\begin{itemize}
\item Let $S$ be the tree obtained from $m$ paths $P_{h}$ by identifying one leaf of each path all together, and let $v_{1}, \ldots , v_{m}$ be the leaves of $S$.
\item Let $S'$ be the tree obtained from $S$ by attaching leaves $\{u_{i}\}$ to a leaf of $S$, such that $S, S'$ have the same numbers of vertices and edges.
\end{itemize}
We can obtain $S'$ from the gracefully labeled star $K_{1, n}$ with central vertex labeled 0 by a sequence of transfers of the first type
$$0\rightarrow n\rightarrow 1\rightarrow n - 1\rightarrow\cdots$$
This sequence of transfers produces a graceful labeling of $S'$ with the following properties:
\begin{itemize}
\item If $h$ is odd, then for some $a, k$,
\begin{itemize}
\item The vertices $v_{1}, \ldots , v_{m}$ are labeled
$$n - a + k,\; a - k,\; n - a + k - 1,\ldots,\; n - a + 1,\; a - 1,\; n - a$$
\item The leaves $\{u_{i}\}$ are labeled
$$a,\; a + 1,\; \ldots\; ,\; n - a - 2,\; n - a - 1$$
\end{itemize}
\item If $h$ is even, then for some $a, k$,
\begin{itemize}
\item The vertices $v_{1}, \ldots , v_{m}$ are labeled
$$a - k, n - a + k - 1,\; a - k + 1,\ldots,\; a - 1,\; n - a,\; a$$
\item The leaves $\{u_{i}\}$ are labeled
$$a + 1,\; a + 2,\; \ldots\; ,\; n - a - 2,\; n - a - 1$$
\end{itemize}
\end{itemize}

We now obtain $T$ from $S'$ by transfers.  We may assume that the branches of $T$ are arranged such that the vertices of $T$ corresponding to $v_{1}, \ldots , v_{m}$ are in the following order:
\begin{enumerate}[(1)]
\item Vertices adjacent to a positive odd number of leaves.
\item Vertices adjacent to a positive even number of leaves.
\item Vertices not adjacent to any leaves.
\end{enumerate}
Then we can obtain $T$ from $S'$ by a sequence of transfers
$$v_{1}\rightarrow v_{2}\rightarrow\cdots\rightarrow v_{m}$$
Therefore, $T$ is graceful.\newline

\noindent\emph{Case 2: The apex has even degree.}\newline

Let $T'$ be the generalized banana tree obtained from $T$ by deleting a branch $H$ at the apex.  Then the apex of $T'$ has odd degree, so $T'$ has a graceful labeling by the case above, with apex labeled 0.  By Lemma \ref{attach_caterpillar}, since $T$ can be formed by attaching the caterpillar $H$ to $T'$ at the apex, $T$ is graceful.
\end{proof}

\begin{corollary}
All banana trees are graceful.
\end{corollary}

We can also consider replacing the paths of fixed length with caterpillars of fixed diameter.  This consideration motivates the following definition.


\begin{definition}
(Jesintha \& Sethuraman \cite{fixed-banana}) An \emph{arbitrarily fixed generalized banana tree} is a tree obtained from a collection of caterpillars by identifying a chosen vertex of maximum eccentricity of each caterpillar all together, where the caterpillars satisfy the following conditions:
\begin{itemize}
\item All caterpillars have the same diameter $N$.
\item All internal vertices of all caterpillars have fixed even degree $d$, except possibly for the internal vertices with distance $N - 1$ from the vertex chosen for identification.
\end{itemize}
\end{definition}
\noindent We make several observations about this definition:
\begin{itemize}
\item Taking $d = 0$ gives a generalized banana tree.
\item Not all banana trees are arbitrarily fixed generalized banana trees, namely the banana trees with some $m_{i} = 1$.
\end{itemize}
Therefore, the class of arbitrarily fixed generalized banana trees is not a true generalization of generalized banana trees.

\begin{theorem}
{\normalfont (Jesintha \& Sethuraman \cite{fixed-banana})} All arbitrarily fixed generalized banana trees are graceful.
\end{theorem}

\begin{proof}
We adapt the proof above.  Let $S'$ be the tree obtained from $T$ by deleting all leaves with distance $N$ from the apex and attaching them all to a vertex with distance $N - 1$ from the apex.
We continue as outlined below:
\begin{itemize}
\item Explicitly construct a graceful labeling of $S'$.
\item Assume the branches of $T$ are conveniently ordered, as above.  
\item Obtain $T$ from $S'$ by transfers.
\end{itemize}
We omit the details, since the explicit construction of a graceful labeling of $S'$ is not particularly instructive.
\end{proof}
\noindent In section \ref{radial_rooted_trees_section}, we show how to obtain a graceful labeling of a tree like $S'$ by transfers rather than by explicit construction.  In section \ref{even-caterpillar_section}, we prove that a more general class of banana trees is graceful.

\subsection{Trees with small diameter}

We can characterize all trees with diameter at most four as follows:
\begin{itemize}
\item A tree with diameter 0 is a single vertex.
\item A tree with diameter 1 is a single edge.
\item A tree with diameter 2 is a star.
\item A tree with diameter 3 is a caterpillar with two internal vertices.
\item A tree with diameter 4 is a banana tree.
\end{itemize}
Therefore, all trees with diameter at most four are graceful.  Moreover, by Theorem \ref{caterpillar}, all trees with diameter at most three are 0-rotatable.  Here we show:
\begin{itemize}
\item Most diameter-4 trees are 0-rotatable.
\item All diameter-5 trees are graceful.
\end{itemize}
We begin with the following definition:

\begin{definition}
(van Bussel \cite{zero-centered}) Let $T$ be a tree, and let $f$ be a graceful labeling of $T$.
\begin{itemize}
\item If $T$ has even diameter, then $T$ has a unique central vertex $v$, and $f$ is \emph{0-centered} if $f(v) = 0$.
\item If $T$ has odd diameter, then $T$ has a unique central edge $uv$, and $f$ is \emph{0-centered} if $f(u) = 0$ or $f(v) = 0$.
\end{itemize}
The tree $T$ is \emph{0-centered graceful} if it has a $0$-centered graceful labeling.
\end{definition}


\begin{theorem} \label{diameter4_0-centered}
{\normalfont (van Bussel \cite{zero-centered})} Let $T$ be a tree of diameter 4, with $n$ edges and two branches $T_{1}, T_{2}$ at the central vertex $v$.  Let $m_{1}, m_{2}$ be the number of leaves of $T_{1}, T_{2}$ other than $v$, and assume $m_{1}\ge m_{2}$.  Then $T$ is $0$-centered graceful if and only if there exist integers $x$ and $r$ such that
$$m_{1} = (m_{2} + 2 - x)(r - 1) - x$$
where $x, r$ are not both odd, and
\begin{align*}
2 & \le r \le n/2\\
0 & \le x \le \min(r - 1,\; m_{2})
\end{align*}
\end{theorem}

\begin{proof}
We proceed as follows:
\begin{itemize}
\item Assign the label 0 to the vertex $v$.
\item Assign the labels $n,\; n - r$ to the vertices adjacent to $v$.
\item Show that $r$ determines the assignment of all labels less than $n - r$.
\item Consider the assignment of the remaining labels.
\end{itemize}
We omit the details.  The restraints arise from the last two steps above.
%
%
\end{proof}

\begin{definition}
(van Bussel \cite{zero-centered}) Let $\mathcal{D}$ be the set of trees $T$ such that 
\begin{itemize}
\item $T$ has diameter 4, and $T$ has two branches at the central vertex.
\item $T$ fails the conditions of Theorem \ref{diameter4_0-centered}, so $T$ is not $0$-centered graceful.
\end{itemize}
Let $\mathcal{D}'$ be the set of trees obtained by identifying a leaf of a path with the center of a tree in $\mathcal{D}$, allowing paths of a single vertex, so that $\mathcal{D}\subset\mathcal{D}'$.
\end{definition}

\noindent We present the following two theorems without proof:


\begin{theorem}
{\normalfont (van Bussel \cite{zero-centered})} Let $T$ be a tree of diameter at most 4.  Then $T$ is 0-centered graceful if and only if $T\not\in\mathcal{D}$.
\end{theorem}

\begin{theorem}
{\normalfont (van Bussel \cite{zero-centered})} Let $T$ be a tree of diameter at most 4.  Then $T$ is 0-rotatable if and only if $T\not\in\mathcal{D}'$.
\end{theorem}
\noindent Moreover, van Bussel conjectures, based on empirical evidence:
\begin{itemize}
\item $\mathcal{D}$ contains all trees that are not 0-centered graceful.
\item $\mathcal{D}'$ contains all trees that are not 0-rotatable.
\end{itemize}
%

Now we sketch a proof that all trees of diameter 5 are graceful.


\begin{theorem} \label{diameter_five}
{\normalfont (Hrn\u{c}iar \& Haviar \cite{diameter5})} All trees of diameter 5 are graceful.
\end{theorem}

\begin{proof}
Consider a banana tree $T$ with an odd number of branches at the apex.  By Theorem \ref{generalized_banana}, there exists a graceful labeling of $T$, such that
\begin{itemize}
\item The apex $u$ is labeled 0.
\item The vertices $u_{1}, \ldots, u_{k}$ adjacent to $u$ are labeled $n,\; 1,\; n - 1,\; 2,\ldots$.
\end{itemize}
Let $T'$ be the tree obtained from $T$ by attaching a leaf $v$ at the apex.  Our graceful labeling of $T$ can be extended to $T'$ by assigning the label $n + 1$ to $v$.  Then we can obtain a graceful tree of diameter 5 by transferring any of the following pairs of subtrees from $u$ to $v$:
$$(T_{u, u_{1}}, T_{u, u_{2}}), (T_{u, u_{3}}, T_{u, u_{4}}), \ldots$$
Using this idea, it can be shown that all trees of diameter 5 are graceful.
\end{proof}

\subsection{Lobsters}

\begin{definition}
A \emph{lobster} is a 2-distant tree.
\end{definition}

Since caterpillars have such a simple graceful labeling, many have considered lobsters a natural next step toward the conjecture (see Bermond \cite{survey0}).  However, only some classes of lobsters are known to be graceful.

\begin{definition} 
A \emph{firecracker} is a tree formed from a path $P_{k}$ and $k + 1$ stars by identifying a leaf of each star with a different vertex of the path.
\end{definition}

\begin{theorem}
{\normalfont (Chen, L\"{u} \& Yeh \cite{firecrackers})} All firecrackers are graceful.
\end{theorem}

The following result is usually attributed to Morgan \cite{lobsters-perfect}, but it follows from the earlier work of Broersma \& Hoede \cite{broersma1999another}.

\begin{theorem}
{\normalfont (Broersma \& Hoede \cite{broersma1999another})} All lobsters with perfect matchings are graceful.
\end{theorem}

\begin{proof}
By Corollary \ref{broersma's_redemption}, since the contree of a lobster with a perfect matching is a caterpillar.
\end{proof}

The strongest results on lobsters are those of Mishra \& Panigrahi \cite{mishra2005graceful, mishra2006graceful, mishra2008some, mishra2010some, mishra2011some}, who extend the idea of Theorem \ref{diameter_five} as follows:

\begin{center}
\emph{Graceful lobsters by transfers from banana trees.}
\end{center}

Consider a banana tree $T$ with an odd number of branches at the apex.  By Theorem \ref{generalized_banana}, there exists a graceful labeling of $T$, such that
\begin{itemize}
\item The apex $u$ is labeled 0.
\item The vertices $u_{1}, \ldots, u_{k}$ adjacent to $u$ are labeled $n,\; 1,\; n - 1,\; 2,\ldots$.
\end{itemize}
Let $T'$ be the tree obtained from $T$ by attaching a path at the apex.  Our graceful labeling of $T$ can be extended to $T'$ by assigning the labels
$$n + 1,\; -1,\; n + 2,\; -2, \ldots$$
to the vertices of the path (temporarily allowing negative labels).  Then we can obtain a graceful lobster by 
a sequence of transfers of the second type
$$0\rightarrow n + 1\rightarrow - 1\rightarrow n + 2\rightarrow - 2\rightarrow\cdots$$
Finally, we fix the labels by adding a constant to each label.\newline

\noindent Mishra \& Panigrahi \cite{mishra2005graceful, mishra2006graceful, mishra2008some, mishra2010some, mishra2011some} compile many of the resulting classes of graceful lobsters in tables, which we will not reproduce here.

\subsection{Spiders}

\begin{definition}
(Bahl, Lake \& Wertheim \cite{Bahl}) A \emph{spider} is a tree with at most one vertex of degree greater than two, called the \emph{center}.  A \emph{leg} is a path from the center to one of the leaves.
\end{definition}


In considering spiders with three or four legs, we make use of two earlier results on labelings of paths.  Let $P$ be a path, and let $v$ be a vertex of $P$.  Then
\begin{itemize}
\item $P$ has a graceful labeling $f$ with $f(v) = 0$, by Lemma \ref{paths_0-rotatable}.
\item $P$ has an $\alpha$-labeling $f$ with $f(v) = 0$, by Lemma \ref{zero_alpha}, unless $P$ is $P_{4}$ and $v$ is the central vertex.
\end{itemize}
We use these results to prove the following theorems:




\begin{theorem}
{\normalfont (Huang, Kotzig \& Rosa \cite{rosa3})} Let $T$ be a spider with three legs, of lengths $p, q, r$.  If $(p, q, r)\ne (2, 2, 2)$, then $T$ has an $\alpha$-labeling, and if $(p, q, r) = (2, 2, 2)$, then $T$ has a graceful labeling but not an $\alpha$-labeling.
\end{theorem}

\begin{proof}

If $(p, q, r)\ne (2, 2, 2)$, then we may assume $(p, q)\ne (2, 2)$.  Then $T$ can be formed from paths $P_{p + q}, P_{r}$ by identifying a vertex $u$ of $P_{p + q}$ with a leaf $v$ of $P_{r}$.  Since $(p, q)\ne (2, 2)$, $u$ is not the central vertex of $P_{4}$, so there exist $\alpha$-labelings of $P_{p + q}, P_{r}$ with $u, v$ labeled 0.  Therefore, $T$ has an $\alpha$-labeling by Theorem \ref{combine_alpha_alpha}.

If $(p, q, r) = (2, 2, 2)$, then $T$ is symmetrical, so $T$ has a graceful labeling with center labeled 0 by Theorem \ref{symmetrical_graceful}.  ($T$ does not have an $\alpha$-labeling by Theorem \ref{no_alpha_1} or \ref{no_alpha_2} below.)
\end{proof}



\begin{theorem} \label{four_leg_spider}
{\normalfont (Huang, Kotzig \& Rosa \cite{rosa3})} Let $T$ be a spider with four legs, of lengths $p, q, r, s$.  If at least two of $p, q, r, s$ are not 2, then $T$ has an $\alpha$-labeling.  Otherwise, $T$ has a graceful labeling but may or may not have an $\alpha$-labeling.
\end{theorem}

\begin{proof}

If at least two of $p, q, r, s$ are not 2, then we may assume $(p, q), (r, s)\ne (2, 2)$.  Then $T$ can be formed from paths $P_{p + q}, P_{r + s}$ by identifying a vertex $u$ of $P_{p + q}$ with a vertex $v$ of $P_{r + s}$.  Since $(p, q), (r, s)\ne (2, 2)$, neither of $u, v$ is the central vertex of $P_{4}$, so there exist $\alpha$-labelings of $P_{p + q}, P_{r + s}$ with $u, v$ labeled 0.  Therefore, $T$ has an $\alpha$-labeling by Theorem \ref{combine_alpha_alpha}.

If at most one of $p, q, r, s$ are not 2, then we can obtain $T$ from the spider $S$ with three legs, all of length 2, by attaching a path to the root.  Since $S$ has a graceful labeling with center labeled 0 as noted above, $T$ is graceful by Theorem \ref{attach_caterpillar}.
%
\end{proof}

We now present several other classes of graceful spiders.


\begin{theorem}
{\normalfont (Bahl, Lake \& Wertheim \cite{Bahl})} Let $T$ be a spider, such that the lengths of any two legs of $T$ differ by at most one.  Then $T$ is graceful.
\end{theorem}

\begin{proof}
$T$ is a generalized banana tree, so $T$ is graceful by Theorem \ref{generalized_banana}.
\end{proof}


\begin{definition} 
An \emph{olive tree} is a spider with $n$ legs of lengths $1, \ldots , n$.
\end{definition}

\begin{theorem}
{\normalfont (Pastel \& Raynaud \cite{olive2}, see Abhyankar \& Bhat-Nayak \cite{olive1}, Edwards \& Howard \cite{survey2})} All olive trees are graceful.
\end{theorem}

\begin{proof}
By direct construction.
\end{proof}



\begin{theorem}
{\normalfont (Kanetkar \& Sane \cite{kanetkar2007graceful})} Let $k, d$ be non-negative integers with $3\le k\le 6$.  Then the spider with $k$ legs, such that the $i$th leg has length $1 + (i - 1)d$, is graceful.
\end{theorem}

\begin{proof}
By direct construction.
\end{proof}

\subsection{Trees with few leaves}

We can characterize all trees with at most three leaves as follows: 
\begin{itemize}
\item A tree with at most 2 leaves is a path.
\item A tree with 3 leaves is a spider with three legs.
\end{itemize}
Therefore, all trees with at most three leaves are graceful.  For trees with four leaves, we present the following theorem, and we give an outline of the proof.



\begin{theorem}
{\normalfont (Huang, Kotzig \& Rosa \cite{rosa3})} All trees with exactly four leaves are graceful.
\end{theorem}

\begin{proof}
Let $T$ be a tree with exactly four leaves.  We have two cases:\newline

\noindent\emph{Case 1: $T$ has a vertex $v$ of degree four.}\newline

Then $T$ is a spider with four legs, so $T$ is graceful by Theorem \ref{four_leg_spider}.\newline

\noindent\emph{Case 2: $T$ has two vertices $v$ of degree three.}\newline

Then $T$ consists of a path of length $r$, with paths of lengths $p, q$ attached to one leaf, and paths of lengths $s, t$ attached to the other leaf.  The proof can be outlined as follows:
\begin{itemize}
\item If one of $p, q$ is not 2, form $T$ from a path $P_{p + q}$ and a spider.
\item If one of $s, t$ is not 2, form $T$ from a path $P_{s + t}$ and a spider.
\item If all of $p, q, s, t$ are 2, give an explicit graceful labeling of $T$.
\end{itemize}
\end{proof}

%
%

\noindent Therefore, all trees with at most four leaves are graceful.

\subsection{Trees with few vertices of one color}
Since all trees are bipartite, we can 2-color any tree $T$.  If $T$ has at most four vertices of one color, we have the following result:

\begin{theorem}
{\normalfont (Huang, Kotzig \& Rosa \cite{rosa3})}  Let $T$ be a tree, and let $(A, B)$ be a bipartition of $T$.  If $|A| \le 4$ or $|B| \le 4$, then $T$ is graceful.
\end{theorem}

\begin{proof}
By casework on the structure of $T$; most cases are handled above.
\end{proof}

\subsection{Trees with few vertices}

The following results, proved by computer, have added significant credibility to the Graceful Tree Conjecture:

\begin{theorem}
{\normalfont (Aldred \& McKay \cite{27vertices})} All trees with at most 27 vertices are graceful.
\end{theorem}


\begin{theorem}
{\normalfont (Fang \cite{A_computational_approach_to_the_graceful_tree_conjecture})}  All trees with at most 35 vertices are graceful.
\end{theorem}



\subsection{Trees with no $\alpha$-labeling}

The following result shows that, in a sense, most trees have an $\alpha$-labeling:

\begin{theorem} \label{subdivide_Kotzig}
{\normalfont (Kotzig \cite{depth})} Let $T$ be a tree, and let $e$ be an edge of $T$.  The class of trees obtained by replacing $e$ with a path of arbitrary length contains only finitely many trees without $\alpha$-labelings.
\end{theorem}

In particular, by choosing an edge incident with a leaf, this result shows that any tree is a subgraph of a tree with an $\alpha$-labeling (Gallian \cite{survey4}).  However, there exist infinite classes of trees with no $\alpha$-labeling, as the following result demonstrates:

\begin{theorem}
{\normalfont (Huang, Kotzig \& Rosa \cite{rosa3})} Let $T$ be a tree with diameter three with internal vertices $u, v$ of degree $r + 1$, $s + 1$.  Let $T'$ be the tree obtained from $T$ by subdividing each edge once.  Then
\begin{itemize}
\item $T'$ has an $\alpha$-labeling if and only if $|r - s|\le 1$.
\item $T$ has a graceful labeling.
\end{itemize}
\end{theorem}
\noindent Moreover, this result can be generalized as follows:

\begin{theorem} \label{no_alpha_1}
{\normalfont (Huang, Kotzig \& Rosa \cite{rosa3})} Let $T$ be a tree, such that all vertices of $T$ have odd degree, and such that $|V(T)|\equiv 0\pmod{4}$.  Let $T'$ be the tree obtained from $T$ by subdividing each edge once.  Then $T'$ does not have an $\alpha$-labeling.
\end{theorem}

\noindent Finally, we present one additional class of trees with no $\alpha$-labeling.

\begin{theorem} \label{no_alpha_2}
{\normalfont (Rosa \cite{rosa})} Let $T$ be a diameter-4 tree, such that the base of the base of $T$ has exactly one vertex, and such that $T$ is not a caterpillar.  Then $T$ does not have an $\alpha$-labeling.
\end{theorem}

\subsection{Graceful graphs}


Not all graphs are graceful.  Here we give a necessary condition for a graph to be graceful, as well as some simple classes of graceful graphs.  We give proofs for a few of these cases, to give the reader a sense of the relevant proof techniques.  For a more complete survey of graceful graphs, see Gallian's comprehensive survey \cite{survey4}.

\begin{theorem}
{\normalfont (Rosa \cite{rosa})}  Suppose $G$ is a graceful graph, and suppose each vertex of $G$ has even degree.  Then $n \equiv 0, 3\pmod{4}$.
\end{theorem}

\begin{proof}
Let $f$ be a graceful labeling of $G$.  We calculate the sum of the induced edge labels in two ways.  On one hand, the sum is
$$\sum_{uv\in E(G)}|f(u) - f(v)|\equiv\sum_{uv\in E(G)}(f(u) + f(v))\pmod{2}$$
Since each vertex of $G$ has even degree, each vertex label occurs an even number of times in the sum above, and hence the sum is even.

On the other hand, the sum is
$$1 + 2 + \cdots + n = \frac{n(n + 1)}{2}$$
Therefore, $n(n + 1)/2$ is even, so we have $n\equiv 0, 3\pmod{4}$.
\end{proof}

\begin{theorem}
{\normalfont (Rosa \cite{rosa})} Let $C_{n}$ be the cycle with $n$ vertices.
\begin{enumerate}[(a)]
\item $C_{n}$ is graceful if and only if $n\equiv 0, 3\pmod{4}$.
\item $C_{n}$ has an $\alpha$-labeling if and only if $n\equiv 0\pmod{4}$.
\end{enumerate}
\end{theorem}

\begin{proof}
By the theorem above, the condition in (a) is necessary.  Moreover, since any graph with an $\alpha$-labeling is bipartite, the condition in (b) is also necessary.  It remains to show that $C_{4m}$ has an $\alpha$-labeling and $C_{4m - 1}$ has a graceful labeling, for each $m\ge 1$.


Label $C_{4m}$ by starting with the label $2m$ and then alternating between the least and greatest remaining labels as we proceed around the cycle, skipping $3m$ to avoid repeating the edge label $2m$, so that the labels are
$$2m, \; 0, \; 4m, \; 1, \; 4m - 1, \; \ldots \; , \; 3m + 1, \; m, \; 3m - 1, \; \ldots \; , \; 2m - 1$$
This gives an $\alpha$-labeling of $C_{4m}$ with index $2m - 1$.

Label $C_{4m - 1}$ by starting with the label $2m$ and then alternating between the least and greatest remaining labels as we proceed around the cycle, skipping $m$ to avoid repeating the edge label $2m$, so that the labels are
$$2m, \; 0, \; 4m - 1, \; 1, \; \ldots \; , \; m - 1, \; 3m, \; m + 1, \; \ldots \; , \; 2m + 1$$
This gives an graceful labeling of $C_{n}$, and hence the theorem is proved.
\end{proof}

\begin{theorem}
{\normalfont (Golomb \cite{Golomb72})} The graph $K_{n}$ is graceful if and only if $n\le 4$.
\end{theorem}

\begin{proof}
For $n\le 4$, we have the following graceful labelings:
\begin{itemize}
\item If $n = 1$, use the label $0$.
\item If $n = 2$, use the labels $0, 1$.
\item If $n = 3$, use the labels $0, 1, 3$.
\item If $n = 4$, use the labels $0, 1, 4, 6$.
\end{itemize}
For $n > 4$, suppose $K_{n}$ has a graceful labeling $f$.  We gradually narrow the possibilities for $f$ by considering the edge labels, starting from the largest, as outlined below:
\begin{itemize}
\item Some edge must be labeled $n$, so we must use $0, n$ as vertex labels.
\item Some edge must be labeled $n - 1$, so we must use 1 or $n - 1$ as a vertex label.  Without loss of generality, assume we use 1.
\item Some edge must be labeled $n - 2$, so we must use 2, $n - 2$, or $n - 1$ as a vertex label.  Either 2 or $n - 1$ produce a second edge labeled 1, so we must use $n - 2$ as a vertex label.

\end{itemize}
Continuing systematically in this way, we soon reach a contradiction.
\end{proof}

\begin{theorem}
{\normalfont (Rosa \cite{rosa})} Every complete bipartite graph has an $\alpha$-labeling.
\end{theorem}

\begin{proof}
Label $K_{m, n}$ by assigning the labels $0, 1, \ldots , m - 1$ to the vertices in the part of order $m$, and the labels $m, 2m, \ldots , nm$ to the vertices in the part of order $n$.  This gives an $\alpha$-labeling with index $m - 1$.
\end{proof}

We present several other classes of graceful graphs without proof.

\begin{definition}
A \emph{wheel} is a graph obtained from a cycle by adding a vertex and connecting it to each vertex of the cycle by an edge.
\end{definition}

\begin{theorem}
{\normalfont (Hoede \& Kuiper \cite{hoede1987all})} All wheels are graceful.
\end{theorem}

\begin{definition}
A \emph{gear} is a graph obtained from a wheel by subdividing each edge of the original cycle once.
\end{definition}

\begin{theorem}
{\normalfont (Ma \& Feng \cite{ma1984gracefulness})} All gears are graceful.
\end{theorem}

\begin{theorem}
{\normalfont (Liu \cite{liu1996crowns})} Let $G$ be a graph obtained from a wheel by subdividing each edge of the original cycle $k$ times, for some fixed $k$.  Then $G$ is graceful.
\end{theorem}

\begin{definition}
Let $G_{1}, G_{2}$ be simple graphs.  The \emph{graph product} $G_{1}\times G_{2}$ is the simple graph $H$ with vertex set $V(G_{1})\times V(G_{2})$, such that two vertices $(u_{1}, u_{2}), (v_{1}, v_{2})$ of $H$ are connected by an edge if and only if one of the following holds:
\begin{itemize}
\item $u_{2}$ is adjacent to $v_{2}$ in $G_{2}$, and $u_{1} = v_{1}$.
\item $u_{1}$ is adjacent to $v_{1}$ in $G_{1}$, and $u_{2} = v_{2}$.
\end{itemize}
\end{definition}

\begin{definition}
A \emph{planar grid} is a graph of the form $P_{m}\times P_{n}$.
\end{definition}

\begin{theorem}
{\normalfont (Acharya \& Gill \cite{acharya1981index}, see Maheo \cite{maheo1980strongly})} All planar grids are graceful.
\end{theorem}

\begin{definition}
A \emph{prism} is a graph of the form $C_{m}\times P_{n}$.
\end{definition}

\begin{theorem}
{\normalfont (Frucht \& Gallian \cite{frucht1988labeling})} All prisms are graceful.
\end{theorem}

%
%
%


\section{Preliminary Results} \label{preliminary}

\subsection{Transfers are independent of context}

We now take up the task of extending the method of transfers to prove our main results.  In order to do this, we present a series of lemmas, along with new proofs of each lemma.  Some of these lemmas have been observed and applied without proof in previous papers, but for our purposes it is necessary to more formally prove them, so that we can extend the method of transfers.

In this section, we define the concepts of an \emph{alternating} sequence of vertices and a \emph{transferable} set of leaves, which together provide the context in which transfers can be performed.  Then we prove that in any two such contexts, where the transferable sets of leaves are of the same size, we can perform corresponding transfers.

\begin{definition}
Let $T$ be a gracefully labeled tree.  An \emph{alternating} sequence of vertices is a finite sequence of distinct vertices $v_{1}, \ldots , v_{m}$ of $T$, such that the corresponding sequence of labels is
$$a,\; b - 1,\; a + 1,\; b - 2,\ldots\qquad\text{or}\qquad a,\; b + 1,\; a - 1,\; b + 2,\ldots$$
The \emph{form} of an alternating sequence is its corresponding general sequence above, and the \emph{parameters} of an alternating sequence are the labels $a, b$.
\end{definition}

\begin{definition}
Let $T$ be a gracefully labeled tree, and let $v_{1}, \ldots , v_{m}$ be an alternating sequence of $T$ with parameters $a, b$.  A set of leaves adjacent to $v_{1}$ is \emph{transferable} if the leaves have consecutive labels
$$c,\; c + 1, \ldots ,\; d$$
where $c + d = a + b$.
\end{definition}

\begin{lemma} \label{independent_of_context}
Let $T, T'$ be trees, let $f, f'$ be graceful labelings of $T, T'$, and let $v_{1}, \ldots , v_{m}$, $v_{1}', \ldots , v_{m}'$ be alternating sequences of vertices of $T, T'$, such that $v_{1}, v_{1}'$ are each adjacent to a set of transferable leaves of the same size.  Then for any sequence of transfers $v_{i}\rightarrow v_{j}$ with $i, j$ of different parity, there exists a corresponding sequence of transfers $v_{i}'\rightarrow v_{j}'$, transferring corresponding leaves at each step.
\end{lemma}

\begin{proof}  Suppose the leaves have labels $c,\; c + 1,\;\ldots ,\; d$ and $c',\; c' + 1,\;\ldots ,\; d'$, respectively.  The sequences $v_{1}, \ldots , v_{m}$ and $v_{1}', \ldots , v_{m}'$ are each of the form
$$a,\; b - 1,\; a + 1,\; b - 2,\; \ldots\qquad\text{or}\qquad a,\; b + 1,\; a - 1,\; b + 2,\; \ldots$$

First suppose the two sequences are of the same form.  Since any transfer can be decomposed into transfers of single vertices and pairs of vertices, it suffices to prove the following two claims:\\

\noindent\emph{Claim: If $c + r$, $d - s$ can be transferred from $v_{i}$ to $v_{j}$, then $c' + r$, $d' - s$ can be transferred from $v_{i}'$ to $v_{j}'$.}\\

\noindent\emph{Proof of Claim.}  Suppose $c + r$, $d - s$ can be transferred from $v_{i}$ to $v_{j}$.  Since $i, j$ have different parity, one of $v_{i}, v_{j}$ corresponds to each of $a, b$, so we can write $v_{i} + v_{j} = a + b + f(i, j)$, where $f(i, j)$ depends only on $i, j$ and the form of $v_{1}, \ldots , v_{m}$.  Then
\begin{align*}
(c + r) + (d - s) &= a + b + f(i, j)\\
a + b + r - s &= a + b + f(i, j)\\
r - s &= f(i, j)
\end{align*}
Therefore,
$$(c' + r) + (d' - s) = a' + b' + f(i, j) = v_{i}' + v_{j}'$$
Hence $c' + r$, $d' - s$ can be transferred from $v_{i}'$ to $v_{j}'$, as desired.\\

\noindent\emph{Claim: If $c + r$ can be transferred from $v_{i}$ to $v_{j}$, then $c' + r$ can be transferred from $v_{i}'$ to $v_{j}'$.}\\

\noindent\emph{Proof of Claim.} Define $s$ such that $c + r = d - s$.  Then $c' + r = d' - s$, and the conditions for transfer can be written as
\begin{align*}
(c + r) + (d - s) &= v_{i} + v_{j}\\
(c' + r) + (d' - s) &= v_{i}' + v_{j}'
\end{align*}
These are exactly the conditions for transfer in the claim above, so by the same argument as above, the claim is proved.\\

Now suppose the two sequences are not of the same form.  It suffices to prove the following two claims:\\

\noindent\emph{Claim: If $c + r$, $d - s$ can be transferred from $v_{i}$ to $v_{j}$, then $c' + s$, $d' - r$ can be transferred from $v_{i}'$ to $v_{j}'$.}\\

\noindent\emph{Proof of Claim.}  Suppose $c + r$, $d - s$ can be transferred from $v_{i}$ to $v_{j}$.  Since $i, j$ have different parity, one of $v_{i}, v_{j}$ corresponds to each of $a, b$, so we can write $v_{i} + v_{j} = a + b + f(i, j)$, where $f(i, j)$ depends only on $i, j$ and the form of $v_{1}, \ldots , v_{m}$.  Then
\begin{align*}
(c + r) + (d - s) &= a + b + f(i, j)\\
a + b + r - s &= a + b + f(i, j)\\
r - s &= f(i, j)
\end{align*}
Since $v_{1}, \ldots , v_{m}$ is of the other form, we have $v_{i}' + v_{j}' = a' + b' - f(i, j)$, so
$$(c' + s) + (d' - r) = a' + b' - f(i, j) = v_{i}' + v_{j}'$$
Hence $c' + s$, $d' - r$ can be transferred from $v_{i}'$ to $v_{j}'$, as desired.\\

\noindent\emph{Claim: If $c + r$ can be transferred from $v_{i}$ to $v_{j}$, then $d' - r$ can be transferred from $v_{i}'$ to $v_{j}'$.}\\

\noindent\emph{Proof of Claim.} Define $s$ such that $c + r = d - s$.  Then $c' + s = d' - r$, and the conditions for transfer can be written as
\begin{align*}
(c + r) + (d - s) &= v_{i} + v_{j}\\
(c' + s) + (d' - r) &= v_{i}' + v_{j}'
\end{align*}
These are exactly the conditions for transfer in the claim above, so by the same argument as above, the claim is proved.\\

This completes the proof of the lemma.
\end{proof}

\subsection{Attainable sequences of integers}

We now restrict our attention to certain sequences of transfers of the first type, since, as we will later find, ignoring transfers of the second type does not limit the power of the transfer method.  

\begin{definition}
Let $T$ be a gracefully labeled tree, and let $v_{1}, \ldots , v_{m}$ be an alternating sequence of vertices of $T$, such that $v_{1}$ is adjacent to a transferable set of leaves.  A \emph{well-behaved} sequence of transfers is a finite sequence of transfers $v_{i}\rightarrow v_{j}$, with $i, j$ of different parity, of these leaves, such that
\begin{itemize}
\item All transfers are of the first type.
\item If one transfer is to $v_{i}$, then the next transfer is from $v_{i}$.
\item The set of leaves transferred at each step is a subset of the set of leaves transferred at the previous step. 
\end{itemize}
The \emph{result} of a well-behaved sequence of transfers is a sequence $n_{1}, \ldots , n_{m}$ of non-negative integers, such that after performing the transfers, $v_{k}$ is adjacent to $n_{k}$ of the original transferable leaves, for each $k$.

\end{definition}

\begin{definition}
Let $T$ be a tree, let $f$ be a graceful labeling of $T$, and let $v_{1}, \ldots , v_{m}$ be an alternating sequence of vertices of $T$, such that $v_{1}$ is adjacent to a transferable set of leaves with labels $c,\ldots , d$.  A sequence of non-negative integers $n_{1}, \ldots , n_{m}$ with
$$n_{1} + \cdots + n_{m} = |\{c, \ldots , d\}|$$
is \emph{attainable} with respect to $T, f, v_{1}, \ldots , v_{m}, c, d$ if there exists a well-behaved sequence of transfers $v_{i}\rightarrow v_{j}$ with result $n_{1}, \ldots , n_{m}$.
\end{definition}

The following lemma shows that we can discuss whether a sequence of non-negative integers $n_{1}, \ldots , n_{m}$ is attainable without reference to a particular $T, f, v_{1}, \ldots , v_{m}, c, d$.

\begin{lemma} \label{attainable_lemma}
Let $T, T'$ be trees, let $f, f'$ be graceful labelings of $T, T'$, and let $v_{1}, \ldots , v_{m}$, $v_{1}', \ldots , v_{m}'$ be alternating sequences of vertices of $T, T'$, such that $v_{1}, v_{1}'$ are each adjacent to a set of transferable leaves of the same size.  Then the sequence $n_{1}, \ldots , n_{m}$ is attainable with respect to $T, f, v_{1}, \ldots , v_{m}, c, d$ if and only if it is attainable with respect to $T', f', v_{1}', \ldots , v_{m}', c', d'$.
\end{lemma}

\begin{proof}
Suppose $n_{1}, \ldots , n_{m}$ is attainable with respect to $T, f, v_{1}, \ldots , v_{m}, c, d$, so that there is a well-behaved sequence $\mathcal{S}$ of transfers with result $n_{1}, \ldots , n_{m}$.  Consider the corresponding sequence $\mathcal{S}'$ of transfers given by Lemma \ref{independent_of_context}.  We make the following observations:
\begin{itemize}
\item If $\mathcal{S}$ satisfies the conditions for a well-behaved sequence, $\mathcal{S}'$ does also.
\item The sequences $\mathcal{S}$ and $\mathcal{S}'$ transfer the same number of leaves at each step, so $\mathcal{S}$ and $\mathcal{S}'$ have the same result.
\end{itemize}
Therefore, $\mathcal{S}'$ is also a well-behaved sequence of transfers with result $n_{1}, \ldots , n_{m}$, so the sequence of integers $n_{1}, \ldots , n_{m}$ is also attainable with respect to $T', f', v_{1}', \ldots , v_{m}', c', d'$.  The converse is analogous, so the lemma is proved.
\end{proof}


\subsection{Nicely attainable sequences of integers}

In order to analyze which sequences of integers are attainable, we introduce the concept of a nicely attainable sequence, which will allow us to combine attainable sequences to obtain larger attainable sequences.  We introduce this concept by the following three-step process:

\begin{itemize}
\item Define the concept of a nicely attainable sequence $n_{1}, \ldots , n_{m}, [n_{m + 1}]$.
\item Show that the value of $n_{m + 1}$ is insignificant.
\item Define the concept of a nicely attainable sequence $n_{1}, \ldots , n_{m}$, in terms of the concept of a nicely attainable sequence $n_{1}, \ldots , n_{m}, [n_{m + 1}]$.
\end{itemize}

\begin{definition}
Let $T$ be a tree, let $f$ be a graceful labeling of $T$, and let $v_{1}, \ldots , v_{m + 1}$ be an alternating sequence of vertices of $T$, such that $v_{1}$ is adjacent to a transferable set of leaves with labels $c,\ldots , d$.  A sequence $n_{1}, \ldots , n_{m}, [n_{m + 1}]$ of non-negative integers with
$$n_{1} + \cdots + n_{m + 1} = |\{c, \ldots , d\}|$$
is \emph{nicely attainable} with respect to $T, f, v_{1}, \ldots , v_{m + 1}, c, d$ if there exists a well-behaved sequence of transfers with result $n_{1}, \ldots , n_{m + 1}$, where the last transfer is $v_{m}\rightarrow v_{m + 1}$, and $v_{m + 1}$ occurs in no other transfer.
\end{definition}

The following lemma shows that, as with attainable sequences, we can discuss whether a sequence $n_{1}, \ldots , n_{m}, [n_{m + 1}]$ is nicely attainable without reference to a particular $T, f, v_{1}, \ldots , v_{m + 1}, c, d$.

\begin{lemma}
Let $T, T'$ be trees, let $f, f'$ be graceful labelings of $T, T'$, and let $v_{1}, \ldots , v_{m + 1}$, $v_{1}', \ldots , v_{m + 1}'$ be alternating sequences of vertices of $T, T'$, such that $v_{1}, v_{1}'$ are each adjacent to a set of transferable leaves of the same size.  The sequence of integers $n_{1}, \ldots , n_{m}, [n_{m + 1}]$ is nicely attainable with respect to $T, f, v_{1}, \ldots , v_{m + 1}, c, d$ if and only if it is nicely attainable with respect to $T', f', v_{1}', \ldots, v_{m + 1}', c', d'$.
\end{lemma}

\begin{proof}
Suppose the sequence $n_{1}, \ldots , n_{m}, [n_{m + 1}]$ is nicely attainable with respect to $T, f, v_{1}, \ldots , v_{m + 1}, c, d$, so that there is a well-behaved sequence $\mathcal{S}$ of transfers with result $n_{1}, \ldots , n_{m + 1}$, where the last transfer is $v_{m}\rightarrow v_{m + 1}$, and $v_{m + 1}$ occurs in no other transfer.  Consider the corresponding sequence $\mathcal{S}'$ of transfers given by Lemma \ref{independent_of_context}.  We make the following observations:
\begin{itemize}
\item If $\mathcal{S}$ satisfies the conditions for a well-behaved sequence, $\mathcal{S}'$ does also.
\item The sequences $\mathcal{S}$ and $\mathcal{S}'$ transfer the same number of leaves at each step, so $\mathcal{S}$ and $\mathcal{S}'$ have the same result.
\item By the correspondence with $\mathcal{S}$, the last transfer of $\mathcal{S'}$ is $v_{m}'\rightarrow v_{m + 1}'$, and $v_{m + 1}'$ occurs in no other transfer.
\end{itemize}

Therefore, $\mathcal{S}'$ is also a well-behaved sequence of transfers with result $n_{1}, \ldots , n_{m + 1}$, where the last transfer is $v_{m}'\rightarrow v_{m + 1}'$, and $v_{m + 1}'$ occurs in no other transfer, so $n_{1}, \ldots , n_{m}, [n_{m + 1}]$ is also nicely attainable with respect to $T', f', v_{1}', \ldots , v_{m + 1}', c', d'$.  The converse is analogous, so the lemma is proved.
\end{proof}


\begin{lemma} \label{change_last}
Let $n_{1}, \ldots , n_{m}$ be non-negative integers, and let $n_{m + 1}, n_{m + 1}'$ be positive integers.  Then $n_{1}, \ldots , n_{m}, [n_{m + 1}]$ is nicely attainable if and only if $n_{1}, \ldots , n_{m}, [n_{m + 1}']$ is nicely attainable.
\end{lemma}

\begin{proof}
Let $T, T'$ be trees, let $f, f'$ be graceful labelings of $T, T'$, and let $v_{1}, \ldots , v_{m + 1}$, $v_{1}', \ldots , v_{m + 1}'$ be alternating sequences of vertices of $T, T'$ of the same form with parameters $a, b$ and $a', b'$, respectively, such that 
\begin{itemize}
\item $v_{1}$ is adjacent to a transferable set of $n_{1} + \cdots + n_{m} + n_{m + 1}$ leaves with labels $c, \ldots , d$.
\item $v_{1}'$ is adjacent to a transferable set of $n_{1} + \cdots + n_{m} + n_{m + 1}'$ leaves with labels $c', \ldots , d'$.
\end{itemize}
(These sets of leaves do not generally have the same size.)

Suppose $n_{1}, \ldots , n_{m}, [n_{m + 1}]$ is nicely attainable, so that there exists a well-behaved sequence $\mathcal{S}$ of transfers with result $n_{1}, \ldots , n_{m}, [n_{m + 1}]$, where the last transfer is $v_{m}\rightarrow v_{m + 1}$, and $v_{m + 1}$ occurs in no other transfer.  We claim $n_{1}, \ldots , n_{m}, [n_{m + 1}']$ is also nicely attainable.

Since each transfer in $\mathcal{S}$ is of the first type, the leaves transferred by the $k$th transfer $v_{i_{k}}\rightarrow v_{j_{k}}$ have consecutive labels
$$c + r_{k},\ldots ,\; d - s_{k}$$

We claim there exists a corresponding sequence $\mathcal{S}'$ of transfers in $T'$, such that the leaves transferred by the $k$th transfer $v_{i_{k}}'\rightarrow v_{j_{k}}'$ have corresponding labels
$$c' + r_{k},\ldots ,\; d' - s_{k}$$
To prove this claim, we must prove the following statements:
\begin{enumerate}[(1)]
\item The leaves with the labels above are adjacent to $v_{i_{k}}'$ after performing the first $k - 1$ transfers in $\mathcal{S}'$.
\item The leaves with the labels above satisfy the conditions for transfer.
\end{enumerate}

To prove (1), we use the subset property of a well-behaved sequence of transfers.  Since $\mathcal{S}$ is well-behaved, we have
$$\{c + r_{k},\ldots ,\; d - s_{k}\}\subseteq\{c + r_{k - 1},\ldots ,\; d - s_{k - 1}\}$$
Hence $r_{k}\ge r_{k - 1}$ and $s_{k}\ge s_{k - 1}$, so we have
$$\{c' + r_{k},\ldots ,\; d' - s_{k}\}\subseteq\{c' + r_{k - 1},\ldots ,\; d' - s_{k - 1}\}$$
Therefore, the leaves with labels $c' + r_{k},\ldots ,\; d' - s_{k}$ are adjacent to $v_{j_{k - 1}}'$ after performing the first $k - 1$ transfers in $\mathcal{S}'$.  Since $\mathcal{S}$ is well-behaved, we have $i_{k} = j_{k - 1}$, so these leaves are adjacent to $v_{i_{k}}'$, as desired.

To prove (2), we note that the transfers in $\mathcal{S}$ satisfy the conditions for transfer.  Since $i_{k}, j_{k}$ have different parity, one of $v_{i_{k}}, v_{j_{k}}$ corresponds to each of $a, b$, so we can write $v_{i_{k}} + v_{j_{k}} = a + b + f(i_{k}, j_{k})$, where $f(i_{k}, j_{k})$ depends only on $i_{k}, j_{k}$ and the form of $v_{1}, \ldots , v_{m + 1}$.  Then
\begin{align*}
(c + r_{k}) + (d - s_{k}) &= a + b + f(i_{k}, j_{k})\\
a + b + r_{k} - s_{k} &= a + b + f(i_{k}, j_{k})\\
r_{k} - s_{k} &= f(i_{k}, j_{k})
\end{align*}
Since $v_{1}',\ldots , v_{m}'$ is of the same form, we have $v_{i_{k}}' + v_{j_{k}}' = a' + b' + f(i_{k}, j_{k})$, so
$$(c' + r_{k}) + (d' - s_{k}) = a' + b' + f(i_{k}, j_{k}) = v_{i_{k}}' + v_{j_{k}}'$$
Therefore, the transfers in $\mathcal{S}'$ satisfy the conditions for transfer, as desired.

Therefore, the sequence $\mathcal{S}'$ exists.  We make the following observations:
\begin{itemize}
\item If $\mathcal{S}$ satisfies the conditions for a well-behaved sequence, $\mathcal{S}'$ does also.
\item By the correspondence with $\mathcal{S}$, the last transfer of $\mathcal{S'}$ is $v_{m}'\rightarrow v_{m + 1}'$, and $v_{m + 1}'$ occurs in no other transfer.
\item The sequences $\mathcal{S}$ and $\mathcal{S}'$ leave behind the same number of leaves at each step, so $\mathcal{S}$ and $\mathcal{S}'$ have the same result, except for the last term.
\end{itemize}
Therefore, the result of $\mathcal{S}'$ is $n_{1}, \ldots , n_{m}, n_{m + 1}'$, so $n_{1}, \ldots , n_{m}, [n_{m + 1}']$ is nicely attainable.  The converse is analogous, so the lemma is proved.
\end{proof}


\begin{definition}
A sequence $n_{1}, \ldots , n_{m}$ of non-negative integers is \emph{nicely attainable} if $n_{1}, \ldots , n_{m}, [n_{m + 1}]$ is nicely attainable for all positive integers $n_{m + 1}$.
\end{definition}



\begin{lemma} \label{nicely_attainable_to_attainable}
If $n_{1}, \ldots , n_{m}$ is nicely attainable, then $n_{1}, \ldots , n_{m}$ is attainable.
\end{lemma}

\begin{proof}
Choose a positive integer $n_{m + 1}$, and let $\mathcal{S}$ be a well-behaved sequence of transfers with result $n_{1}, \ldots , n_{m}, n_{m + 1}$, where the last transfer is $v_{m}\rightarrow v_{m + 1}$, and $v_{m + 1}$ occurs in no other transfer.  Consider applying the argument from the proof of Lemma \ref{change_last} to $\mathcal{S}$ with $n_{m + 1}' = 0$.  The last transfer in the corresponding sequence $\mathcal{S}'$ would transfer 0 leaves, so it is not actually a transfer.  However, deleting it gives a well-behaved sequence of transfers with result $n_{1}, \ldots , n_{m}$, so $n_{1}, \ldots , n_{m}$ is attainable.
\end{proof}

The converse of the lemma above is not true, since a well-behaved sequence with result $n_{1}, \ldots , n_{m}$ does not necessarily end at $v_{m}$.

\begin{lemma} \label{combine_attainable}
If $n_{1}, \ldots , n_{m}$ is nicely attainable, and $n_{1}', \ldots , n_{m'}'$ is attainable, then $n_{1}, \ldots , n_{m}, n_{1}', \ldots , n_{m'}'$ is attainable.
\end{lemma}

\begin{proof}
Let $T$ be a tree, let $f$ be a graceful labeling of $T$, and let $v_{1}, \ldots , v_{m + m'}$ be an alternating sequence of vertices of $T$, such that $v_{1}$ is adjacent to a transferable set of $n_{1} + \cdots + n_{m} + n_{1}' + \cdots + n_{m'}'$ leaves.  Since $n_{1}, \ldots , n_{m}$ is nicely attainable, in particular $n_{1}, \ldots , n_{m}, [n_{1}' + \cdots + n_{m'}']$ is nicely attainable.  We claim it is possible to perform the following transfers:
\begin{itemize}
\item Perform the transfers associated with the nicely attainable sequence $n_{1}, \ldots , n_{m}, [n_{1}' + \cdots + n_{m'}']$ on the alternating sequence of vertices $v_{1}, \ldots , v_{m + 1}$.
\item Perform the transfers associated with the attainable sequence $n_{1}', \ldots , n_{m'}'$ on the alternating sequence of vertices $v_{m + 1}, \ldots , v_{m'}$.
\end{itemize}

It suffices to prove that the $n_{1}' + \cdots + n_{m'}'$ leaves adjacent to $v_{m + 1}$ after the first step are a transferable set of leaves with respect to the alternating sequence of vertices $v_{m + 1}, \ldots , v_{m}$.  The leaves adjacent to $v_{m + 1}$ after the first step are all transferred from $v_{m}$ by the last transfer of the first step, so they have consecutive labels $c, \ldots , d$ with $c + d = v_{m} + v_{m + 1}$.

Note that the alternating sequence of vertices $v_{m + 1}, \ldots , v_{m'}$, regardless of its form, has parameters $v_{m + 1}, v_{m}$.  Therefore, the leaves adjacent to $v_{m + 1}$ after the first step are a transferable set of leaves with respect to the alternating sequence of vertices $v_{m + 1}, \ldots , v_{m}$, and hence it is possible to perform the transfers above.

Therefore, $n_{1}, \ldots , n_{m}, n_{1}', \ldots , n_{m'}'$ is attainable.
\end{proof}

\begin{lemma} \label{combine_nicely_attainable}
If $n_{1}, \ldots , n_{m}$ is nicely attainable, and $n_{1}', \ldots , n_{m'}'$ is nicely attainable, then $n_{1}, \ldots , n_{m}, n_{1}', \ldots , n_{m'}'$ is nicely attainable.
\end{lemma}

\begin{proof}
Choose an integer $n_{m' + 1}'$.  Then $n_{1}, \ldots , n_{m}, [n_{1}' + \cdots + n_{m' + 1}']$ and $n_{1}', \ldots , n_{m'}', [n_{m' + 1}']$ are both nicely attainable.  Combining the associated sequences of transfers as in the proof above, we obtain a well-behaved sequence of transfers with result $n_{1}, \ldots , n_{m} , n_{1}' , \ldots , n_{m' + 1}'$, where the last transfer is $v_{m + m'}\rightarrow v_{m + m' + 1}$, and $v_{m + m' + 1}$ occurs in no other transfer.

Therefore, $n_{1}, \ldots , n_{m}, n_{1}', \ldots , n_{m'}', [n_{m' + 1}']$ is nicely attainable for each $n_{m' + 1}'$, so $n_{1}, \ldots , n_{m}, n_{1}', \ldots , n_{m'}'$ is nicely attainable.
\end{proof}

\begin{lemma}
If $n_{1}, \ldots , n_{m}$ is attainable, then $n_{1}, \ldots , n_{m}, 0$ is attainable.
\end{lemma}

\begin{proof}
Use the same well-behaved sequence of transfers.
\end{proof}

\subsection{Some important attainable and nicely attainable sequences}


We now give conditions for when it is possible for a transfer to leave behind a certain number of leaves.  Roughly, given a sequence of transfers $v_{i}\rightarrow v_{j}\rightarrow v_{k}$, the number of leaves left at $v_{j}$ is restricted by the parity and size of $|k - i|/2$.  The following lemma makes this notion precise:

\begin{lemma} \label{transfer_parity}
Let $T$ be a tree, let $f$ be a graceful labeling of $T$, and let $v_{1}, \ldots , v_{m}$ be an alternating sequence of vertices of $T$.  Then
\begin{enumerate}[(1)]
\item If $v_{1}$ is adjacent to a transferable set of leaves, then for each even $k$, and for each $l$ of the same parity as $k/2$ and at least $k/2$, there exists a transfer $v_{1}\rightarrow v_{k}$ of the first type, leaving behind $l$ of the leaves at $v_{1}$.
\item If $v_{j}$ is adjacent to a set of leaves with consecutive labels transferred from $v_{i}$, then for each $k$ of the same parity as $j$, and for each $l$ of the same parity as $|k - i|/2$ and at least $|k - i|/2$, there exists a transfer $v_{j}\rightarrow v_{k}$ of the first type, leaving behind $l$ of the leaves at $v_{j}$.
\end{enumerate}
\end{lemma}


\begin{proof}
Consider (1).  Suppose the sequence $v_{1}, \ldots , v_{m}$ has parameters $a, b$, and suppose the leaves have labels $c, \ldots , d$, so that $c + d = a + b$.  Depending on the form of $v_{1}, \ldots , v_{m}$, we have
\begin{alignat*}{2}
v_{1} + v_{k} &= a + b + k/2\qquad\text{or}\qquad v_{1} + v_{k}\;& = a + b - k/2\\
& = c + d + k/2 & = c + d - k/2
\end{alignat*}
Therefore, the largest set of leaves we can transfer by a transfer $v_{1}\rightarrow v_{k}$ of the first type is
$$\{c + k/2, \ldots , d\}\qquad\text{or}\qquad\{c, \ldots , d - k/2\}$$

This transfer leaves behind $k/2$ of the leaves at $v_{1}$.  By removing the same number of elements from both ends of the set above, we can also leave behind any larger number of leaves of the same parity as $k/2$, and hence (1) is proved.

Consider (2).  Since $v_{1}, \ldots , v_{m}$ is an alternating sequence of vertices, we have $|v_{k} - v_{i}| = |k - i|/2$.  Suppose the leaves have labels $c, \ldots , d$, so that $c + d = v_{i} + v_{j}$.  We have
$$v_{j} + v_{k} = \left\{
     \begin{array}{lr}
       c + d + |k - i|/2 & \text{if } v_{k}\ge v_{i}\\
       c + d - |k - i|/2 & \text{if } v_{k} < v_{i}
     \end{array}
   \right.
$$
Therefore, the largest set of leaves we can transfer by a transfer $v_{j}\rightarrow v_{k}$ of the first type is
\begin{align*}
\{c + |k - i|/2, \ldots , d\} & \qquad\text{if } v_{k}\ge v_{i}\\
\{c, \ldots , d - |k - i|/2\} & \qquad\text{if } v_{k} < v_{i}
\end{align*}

This transfer leaves behind $|k - i|/2$ of the leaves at $v_{j}$.  By removing the same number of elements from both ends of the set above, we can also leave behind any larger number of leaves of the same parity as $|k - i|/2$, and hence (2) is proved.
\end{proof}

The first result in the lemma above, when viewed from the right perspective, is a special case of the second result.  Under the hypothesis of the first result, we can view the leaves adjacent to $v_{1}$ as a set of leaves with consecutive labels transferred from an imaginary vertex $v_{0}$.  Then the first result follows from the second.

Motivated by the role of parity in the lemma above, we use the symbols $o$, $e$, $0$, and $e/0$ to represent certain integers, where
\begin{itemize}
\item $o$ represents a positive odd integer.
\item $e$ represents a positive even integer.
\item $0$ represents the integer $0$.
\item $e/0$ represents a non-negative even integer.
\end{itemize}


\begin{definition}

Let $a_{1}, \ldots , a_{m}$ be a sequence of $o$'s, $e$'s, $0$'s, and $e/0$'s.
\begin{itemize}
\item The sequence $a_{1}, \ldots , a_{m}$ is \emph{attainable} if all sequences of integers $n_{1}, \ldots , n_{m}$ represented by $a_{1}, \ldots , a_{m}$ are attainable.
\item The sequence $a_{1}, \ldots , a_{m}$ is \emph{nicely attainable} if all sequences of integers $n_{1}, \ldots , n_{m}$ represented by $a_{1}, \ldots , a_{m}$ are nicely attainable.
\end{itemize}

\end{definition} 

We now list some important attainable and nicely attainable sequences of $o$'s, $e$'s, $0$'s, and $e/0$'s.  All are proved below.

\begin{center}
\begin{tabular}{l|l}
\emph{Nicely attainable sequences} & \emph{Attainable sequences}\\
\hline
$o$ & $e, \ldots , e$\\
$e, o, o, o, e$ & $e, e/0, e, o$\\
$e, o, e, e, o, e$ & \\
$e, e/0, \underbrace{o, \ldots , o}_{\clap{\parbox{2.5 cm}{\centering\footnotesize non-negative even number of $o$'s}}}, e/0, e$ & $e, e/0, \underbrace{o, \ldots , o}_{\clap{\parbox{2.5 cm}{\centering\footnotesize non-negative even number of $o$'s}}}$
\end{tabular}
\end{center}

Hrn\u{c}iar \& Haviar \cite{diameter5} originally observed that we can perform an arbitrary number of transfers leaving an odd number of leaves at each step, which is equivalent to the fact that $o$ is nicely attainable.  Hrn\u{c}iar \& Haviar \cite{diameter5} also introduced the \emph{backwards double-8 transfer}, a sequence of transfers with result $e, e, e, e$, which shows that $e, e, e, e$ is nicely attainable, as in Lemma \ref{backwards_double_8} below.  Using a different sequence of transfers, we show that the more general sequence $e, e/0, e/0, e$ is nicely attainable, as well as the even more general sequence $e, e/0, o, \ldots , o, e/0, e$ with a non-negative even number of $o$'s.  We also introduce two other new nicely attainable sequences, $e, o, o, o, e$ and $e, o, e, e, o, e$.

Our result that $e, \ldots , e$ is attainable deserves special attention, and suggests a shift in thinking about transfers.  Since Hrn\u{c}iar \& Haviar \cite{diameter5}, transfers of the second type have been used to leave even numbers of leaves at the end of a sequence of transfers.  However, in showing that $e, \ldots , e$ is attainable, we show that a clever sequence of transfers of the first type accomplishes the same task, and even results in the same final configuration of leaves.  As a result, we suggest that transfers of the second type be removed from further work on the conjecture.

\begin{lemma}
$o$ is nicely attainable.
\end{lemma}

\begin{proof}
By Lemma \ref{transfer_parity}, for each odd $l$, there exists a transfer $v_{1}\rightarrow v_{2}$ of the first type, leaving behind $l$ of the leaves at $v_{1}$.  Therefore, $o$ is nicely attainable.
\end{proof}

\begin{lemma} \label{backwards_double_8}
$e, e, e, e$ is nicely attainable.
\end{lemma}

\begin{proof}
Consider the sequence of transfers
$$v_{1}\rightarrow v_{2}\rightarrow v_{3}\rightarrow v_{4}\rightarrow v_{1}\rightarrow v_{2}\rightarrow v_{3}\rightarrow v_{4}\rightarrow v_{5}$$
By Lemma \ref{transfer_parity}, this sequence can leave behind any odd number of leaves at each step.  Since each of $v_{1}, v_{2}, v_{3}, v_{4}$ occurs twice, this sequence can leave behind any positive even number of leaves at each.  Therefore, $e, e, e, e$ is nicely attainable.
\end{proof}

\begin{lemma} \label{double_8_new}
$e, e/0, e/0, e$ is nicely attainable.
\end{lemma}

\begin{proof}
Consider the sequence of transfers
$$v_{1}\rightarrow v_{2}\rightarrow v_{1}\rightarrow v_{4}\rightarrow v_{3}\rightarrow v_{4}\rightarrow v_{5}$$
By Lemma \ref{transfer_parity}, this sequence leaves behind odd numbers of leaves twice at each of $v_{1}, v_{4}$, and non-negative even numbers of leaves once at each of $v_{2}, v_{3}$.  Therefore, $e, e/0, e/0, e$ is nicely attainable.
\end{proof}

\begin{lemma}
Any sequence of the form $e, e/0, o, \ldots , o, e/0, e$, with a non-negative even number of $o's$, is nicely attainable.
\end{lemma}

\begin{proof}
We generalize the proof above.  For the first three sequences, consider the following transfers, where we use brackets to make the pattern evident:
$$v_{1}\rightarrow [v_{2}\rightarrow v_{1}]\rightarrow [v_{4}\rightarrow v_{3}]\rightarrow v_{4}\rightarrow v_{5}$$
$$v_{1}\rightarrow [v_{2}\rightarrow v_{1}]\rightarrow [v_{4}\rightarrow v_{3}]\rightarrow [v_{6}\rightarrow v_{5}]\rightarrow v_{6}\rightarrow v_{7}$$
$$v_{1}\rightarrow [v_{2}\rightarrow v_{1}]\rightarrow [v_{4}\rightarrow v_{3}]\rightarrow [v_{6}\rightarrow v_{5}]\rightarrow [v_{8}\rightarrow v_{7}]\rightarrow v_{8}\rightarrow v_{9}$$
This pattern extends to all sequences of the desired form, and, by Lemma \ref{transfer_parity}, leaves the appropriate numbers of leaves at each vertex.  Therefore, any sequence of the form $e, e/0, o, \ldots , o, e/0, e$, with a non-negative even number of $o's$, is nicely attainable.
\end{proof}

\begin{lemma}
$e, o, o, o, e$ is nicely attainable.
\end{lemma}

\begin{proof}
Consider the sequence of transfers
$$v_{1}\rightarrow v_{4}\rightarrow v_{3}\rightarrow v_{2}\rightarrow v_{5}\rightarrow v_{6}$$
By Lemma \ref{transfer_parity}, this sequence can leave behind any odd number of leaves at each of $v_{2}, v_{3}, v_{4}$, and any positive even number of leaves at each of $v_{1}, v_{5}$.  Therefore, $e, o, o, o, e$ is nicely attainable.
\end{proof}

\begin{lemma}
$e, o, e, e, o, e$ is nicely attainable.
\end{lemma}

\begin{proof}
Consider the sequence of transfers
$$v_{1}\rightarrow v_{4}\rightarrow v_{5}\rightarrow v_{2}\rightarrow v_{3}\rightarrow v_{6}\rightarrow v_{7}$$
By Lemma \ref{transfer_parity}, this sequence can leave behind any odd number of leaves at each of $v_{2}, v_{5}$, and any positive even number of leaves at each of $v_{1}, v_{3}, v_{4}, v_{6}$.  Therefore, $e, o, e, e, o, e$ is nicely attainable.
\end{proof}

\begin{lemma}
$e, \ldots , e$ is attainable.
\end{lemma}

\begin{proof}
Suppose this sequence has $k$ $e$'s, and consider the sequence of transfers
$$v_{1}\rightarrow v_{2}\rightarrow\cdots\rightarrow v_{k - 1}\rightarrow v_{k}\rightarrow v_{k - 1}\rightarrow\cdots\rightarrow v_{2}\rightarrow v_{1}$$
By Lemma \ref{transfer_parity}, this sequence leaves behind odd numbers of leaves twice at each of $v_{2}, \ldots , v_{k - 1}$, and a non-negative even number of leaves once at $v_{k}$.  
At $v_{1}$, this sequence leaves behind an odd number of leaves at the beginning, and any number of leaves at the end.  Therefore, $e, \ldots , e$ is attainable.
\end{proof}

\begin{lemma} \label{special_attainable_sequence}
$e, e/0, e, o$ is attainable.
\end{lemma}

\begin{proof}
Consider the sequence of transfers
$$v_{1}\rightarrow v_{2}\rightarrow v_{1}\rightarrow v_{4}\rightarrow v_{3}$$
By Lemma \ref{transfer_parity}, this sequence leaves behind the appropriate number of leaves at each $v_{i}$, so $e, e/0, e, o$ is attainable.
\end{proof}

\begin{lemma}
Any sequence of the form $e, e/0, o, \ldots , o$, with a non-negative even number of $o$'s, is attainable.
\end{lemma}

\begin{proof}
Suppose this sequence has $2k$ $o$'s, and consider the sequence of transfers
$$v_{1}\rightarrow [v_{2}\rightarrow v_{1}]\rightarrow [v_{4}\rightarrow v_{3}]\rightarrow\cdots\rightarrow [v_{2k + 2}\rightarrow v_{2k + 1}]$$
By Lemma \ref{transfer_parity}, this sequence leaves behind the appropriate number of leaves at each $v_{i}$, so any sequence of the form $e, e/0, o, \ldots , o$, with a non-negative even number of $o$'s, is attainable.
\end{proof}

\subsection{Order of leaves}

Until now, we have worked with alternating sequences of vertices $v_{1}, \ldots , v_{m}$ and transferable sets of leaves $c, \ldots , d$, and we have mostly thought of these two sets of vertices as disjoint.  However, to build trees with larger diameter, we need to allow the alternating sequence $v_{1}, \ldots , v_{m}$ to contain some of the leaves $c, \ldots , d$.  
This opens the possibility of transferring leaves to the leaves we have already transferred, so that we can transfer across multiple levels of a tree (see Hrn\u{c}iar \& Monoszova \cite{banana}, Jesintha \& Sethuraman \cite{jesinthageneration}).

\begin{definition}
An alternating sequence of vertices $v_{1}, \ldots , v_{m}$ \emph{converges} if the labels of $v_{2}, \ldots , v_{m}$ are all between the parameters $a, b$.
\end{definition}

\begin{definition}
Let $v_{1}, \ldots , v_{m}$ be a convergent alternating sequence of vertices.  The \emph{closure} of $v_{1}, \ldots , v_{m}$ is the alternating sequence of maximal length beginning with $v_{1}, \ldots , v_{m}$.  The sequence $v_{1}, \ldots , v_{m}$ is \emph{closed} if its closure is itself.
\end{definition}

For example, suppose $v_{1}, \ldots , v_{5}$ is a sequence of vertices with labels
$$2, 10, 3, 9, 4$$
Then $v_{1}, \ldots , v_{5}$ is a convergent alternating sequence of vertices, and its closure has labels
$$2, 10, 3, 9, 4, 8, 5, 7, 6$$


In general, we are interested in the positions of the leaves $c, \ldots , d$ in the closure of $v_{1}, \ldots , v_{m}$, since their positions determine the order in which we will reach them in the transfer process.

\begin{lemma} \label{leaf_ordering}
Let $v_{1}, \ldots , v_{m}$ be a closed convergent alternating sequence of vertices, such that $v_{1}$ is adjacent to a transferable set of leaves with labels $c, \ldots , d$.  Then
\begin{enumerate}[(1)]
\item The leaves $c, \ldots , d$ are the last last $d - c + 1$ terms of $v_{1}, \ldots , v_{m}$.
\item The leaves $c, \ldots , d$ are in the order
\begin{align*}
d, \; c, \; d - 1, \ldots\qquad\text{if $v_{1}, \ldots , v_{m}$ has form $a,\; b - 1,\; a + 1\ldots$}\\
c, \; d, \; c + 1, \ldots\qquad\text{if $v_{1}, \ldots , v_{m}$ has form $a,\; b + 1,\; a - 1\ldots$}
\end{align*}
\item The first vertex with label $c$ or $d$ has even index.
\end{enumerate}
\end{lemma}

\begin{proof}

We have
\begin{align*}
v_{1} + v_{2} &= a + b \pm 1\\
v_{2} + v_{3} &= a + b\\
v_{3} + v_{4} &= a + b \pm 1\\
v_{4} + v_{5} &= a + b\\
&\;\;\vdots
\end{align*}
Let $v_{i}$ be the first vertex with label $c$ or $d$.  If $i$ is odd, then
$$v_{i - 1} + v_{i} = a + b = c + d$$
Therefore, $v_{i - 1}, v_{i}$ have labels $c, d$, a contradiction.  Therefore, $i$ is even, and hence (3) is proved.

Since $i$ is even, we have
$$v_{i} + v_{i + 1} = a + b = c + d$$
Therefore, $v_{i}, v_{i + 1}$ have labels $c, d$, and hence the remaining terms are the leaves with labels $c, \ldots , d$ in the order $c, \; d, \; c + 1, \ldots$ or $d,\; c,\; d - 1, \ldots$, and hence (1) is proved.

If $v_{1}, \ldots , v_{m}$ has form $a,\; b - 1,\; a + 1\ldots$, then since $v_{1}, \ldots , v_{m}$ converges, we have $a < b$.  Since $v_{1}, \ldots , v_{m}$ alternates between small and large labels, and since $i$ is even, the vertex $v_{i}$ has the large label $d$.  Similarly, if $v_{1}, \ldots , v_{m}$ has form $a,\; b + 1,\; a - 1\ldots$, then $a > b$, and $v_{i}$ has the small label $c$, and hence (2) is proved.
\end{proof}

Therefore, given a convergent alternating sequence of vertices, there is a natural ordering of the leaves with labels $c, \ldots , d$. 
As a result, we can refer to the first $l$ of a set of leaves, as in Corollary \ref{leaf_ordering_odd} below.

Here we prove a simple result that gives the order of leaves for certain well-behaved sequences of transfers.  For the more complicated sequences of transfers, the order of leaves is predictable but difficult to generally describe.

\begin{lemma} \label{jake_the_snake}
Let $v_{1}, \ldots , v_{m}$ be a closed convergent alternating sequence of vertices.  Then any transfer $v_{i}\rightarrow v_{i + 1}$ or $v_{i + 1}\rightarrow v_{i}$ of the first type transfers the leaves $v_{k}, \ldots , v_{m}$ for some $k$.
\end{lemma}

\begin{proof}
If $v_{1}, \ldots , v_{m}$ has form $a,\; b - 1,\; a + 1\ldots$, then we have
$$v_{1} + v_{2} = a + b - 1 = c + d - 1$$
Therefore, to perform a transfer $v_{1}\rightarrow v_{2}$ of the first type, we must transfer a set of leaves of the form $\{c + i,\ldots ,\; d - i - 1\}$ for some $i\ge 0$.  By Lemma \ref{leaf_ordering}, these are the leaves $v_{k}, \ldots , v_{m}$ for some $k$.

Similarly, if $v_{1}, \ldots , v_{m}$ has form $a,\; b + 1,\; a - 1\ldots$, then we have
$$v_{1} + v_{2} = a + b + 1 = c + d + 1$$
Therefore, to perform a transfer $v_{1}\rightarrow v_{2}$ of the first type, we must transfer a set of leaves of the form $\{c + i + 1,\ldots ,\; d - i\}$ for some $i\ge 0$.  Again, by Lemma \ref{leaf_ordering}, these are the leaves $v_{k}, \ldots , v_{m}$ for some $k$, and hence the lemma is proved.
\end{proof}

\begin{corollary} \label{leaf_ordering_odd}
Let $v_{1}, \ldots , v_{m}$ be a closed convergent alternating sequence of vertices, such that $v_{1}$ is adjacent to a transferable set of leaves.  The well-behaved sequence of transfers corresponding to the nicely attainable sequence $l, *$, for $l$ odd, leaves the first $l$ of these leaves at $v_{1}$.
\end{corollary}

\begin{proof}
The sequence of transfers consists of a single transfer $v_{1}\rightarrow v_{2}$.  By the lemma, this transfers the leaves $v_{k}, \ldots , v_{m}$ for some $k$, so it leaves the first $l$ leaves at $v_{1}$.
\end{proof}


\subsection{Radial rooted trees} \label{radial_rooted_trees_section}

\begin{definition}
Let $T$ be a rooted tree.  A \emph{lexicographical order} on the vertices of $T$ is an order by increasing depth, such that for any two vertices $u, v$ with $u$ before $v$, all children of $u$ occur before all children of $v$.

\end{definition}

\begin{definition}
A rooted tree is \emph{radial} if all leaves are at the same level.
\end{definition}

\begin{definition}
A radial rooted tree is \emph{odd} if each internal vertex has an odd number of children.
\end{definition}

\begin{lemma} \label{auxiliary_radial_tree}
Let $T$ be an odd radial rooted tree with root $v$, and let $v_{1}, \ldots , v_{m}$ be the leaves of $T$, ordered according to some lexicographical order $\mathcal{L}$ on the vertices of $T$.  Let $T'$ be a rooted tree obtained from $T$ by attaching a non-negative number of leaves to $v_{1}$.  Then $T'$ has a graceful labeling $f$ with $f(v) = 0$, such that
\begin{itemize}
\item $v_{1}, \ldots , v_{m}$ is an alternating sequence of vertices under $f$.
\item The attached leaves form a transferable set of leaves under $f$.
\end{itemize}
\end{lemma}

\begin{proof}
Let the vertices of $T'$ be $t_{1}, \ldots , t_{n + 1}$, ordered according to $\mathcal{L}$, so that $t_{1}$ is the root $v$, and $t_{k}$ is the vertex $v_{1}$.  Since $T$ is radial, the internal vertices of $T'$ are $t_{1}, \ldots , t_{k}$.  Let $n_{i}$ be the number of children of $t_{i}$ in $T'$ for each $i\le k$, so that all $n_{i}$ are odd.

Consider the gracefully labeled star $K_{1, n}$, with vertices $s_{1}, \ldots , s_{n + 1}$, where $s_{1}$ is the central vertex and is labeled 0, and $s_{1}, \ldots , s_{n + 1}$ is a closed convergent alternating sequence.  Consider the sequence of transfers
$$s_{1}\rightarrow s_{2}\rightarrow \cdots \rightarrow s_{k}$$
By Lemma \ref{transfer_parity}, this sequence can leave behind any odd number of leaves at each step.  Let $S$ be the tree obtained from $K_{1, n}$ by performing this sequence of transfers, leaving behind $n_{i}$ leaves at each step.

By Corollary \ref{leaf_ordering_odd}, each transfer leaves behind the leaves with smallest indices in the transferable set.  Therefore, of the initial transferable set of leaves $\{s_{2}, \ldots , s_{k}\}$, the sequence of transfers leaves the first $n_{1}$ at $s_{1}$, the next $n_{2}$ at $s_{2}$, and so on.  Therefore, if we consider $S$ as a rooted tree with root $s_{1}$, then $S, T'$ have corresponding roots $s_{i}, t_{1}$, and $s_{i}, t_{i}$ have corresponding children for all $i$.  Therefore, $S, T'$ are isomorphic.

Since $S$ is obtained by transfers from a gracefully labeled tree, $S$ has a graceful labeling $g$ with $g(s_{1}) = 0$.  Therefore, $T'$ has a corresponding graceful labeling $f$ with $f(v) = 0$.\newline

\noindent\emph{Claim: $v_{1}, \ldots , v_{m}$ is an alternating sequence of vertices under $f$.}\newline

\noindent\emph{Proof of Claim.} Since $s_{1}, \ldots , s_{n + 1}$ is an alternating sequence under $g$, the corresponding sequence $t_{1}, \ldots , t_{n + 1}$, and hence its subsequence $v_{1}, \ldots , v_{m}$, are alternating sequences under $f$, since $v_{1}, \ldots , v_{m}$ are $m$ consecutive terms of $t_{1}, \ldots , t_{n + 1}$.\newline

\noindent\emph{Claim: The attached leaves form a transferable set of leaves under $f$.}\newline

\noindent\emph{Proof of Claim.} The leaves of $T'$ adjacent to $v_{1}$ correspond to the leaves of $S$ transferred by the last transfer $s_{k - 1}\rightarrow s_{k}$, since $v_{1}$ is $t_{k}$.  Therefore, of these leaves, the smallest and largest labels sum to $s_{k - 1} + s_{k}$.  Since the parameters of $v_{1}, \ldots , v_{m}$ are $t_{k}, t_{k - 1}$, the leaves of $T'$ adjacent to $v_{1}$ form a transferable set with respect to the alternating sequence $v_{1}, \ldots , v_{m}$.\newline

Therefore, $f$ is the desired graceful labeling of $T'$.
\end{proof}

\subsection{Blockwise permutable sequences}
\label{blockwise_section}

Let $T''$ be a tree obtained by attaching leaves to some of the leaves of an odd radial rooted tree $T$.  Our general strategy for proving that $T''$ is graceful is as follows:
\begin{itemize}
\item Choose a lexicographical order $\mathcal{L}$ on the vertices of $T$, so that the leaves of $T$ are $v_{1}, \ldots , v_{m}$, ordered according to $\mathcal{L}$.
\item Get a graceful tree $T'$ by attaching leaves to $v_{1}$, by Lemma \ref{auxiliary_radial_tree}.
\item Obtain $T''$ from $T'$ by transfers $v_{i}\rightarrow v_{j}$.
\end{itemize}
The last step requires that the numbers of children of $v_{1}, \ldots , v_{m}$ form an attainable sequence $n_{1}, \ldots , n_{m}$.  Our key idea is that the choice of $\mathcal{L}$ allows us to apply certain permutations to the sequence $n_{1}, \ldots , n_{m}$.  Therefore, it is useful to consider the following question:

\begin{question}
Given $n_{1}, \ldots , n_{m}$ and a set of permutations $S\subseteq S_{m}$, does there exist a permutation $\sigma\in S$ such that $n_{\sigma(1)}, \ldots , n_{\sigma(m)}$ is attainable?
\end{question}

Here we are considering only trees $T''$ formed from odd radial rooted trees, which ensures that the relevant set of permutations exhibits a certain structure.  We call a sequence of integers that can be permuted by such a set of permutations a \emph{blockwise permutable sequence (BPS)}, defined below.  We reserve the term, \emph{permutable sequence}, for the general question above, which we do not consider in this paper.

%
%
%
%
%
%
%
%
%
%
%
%

\begin{definition}
We define a \emph{blockwise permutable sequence (BPS)}, and the \emph{depth} of a BPS, recursively as follows:
\begin{itemize}
\item Each individual non-negative integer $n$ is a BPS of depth 0.

\item If $\mathcal{B}_{1},\ldots , \mathcal{B}_{k}$ are BPS's of maximum depth $d$, then $(\mathcal{B}_{1}, \ldots , \mathcal{B}_{k})$ is a BPS of depth $d + 1$.
\end{itemize}
The $\mathcal{B}_{i}$ in $(\mathcal{B}_{1}, \ldots , \mathcal{B}_{k})$ are unordered.  For example, we consider $(\mathcal{B}_{1}, \mathcal{B}_{2}, \mathcal{B}_{3})$ and $(\mathcal{B}_{2}, \mathcal{B}_{1}, \mathcal{B}_{3})$ to be the same BPS.
\end{definition}

A BPS can be thought of as a set of sequences of non-negative integers, such that any two sequences are permutations of one another.  This thought motivates the following terminology:

\begin{definition}
We say that a BPS \emph{contains} certain sequences of non-negative integers, according to the following recursive rules:
\begin{itemize}
\item The BPS $n$ contains the sequence $n$.
\item Suppose $\mathcal{B}_{1},\ldots , \mathcal{B}_{k}$ are BPS's, and for each $i$, suppose $\mathcal{B}_{i}$ contains a sequence $B_{i}$.  Then for all permutations $\sigma\in S_{k}$, $(\mathcal{B}_{1}, \ldots , \mathcal{B}_{k})$ contains the concatenated sequence $B_{\sigma(1)}, \ldots , B_{\sigma(k)}$.
\end{itemize}
\end{definition}

For example, $((7, 0, 0), 2, 2)$ is a BPS of depth 2 and contains the following sequences.  Each sequence corresponds to a different representation of the BPS, as shown.

\begin{center}
\begin{tabular}{lr}
$7, 0, 0, 2, 2$\qquad&\qquad$((7, 0, 0), 2, 2)$\\
$0, 7, 0, 2, 2$\qquad&\qquad$((0, 7, 0), 2, 2)$\\
$0, 0, 7, 2, 2$\qquad&\qquad$((0, 0, 7), 2, 2)$\\
$2, 7, 0, 0, 2$\qquad&\qquad$(2, (7, 0, 0), 2)$\\
$2, 0, 7, 0, 2$\qquad&\qquad$(2, (0, 7, 0), 2)$\\
$2, 0, 0, 7, 2$\qquad&\qquad$(2, (0, 0, 7), 2)$\\
$2, 2, 7, 0, 0$\qquad&\qquad$(2, 2, (7, 0, 0))$\\
$2, 2, 0, 7, 0$\qquad&\qquad$(2, 2, (0, 7, 0))$\\
$2, 2, 0, 0, 7$\qquad&\qquad$(2, 2, (0, 0, 7))$\\
\end{tabular}
\end{center}


\begin{definition}
A BPS is \emph{attainable} if it contains an attainable sequence.
\end{definition}

\begin{definition}
A BPS is \emph{nicely attainable} if it contains a nicely attainable sequence.
\end{definition}

\begin{lemma}
If a BPS is nicely attainable, then it is also attainable.
\end{lemma}

\begin{proof}
By Lemma \ref{nicely_attainable_to_attainable}.
\end{proof}


As with ordinary sequences of integers, we also use the symbols $o$, $e$, $0$, and $e/0$ to within a BPS.  More specifically, instead of building a BPS from non-negative integers $n$, we can build a BPS from the symbols $o$, $e$, $0$, and $e/0$, where each symbol represents certain non-negative integers.  Then a BPS using the symbols $o, e, 0, e/0$ represents certain BPS's of non-negative integers, and we will speak of such a BPS as being \emph{attainable} or \emph{nicely attainable} as follows:


\begin{definition}
Let $\mathcal{B}$ be a BPS using the symbols
$o, e, 0, e/0$.
\begin{itemize}
\item $\mathcal{B}$ is \emph{attainable} if every BPS represented by $\mathcal{B}$ is attainable.
\item $\mathcal{B}$ is \emph{nicely attainable} if every BPS represented by $\mathcal{B}$ is nicely attainable.
\end{itemize}
\end{definition}

\begin{lemma} \label{combine_permutable_sequences}
If $(\mathcal{B}_{1}, \ldots , \mathcal{B}_{k})$ contains $n_{1}, \ldots , n_{m}$, and $(\mathcal{B}_{1}', \ldots , \mathcal{B}_{k'}')$ contains $n_{1}', \ldots , n_{m'}'$, then $(\mathcal{B}_{1}, \ldots , \mathcal{B}_{k}, \mathcal{B}_{1}', \ldots , \mathcal{B}_{k'}')$ contains
$$n_{1}, \ldots , n_{m}, n_{1}', \ldots , n_{m'}'$$
\end{lemma}

\begin{proof}
By definition, 
\begin{itemize}
\item Each $\mathcal{B}_{i}$ contains a sequence $B_{i}$, such that, for some permutation $\sigma\in S_{k}$, the concatenated sequence $B_{\sigma(1)}, \ldots , B_{\sigma(k)}$ is $n_{1}, \ldots , n_{m}$.
\item Each $\mathcal{B}_{i}'$ contains a sequence $B_{i}'$, such that, for some permutation $\sigma'\in S_{k'}$, the concatenated sequence $B_{\sigma'(1)}', \ldots , B_{\sigma'(k')}'$ is $n_{1}', \ldots , n_{m'}'$.
\end{itemize}
Then, under the properly chosen permutation, $(\mathcal{B}_{1}, \ldots , \mathcal{B}_{k}, \mathcal{B}_{1}', \ldots , \mathcal{B}_{k'}')$ contains the concatenated sequence $B_{\sigma(1)}, \ldots , B_{\sigma(k)}, B_{\sigma'(1)}', \ldots , B_{\sigma'(k')}'$, which is $n_{1}, \ldots , n_{m}, n_{1}', \ldots , n_{m'}'$, as desired.
\end{proof}
\subsection{Odd blockwise permutable sequences of depth at most 2}

\begin{lemma} \label{BPS_depth_1}
Every BPS of depth 1 is attainable.
\end{lemma}

\begin{proof}
It suffices to prove that all BPS's $(a_{1}, \ldots , a_{k})$ are attainable, where each $a_{i}$ is one of $o, e, 0$.  Each $(a_{1}, \ldots , a_{k})$ contains a sequence of the form 
$$o, \ldots , o, e, \ldots , e, 0, \ldots , 0$$
with a non-negative number of each of $o, e, 0$.  Since $o, *$ is nicely attainable and $e, \ldots , e$ is attainable, 
the sequence $o, \ldots , o, e, \ldots , e, 0, \ldots , 0$ is attainable.  Therefore, 
$(a_{1}, \ldots , a_{k})$ is attainable, as desired.
\end{proof}

\begin{definition}
An \emph{odd} BPS is a type of BPS, defined recursively as follows:
\begin{itemize}
\item Each individual non-negative integer $n$ is an odd BPS of depth 0.
\item If $\mathcal{B}_{1},\ldots , \mathcal{B}_{k}$ are odd BPS's with $k$ odd, $(\mathcal{B}_{1}, \ldots , \mathcal{B}_{k})$ is an odd BPS.
\end{itemize}
\end{definition}


\begin{lemma}
Let $\mathcal{B}_{1}, \ldots , \mathcal{B}_{k}$ be odd BPS's of positive integers of depth at most 1.  Then $(\mathcal{B}_{1}, \ldots , \mathcal{B}_{k})$ is attainable.
\end{lemma}

\begin{proof} 
Denote by $a/b$ the BPS of depth 1
$$(\underbrace{\overbrace{e, \ldots, e}^{\text{$a$ $e$'s}}, o, \ldots , o}_{\text{$b$ terms}})$$
We begin with the following claim. \newline

\noindent\emph{Claim:} If each $\mathcal{B}_{i}$ is one of $0/1$, $1/3$, $2/3$, $3/5$, then $(\mathcal{B}_{1}, \ldots , \mathcal{B}_{k})$ contains a sequence $a_{1}, \ldots , a_{m}, e, \ldots , e$, where $a_{1}, \ldots, a_{m}$ is nicely attainable.\newline


\noindent\emph{Proof of Claim.} Since $o, *$ and $e, e, e, e, *$ are nicely attainable, and since nicely attainable sequences can be combined by Lemma \ref{combine_nicely_attainable}, each of the following BPS's are nicely attainable.\newline
\begin{center}
\begin{tabular}{|p{3.1 cm}|p{8.4 cm}|}
\hline
\multicolumn{2}{|c|}{Nicely attainable BPS's}\\
\hline
$(0/1)$ & $((o))$\\
\hline
$(1/3, 1/3)$ & $((\underbrace{e, o, o), (o, e}, o))$\\
\hline
$(1/3, 3/5)$ & $((o, o, \underbrace{e), (e, e, e}, o, o))$\\
\hline
$(2/3, 2/3)$ & $((o, \underbrace{e, e), (e, e}, o))$\\
\hline
$(2/3, 3/5, 3/5)$ & $((o, \underbrace{e, e), (e, e}, o, o, \underbrace{e), (e, e, e}, o, o))$\\
\hline
$(3/5, 3/5, 3/5, 3/5)$ & $((o, o, \underbrace{e, e, e), (e}, o, o, \underbrace{e, e), (e, e}, o, o, \underbrace{e), (e, e, e}, o, o))$\\
\hline
\end{tabular}\newline
\end{center}
We say that $(\mathcal{B}_{1}, \ldots, \mathcal{B}_{k})$ is \emph{irreducible} if no subset of the $\mathcal{B}_{i}$ corresponds to one of the sequences above.  We list all irreducible BPS's below.

\begin{center}
\begin{tabular}{|p{3.1 cm}|p{8.4 cm}|}
\hline
\multicolumn{2}{|c|}{Irreducible BPS's}\\
\hline
$\varnothing$ & n/a\\
\hline
$(1/3)$ & $((o, o, e))$\\
\hline
$(2/3)$ & $((o, e, e))$\\
\hline
$(3/5)$ & $((o, o, e, e, e))$\\
\hline
$(1/3, 2/3)$ & $((\underbrace{e, o, o), (o, e}, e))$\\
\hline
$(2/3, 3/5)$ & $((o, \underbrace{e, e), (e, e}, o, o, e))$\\
\hline
$(3/5, 3/5)$ & $((o, o, \underbrace{e, e, e), (e}, o, o, e, e))$\\
\hline
$(3/5, 3/5, 3/5)$ & $((o, o, \underbrace{e, e, e), (e}, o, o, \underbrace{e, e), (e, e}, o, o, e))$\\
\hline
\end{tabular}\newline
\end{center}
As shown, each of these BPS's contains a sequence with the desired property, that is, a sequence $a_{1}, \ldots , a_{m}, e, \ldots , e$, where $a_{1}, \ldots , a_{m}$ is nicely attainable.

It remains to prove that each $(\mathcal{B}_{1}, \ldots , \mathcal{B}_{k})$ contains a sequence with the desired property.  We can write $(\mathcal{B}_{1}, \ldots , \mathcal{B}_{k})$ as
\begin{align*}
(\mathcal{B}_{1, 1}, &\ldots, \mathcal{B}_{1, k_{1}},\\
\mathcal{B}_{2, 1}, &\ldots , \mathcal{B}_{2, k_{2}},\\
&\;\;\vdots\\
\mathcal{B}_{n, 1}, &\ldots, \mathcal{B}_{n, k_{n}})
\end{align*}
where
\begin{itemize}
\item $(\mathcal{B}_{i, 1}, \ldots , \mathcal{B}_{i, k_{i}})$ is a nicely attainable BPS, and hence contains a nicely attainable sequence, 
for all $i\ne n$.
\item $(\mathcal{B}_{n, 1}, \ldots , \mathcal{B}_{n, k_{n}})$ is an irreducible BPS, and hence contains a sequence with the desired property, by the above.
\end{itemize}
By Lemma \ref{combine_nicely_attainable}, the concatenation of these sequences also has the desired property, and by Lemma \ref{combine_permutable_sequences}, $(\mathcal{B}_{1}, \ldots , \mathcal{B}_{k})$ contains the concatenation of these sequences.  Therefore, the claim is proved.\newline

\noindent\emph{Claim:} If each $\mathcal{B}_{i}$ is one of $0/1$, $1/3$, $2/3$, $3/5$, $1/1$, $(\mathcal{B}_{1}, \ldots , \mathcal{B}_{k})$ is attainable.\newline

\noindent\emph{Proof of Claim.} We can write $(\mathcal{B}_{1}, \ldots , \mathcal{B}_{k})$ as $(\mathcal{B}_{1}', \ldots , \mathcal{B}_{k'}', 1/1, \ldots , 1/1)$, where each $\mathcal{B}_{i}'$ is one of $0/1$, $1/3$, $2/3$, $3/5$.  Then
\begin{itemize}
\item $(\mathcal{B}_{1}', \ldots , \mathcal{B}_{k'}')$ contains a sequence $a_{1}, \ldots , a_{m}, e, \ldots , e$, where $a_{1}, \ldots , a_{m}$ is nicely attainable, by the previous claim.
\item $(1/1, \ldots, 1/1)$ contains a sequence $e, \ldots , e$.
\end{itemize}
Therefore, by Lemma \ref{combine_permutable_sequences}, $(\mathcal{B}_{1}, \ldots , \mathcal{B}_{k})$ also contains $a_{1}, \ldots , a_{m}, e, \ldots, e$, and hence is attainable.\newline

Now we return to the original problem.  We associate each $a/b$ with one of $0/1$, $1/3$, $2/3$, $3/5$, $1/1$, giving five classes of permutable sequences $a/b$:
\begin{itemize}
\item If $a < b$ and $a\equiv 0\pmod{4}$, associate $a/b$ with $0/1$.
\item If $a < b$ and $a\equiv 1\pmod{4}$, associate $a/b$ with $1/3$.
\item If $a < b$ and $a\equiv 2\pmod{4}$, associate $a/b$ with $2/3$.
\item If $a < b$ and $a\equiv 3\pmod{4}$, associate $a/b$ with $3/5$.
\item If $a = b$, associate $a/b$ with $1/1$.
\end{itemize}
Consider repeating the proofs of the claims above, but where each of $0/1$, $1/3$, $2/3$, $3/5$, $1/1$ represents any $a/b$ in its class.  To show that the proofs still hold, it suffices to prove the following three statements:

\begin{enumerate}[(1)]
\item If $(\mathcal{B}_{1}, \ldots , \mathcal{B}_{k})$ is represented by a nicely attainable BPS in the table above, then $(\mathcal{B}_{1}, \ldots , \mathcal{B}_{k})$ is nicely attainable.
\item If $(\mathcal{B}_{1}, \ldots , \mathcal{B}_{k})$ is represented by an irreducible BPS in the table above, then $(\mathcal{B}_{1}, \ldots , \mathcal{B}_{k})$ contains a sequence $a_{1}, \ldots , a_{m}, e, \ldots , e$, where $a_{1}, \ldots , a_{m}$ is nicely attainable.
\item If $a/b$ is represented by $1/1$, then $a/b$ contains a sequence $e, \ldots , e$.
\end{enumerate}

We observe that (3) is obvious, so we turn our attention to (1) and (2), in which the $\mathcal{B}_{i}$ are each represented by $0/1$, $1/3$, $2/3$, $3/5$.  In each case, we can obtain $\mathcal{B}_{i}$ from its representative BPS by adding $o$'s and groups of four $e$'s.  Consider the following modified versions of the tables above:


\begin{center}
\begin{tabular}{|p{3.1 cm}|p{8.4 cm}|}
\hline
\multicolumn{2}{|c|}{Nicely attainable BPS's}\\
\hline
$(0/1)$ & $((\ldots , o))$\\
\hline
$(1/3, 1/3)$ & $((\ldots, \underbrace{e, o, o), (o, e},\ldots , o))$\\
\hline
$(1/3, 3/5)$ & $((\ldots, o, o, \underbrace{e), (e, e, e}, \ldots, o, o))$\\
\hline
$(2/3, 2/3)$ & $((\ldots , o, \underbrace{e, e), (e, e},\ldots , o))$\\
\hline
$(2/3, 3/5, 3/5)$ & $((\ldots , o, \underbrace{e, e), (e, e},\ldots , o, o, \underbrace{e), (e, e, e}, \ldots, o, o))$\\
\hline
$(3/5, 3/5, 3/5, 3/5)$ &
\hangindent=3.8cm
$((\ldots , o, o, \underbrace{e, e, e), (e}, \ldots , o, o, \underbrace{e, e), (e, e},$
$\ldots , o, o, \underbrace{e), (e, e, e}, \ldots , o, o))$\\
\hline
\end{tabular}\newline
\end{center}

\begin{center}
\begin{tabular}{|p{3.1 cm}|p{8.4 cm}|}
\hline
\multicolumn{2}{|c|}{Irreducible BPS's}\\
\hline
$\varnothing$ & $((\ldots))$\\
\hline
$(1/3)$ & $((\ldots , o, o, e))$\\
\hline
$(2/3)$ & $((\ldots , o, e, e))$\\
\hline
$(3/5)$ & $((\ldots , o, o, e, e, e))$\\
\hline
$(1/3, 2/3)$ & $((\ldots , \underbrace{e, o, o), (o, e},\ldots , e))$\\
\hline
$(2/3, 3/5)$ & $((\ldots , o, \underbrace{e, e), (e, e}, \ldots , o, o, e))$\\
\hline
$(3/5, 3/5)$ & $((\ldots , o, o, \underbrace{e, e, e), (e}, \ldots , o, o, e, e))$\\
\hline
$(3/5, 3/5, 3/5)$ & $((\ldots , o, o, \underbrace{e, e, e), (e}, \ldots , o, o, \underbrace{e, e), (e, e}, \ldots , o, o, e))$\\
\hline
\end{tabular}\newline
\end{center}

For each entry in the tables above, we can add $o$'s and groups of four $e$'s wherever there are dots, since $o$ and $e, e, e, e$ are nicely permutable.  This gives the desired sequence for each $(\mathcal{B}_{1}, \ldots , \mathcal{B}_{k})$.

For example, consider $(5/9, 3/7)$, which is represented by $(1/3, 3/5)$.
\begin{itemize}
\item We can obtain $5/9$ from $1/3$ by adding two $o$'s and a group of four $e$'s.
\item We can obtain $3/7$ from $3/5$ by adding two $o$'s.
\end{itemize}
Therefore, we can obtain a nicely attainable sequence contained by $(5/9, 3/7)$ by adding $o$'s and $e$'s to $(1/3, 3/5)$ as follows:
\begin{center}
\begin{tabular}{p{3.1 cm}p{8.4 cm}}
$(1/3, 3/5)$ & $((\ldots, o, o, \underbrace{e), (e, e, e}, \ldots, o, o))$\\
$(5/9, 3/7)$ & $((o, o, \underbrace{e, e, e, e}, o, o, \underbrace{e), (e, e, e}, o, o, o, o)$
\end{tabular}\newline
\end{center}
Applying this process in general, (1) and (2) hold, and the lemma is proved.

%
%
%
%
%
%
%
\end{proof}

\begin{corollary} \label{odd_BPS_depth_2}
All odd BPS's of positive integers of depth 2 are attainable.
\end{corollary}

\begin{lemma} \label{odd_BPS_depth_2_with_zeros}
Let $\mathcal{B}_{1}, \ldots , \mathcal{B}_{k}$ be odd BPS's of depth at most 1, such that
\begin{itemize}
\item Each $\mathcal{B}_{i}$ has at most one 0.
\item If $\mathcal{B}_{i}$ has a 0, then $\mathcal{B}_{i}$ also has a positive even integer.
\end{itemize}
Then $(\mathcal{B}_{1}, \ldots , \mathcal{B}_{k})$ is attainable.
\end{lemma}

\begin{proof}
We adapt the proof above.  Denote by $a/b$ the BPS of depth 1
\begin{align*}
(\underbrace{e/0, \overbrace{e, \ldots, e}^{\text{$a$ $e$'s}}, o, \ldots , o}_{\text{$b$ terms}})\qquad & \text{if $a > 1$}\\
(\underbrace{e, o, \ldots , o}_{\text{$b$ terms}})\qquad & \text{if $a = 1$}\\
(\underbrace{o, \ldots , o}_{\text{$b$ terms}})\qquad & \text{if $a = 0$}
\end{align*}
Then each $\mathcal{B}_{i}$ can be expressed as $a/b$.  We now associate each $a/b$ with one of $0/1$, $1/3$, $2/3$, $3/3$, $1/1$, giving five classes of permutable sequences $a/b$:
\begin{itemize}
\item If $a\equiv 0\pmod{4}$, associate $a/b$ with $0/1$.
\item If $a\equiv 1\pmod{4}$ and $a < b$, associate $a/b$ with $1/3$.
\item If $a\equiv 1\pmod{4}$ and $a = b$, associate $a/b$ with $1/1$.
\item If $a\equiv 2\pmod{4}$, associate $a/b$ with $2/3$.
\item If $a\equiv 3\pmod{4}$, associate $a/b$ with $3/3$.
\end{itemize}
Now we list nicely attainable BPS's.  The strategy of putting $1/1$'s at the end no longer works in this context, since some of our irreducible BPS's end with $e/0$.  Therefore, we include BPS's with $1/1$ in the list below.

\begin{center}
\begin{tabular}{|p{3.1 cm}|p{8.4 cm}|}
\hline
\multicolumn{2}{|c|}{Nicely attainable BPS's}\\
\hline
$(0/1)$ & $((\ldots , o))$\\
\hline
$(1/3, 1/3)$ & $((\ldots, \underbrace{e, o, o), (o, e},\ldots , o))$\\
\hline
$(1/3, 3/3)$ & $((\ldots, o, o, \underbrace{e), (e/0, e, e}, \ldots))$\\
\hline
$(2/3, 2/3)$ & $((\ldots , o, \underbrace{e, e/0), (e/0, e},\ldots , o))$\\
\hline
$(3/3, 1/1)$ & $((\ldots, \underbrace{e, e/0, e), (e}, \ldots))$\\
\hline
$(1/3, 2/3, 1/1)$& $((\ldots, o, o, \underbrace{e), (\ldots , e), (e/0, e}, \ldots , o))$\\ 
\hline
$(2/3, 3/3, 3/3)$ & $((\ldots , o, \underbrace{e, e/0), (e/0, e}, \ldots , \underbrace{e), (e/0, e, e}, \ldots))$\\
\hline
$(2/3, 1/1, 1/1)$ & $((\ldots , o, \underbrace{e, e/0), (\ldots , e), (\ldots, e}))$\\
\hline
$(3/3, 3/3, 3/3, 3/3)$ &
\hangindent=4.9cm
$((\ldots, \underbrace{e, e/0, e), (e}, \ldots , \underbrace{e, e/0), (e/0, e},$
$\ldots , \underbrace{e), (e/0, e, e}, \ldots))$\\
\hline
$(1/3, 1/1, 1/1, 1/1)$ & $((\ldots , o, o, \underbrace{e), (\ldots , e), (\ldots , e), (\ldots , e}))$\\
\hline
$(1/1, 1/1, 1/1, 1/1)$ & $((\ldots , \underbrace{e), (\ldots , e), (\ldots , e), (\ldots, e}))$\\
\hline
\end{tabular}\newline
\end{center}

\noindent We say that $(\mathcal{B}_{1}, \ldots, \mathcal{B}_{k})$ is \emph{irreducible} if no subset of the $\mathcal{B}_{i}$ corresponds to one of the sequences above.  We list all irreducible BPS's below.

\begin{center}
\begin{tabular}{|p{3.1 cm}|p{8.4 cm}|}
\hline
\multicolumn{2}{|c|}{Irreducible BPS's}\\
\hline
$\varnothing$ & $((\ldots))$\\
\hline
$(1/3)$ & $((\ldots , o, o, e))$\\
\hline
$(2/3)$ & $((\ldots , o, e, e/0))$\\
\hline
$(3/3)$ & $((\ldots , e, e, e/0))$\\
\hline
$(1/1)$ & $((\ldots, e))$\\
\hline
$(1/3, 2/3)$ & $((\ldots , \underbrace{e, o, o), (o, e},\ldots , e/0))$\\
\hline
$(1/3, 1/1)$ & $((\ldots , o, o, e), (\ldots , e))$\\
\hline
$(2/3, 3/3)$ & $((\ldots , o, \underbrace{e, e/0), (e/0, e}, \ldots , e))$\\
\hline
$(2/3, 1/1)^{*}$ & $((e), (e/0, e, o))$\hfill(by Lemma \ref{special_attainable_sequence})\\
\hline
$(3/3, 3/3)$ & $((\ldots , \underbrace{e, e/0, e), (e}, \ldots , e, e/0))$\\
\hline
$(1/1, 1/1)$ & $((\ldots, e), (\ldots , e))$\\
\hline
$(1/3, 1/1, 1/1)$ & $((\ldots , o, o, e), (\ldots , e), (\ldots , e))$\\
\hline
$(3/3, 3/3, 3/3)$ & $((\ldots , \underbrace{e, e/0, e), (e}, \ldots , \underbrace{e, e/0), (e/0, e}, \ldots , e))$\\
\hline
$(1/1, 1/1, 1/1)$ & $((\ldots , e), (\ldots, e), (\ldots , e))$\\
\hline
\end{tabular}\newline
\end{center}
It now suffices to prove the following two statements:
\begin{enumerate}[(1)]
\item If $(\mathcal{B}_{1}, \ldots , \mathcal{B}_{k})$ is represented by a nicely attainable BPS in the table above, then $(\mathcal{B}_{1}, \ldots , \mathcal{B}_{k})$ is nicely attainable.
\item If $(\mathcal{B}_{1}, \ldots , \mathcal{B}_{k})$ is represented by an irreducible BPS in the table above, then $(\mathcal{B}_{1}, \ldots , \mathcal{B}_{k})$ is attainable.
\end{enumerate}

First consider $\mathcal{B}_{i}$ represented by $0/1, 1/3, 2/3, 3/3$.  We can obtain these $\mathcal{B}_{i}$ from their representative BPS's by adding $o$'s and groups of four $e$'s, where at most one $e$ is replaced by an $e/0$.  For these $\mathcal{B}_{i}$, it is clear that we can add $o$'s and groups of four $e$'s wherever there are dots, as in the proof above.

Now consider $\mathcal{B}_{i}$ represented by $1/1$.  We can obtain these $\mathcal{B}_{i}$ from $1/1$ by adding groups of four $e$'s, where at most one $e$ is replaced by an $e/0$.  For these $\mathcal{B}_{i}$, two cases require special consideration:

\begin{center}
\emph{Dots within a nicely attainable string $e, e/0, e/0, e$}
\end{center}

In this case, adding groups of four $e$'s produces a string of $e$'s and $e/0$'s with length a multiple of four.  If an $e$ in any group is replaced by an $e/0$, we claim that we can choose the location of the $e/0$'s so that the resulting string is a nicely attainable string of the form
$$\underbrace{e, e/0, e/0, e}, \underbrace{e, e/0, e/0, e}, \ldots , \underbrace{e, e/0, e/0, e}$$
Since two of every four consecutive terms of the sequence above are $e/0$'s, we can choose the location of the $e/0$'s so that the resulting string is a nicely attainable string.  Therefore, we can add groups of four $e$'s in this case.

\begin{center}
\emph{Dots within a string of trailing $e$'s}
\end{center}

In this case, adding groups of four $e$'s produces a longer string of trailing $e$'s.  If an $e$ in any group is replaced by an $e/0$, we claim that we can choose the location of the $e/0$'s so that the resulting string is a nicely attainable string of the form
$$\underbrace{e, e/0, e/0, e}, \underbrace{e, e/0, e/0, e}, \ldots , \underbrace{e, e/0, e/0, e}, e, \ldots , e$$
with at most three $e$'s after the last brace.  If the four $e$'s are the last four terms, we can replace the last $e$ with an $e/0$.  Otherwise, we can choose the location of the $e/0$'s according to the sequence above.  Therefore, we can add groups of four $e$'s in this case.\newline

Finally, consider $(\mathcal{B}_{1}, \mathcal{B}_{2})$ represented by the irreducible BPS $(2/3, 1/1)$.  Every such $(\mathcal{B}_{1}, \mathcal{B}_{2})$ is either $(2/3, 1/1)$ itself, or can be obtained from one of the following by adding $o$'s and groups of four $e$'s as above:
\begin{center}
\begin{tabular}{|p{3.1 cm}|p{8.4 cm}|}
\hline
\multicolumn{2}{|c|}{BPS's represented by $(2/3, 1/1)$}\\
\hline
$(2/3, 5/5)$ & $((\ldots , o, \underbrace{e, e/0), (e/0, e}, \ldots , e, e, e))$\\
\hline
$(2/5, 1/1)$ & $((\ldots , \underbrace{e), (o, o, o, e}, \ldots , e/0))$\\
\hline
$(5/6, 1/1)$ & $((\ldots , o, \underbrace{e, e/0, e, e}, e), (\ldots , e))$\\
\hline
\end{tabular}
\end{center}
Therefore, (1) and (2) are proved, and hence the lemma is proved.
\end{proof}

\subsection{Odd blockwise permutable sequences of large depth}

\begin{lemma} \label{large_depth_attainable}
All odd BPS's of positive integers, such that the number of even integers is not $3\pmod{4}$, are attainable.
\end{lemma}

\begin{proof}
Our key idea is that each odd BPS $\mathcal{B}$ of positive integers can function in certain basic ways, based on the number of even integers in $\mathcal{B}$ modulo 4.
\begin{itemize}
\item If the number of even integers in $\mathcal{B}$ is $0\pmod{4}$, then $\mathcal{B}$ can function as a long string of $o$'s.
\item If the number of even integers in $\mathcal{B}$ is $1\pmod{4}$, then $\mathcal{B}$ can function as a long string of $o$'s, with one $e$ at the beginning or end.
\item If the number of even integers in $\mathcal{B}$ is $2\pmod{4}$, then $\mathcal{B}$ can function as a long string of $o$'s, with two consecutive $e$'s in the middle.
\item If the number of even integers in $\mathcal{B}$ is $3\pmod{4}$, then $\mathcal{B}$ can function as a long string of $o$'s, with two consecutive $e$'s in the middle and one $e$ at the beginning or end.
\end{itemize}
To formalize this idea, we define the \emph{ending} of a sequence of $e$'s and $o$'s:
\begin{itemize}
\item A sequence has ending $\varnothing$ if it is nicely attainable.
\item A sequence has ending $E1$ if it can be expressed as
$$a_{1}, \ldots , a_{m}, e$$
for some nicely attainable sequence $a_{1}, \ldots , a_{m}$ of $e$'s and $o's$.
\item A sequence has ending $E2$ if it can be expressed as
$$a_{1}, \ldots , a_{m}, e, e, \underbrace{o, \ldots , o}_{\clap{\parbox{2.5 cm}{\centering\footnotesize non-negative even number of $o$'s}}}$$
for some nicely attainable sequence $a_{1}, \ldots , a_{m}$ of $e$'s and $o's$.
\item A sequence has ending $E2'$ if it can be expressed as
$$a_{1}, \ldots , a_{m}, e, e, \underbrace{o, \ldots , o}_{\clap{\parbox{2.5 cm}{\centering\footnotesize odd number of $o$'s}}}$$
for some nicely attainable sequence $a_{1}, \ldots , a_{m}$ of $e$'s and $o's$.
\item A sequence has ending $E3$ if it can be expressed as
$$a_{1}, \ldots , a_{m}, e, e, \underbrace{o, \ldots , o}_{\clap{\parbox{2.5 cm}{\centering\footnotesize non-negative even number of $o$'s}}}, e$$
for some nicely attainable sequence $a_{1}, \ldots , a_{m}$ of $e$'s and $o's$.
\end{itemize}
We associate a BPS with a pair of endings $(\mathcal{E}_{1}, \mathcal{E}_{2})$ if the BPS contains a sequence $a_{1}, \ldots , a_{m}$, such that appending the sequence $a_{1}, \ldots , a_{m}$ onto any sequence with ending $\mathcal{E}_{1}$ gives a sequence with ending $\mathcal{E}_{2}$.  Our original idea then takes the form of the following claim:\newline

\noindent\emph{Claim: Let $\mathcal{B}$ be an odd BPS of positive integers.
\begin{itemize}
\item If the number of even integers in $\mathcal{B}$ is $0\pmod{4}$, then $\mathcal{B}$ is associated with $(\varnothing, \varnothing), (E2, E2'), (E2', E2)$.
\item If the number of even integers in $\mathcal{B}$ is $1\pmod{4}$, then $\mathcal{B}$ is associated with $(\varnothing, E1), (E1, E2), (E2, E3), (E3, \varnothing)$.
\item If the number of even integers in $\mathcal{B}$ is $2\pmod{4}$, then $\mathcal{B}$ is associated with $(\varnothing, E2), (\varnothing, E2'), (E2, \varnothing), (E2', \varnothing)$.
\item If the number of even integers in $\mathcal{B}$ is $3\pmod{4}$, then $\mathcal{B}$ is associated with $(\varnothing, E3), (E1, \varnothing), (E2, E1), (E3, E2)$.
\end{itemize}}

\noindent\emph{Proof of Claim.}  We first verify that each of the depth-0 BPS's of positive integers, $e$ and $o$, satisfy the claim:\newline


\begin{center}
\emph{$o$ is associated with $(\varnothing, \varnothing), (E2, E2'), (E2', E2)$}
\end{center}


\begin{itemize}
\item $o$ is associated with $(E2, E2'), (E2', E2)$ by definition.
\item $o$ is associated with $(\varnothing, \varnothing)$ since $o$ is nicely attainable, and combining two nicely attainable sequences gives another, by Lemma \ref{combine_nicely_attainable}.\newline 
\end{itemize}


\begin{center}
\emph{$e$ is associated with $(\varnothing, E1), (E1, E2), (E2, E3), (E3, \varnothing)$}
\end{center}


\begin{itemize}
\item $e$ is associated with $(\varnothing, E1), (E1, E2), (E2, E3)$ by definition.
\item $e$ is associated with $(E3, \varnothing)$ since
$$e, e, \underbrace{o, \ldots , o}_{\clap{\parbox{2.5 cm}{\centering\footnotesize non-negative even number of $o$'s}}}, e, e$$
is nicely attainable, and combining two nicely attainable sequences gives another, by Lemma \ref{combine_nicely_attainable}.\newline
\end{itemize}

\begin{center}
\emph{If $\mathcal{B}_{1}, \mathcal{B}_{2}, \mathcal{B}_{3}$ each satisfy the claim,\\then $(\mathcal{B}_{1}, \mathcal{B}_{2}, \mathcal{B}_{3})$ satisfies the claim.}
\end{center}

Given $(\mathcal{B}_{1}, \mathcal{B}_{2}, \mathcal{B}_{3})$, if we find a permutation $\sigma\in S_{3}$ and ending pairs
\begin{align*}
(\mathcal{E}_{1}, \mathcal{E}_{2})\text{ of }\mathcal{B}_{\sigma(1)}\\
(\mathcal{E}_{2}, \mathcal{E}_{3})\text{ of }\mathcal{B}_{\sigma(2)}\\
(\mathcal{E}_{3}, \mathcal{E}_{4})\text{ of }\mathcal{B}_{\sigma(3)}
\end{align*}
then $(\mathcal{B}_{1}, \mathcal{B}_{2}, \mathcal{B}_{3})$ is associated with $(\mathcal{E}_{1}, \mathcal{E}_{4})$.  We can show that $(\mathcal{B}_{1}, \mathcal{B}_{2}, \mathcal{B}_{3})$ satisfies the claim by doing this for each ending pair listed in the claim.

For example, suppose the number of $e$'s in each $\mathcal{B}_{i}$ is $1\pmod{4}$, so that each $\mathcal{B}_{i}$ is associated with $(\varnothing, E1), (E1, E2), (E2, E3), (E3, \varnothing)$, and the number of $e$'s in $\mathcal{B}$ is $3\pmod{4}$.  Then we have
\begin{center}
\begin{tabular}{l|l}
Ending pairs of $\mathcal{B}_{i}$ & Ending pairs of $\mathcal{B}$\\
\hline
$(\varnothing, E1), (E1, E2), (E2, E3)$ & $(\varnothing, E3)$\\
$(E1, E2), (E2, E3), (E3, \varnothing)$ & $(E1, \varnothing)$\\
$(E2, E3), (E3, \varnothing), (\varnothing, E1)$ & $(E2, E1)$\\
$(E3, \varnothing), (\varnothing, E1), (E1, E2)$ & $(E3, E2)$
\end{tabular}
\end{center}
With this framework in place, proving that all $(\mathcal{B}_{1}, \mathcal{B}_{2}, \mathcal{B}_{3})$ satisfy the claim is a matter of casework and computation, and we omit the details.  It is helpful to structure the casework around the following observations:
\begin{itemize}
\item If $\mathcal{B}_{1}, \mathcal{B}_{2}$ each have an odd number of $e$'s, then $(\mathcal{B}_{1}, \mathcal{B}_{2})$ is associated with $(\varnothing, \varnothing), (E2, E2)$ or $(\varnothing, E2), (E2, \varnothing)$.
\item If $\mathcal{B}_{1}, \mathcal{B}_{2}$ each have an even number of $e$'s, then $(\mathcal{B}_{1}, \mathcal{B}_{2})$ is associated with $(\varnothing, \varnothing), (E2, E2)$ or $(\varnothing, E2), (E2, \varnothing)$.\newline
\end{itemize}
%


Now we prove the claim.  Since $e$ and $o$ satisfy the claim, we may assume $\mathcal{B}$ has depth at least one, so that $\mathcal{B} = (\mathcal{B}_{1}, \ldots , \mathcal{B}_{k})$.  We apply double induction as follows:
\begin{itemize}
\item By induction on the length (the total number of $e$'s and $o$'s) of $\mathcal{B}$, we may assume that every odd BPS of positive integers with smaller length satisfies the claim.
\item By induction on the depth of $\mathcal{B}$, we may assume that every odd BPS of positive integers with the same length and smaller depth satisfies the claim.
\end{itemize}

If $k = 1$, then $\mathcal{B} = (\mathcal{B}_{1})$, where $\mathcal{B}_{1}$ has the same length as $\mathcal{B}$ and smaller depth.  By induction, $\mathcal{B}_{1}$ satisfies the claim, so $\mathcal{B}$ also satisfies the claim.  Therefore, we may assume $k\ge 3$.

By induction, $\mathcal{B}_{1}, \mathcal{B}_{2}, (\mathcal{B}_{3}, \ldots , \mathcal{B}_{k})$ each satisfy the claim, so by the above, $(\mathcal{B}_{1}, \mathcal{B}_{2}, (\mathcal{B}_{3}, \ldots , \mathcal{B}_{k}))$ also satisfies the claim.  Since $\mathcal{B}$ contains all sequences that $(\mathcal{B}_{1}, \mathcal{B}_{2}, (\mathcal{B}_{3}, \ldots , \mathcal{B}_{k}))$ contains, $\mathcal{B}$ also satisfies the claim, as desired.\newline

Now we return to the original problem.  If the number of even integers in $\mathcal{B}$ is not $3\pmod{4}$, then $\mathcal{B}$ is associated with one of the ending pairs $(\varnothing, \varnothing), (\varnothing, E1), (\varnothing, E2)$.  Since the sequences associated with the endings $E1$ and $E2$ are attainable, $\mathcal{B}$ is attainable, as desired.  (In general, the sequence associated with the ending $E3$ is not attainable.)
\end{proof}
%
%
%
%

\begin{lemma} \label{big_tree_last_lemma}
Let $\mathcal{B}$ be an odd BPS, such that the number of even integers is not $3\pmod{4}$, and such that each 0 is in some $\mathcal{B}_{i}$ of depth 1, such that
\begin{itemize}
\item $\mathcal{B}_{i}$ has no other 0's.
\item $\mathcal{B}_{i}$ also has a positive even integer.
\end{itemize}
Then $\mathcal{B}$ is attainable.
\end{lemma}

\begin{proof}
We adapt the proof above.  First, we adjust the definitions of the endings $E2, E2', E3$ as follows to allow for $e/0$'s:
\begin{itemize}
\item A sequence has ending $E2$ if it can be expressed as
$$a_{1}, \ldots , a_{m}, e, e/0, \underbrace{o, \ldots , o}_{\clap{\parbox{2.5 cm}{\centering\footnotesize non-negative even number of $o$'s}}}$$
for some nicely attainable sequence $a_{1}, \ldots , a_{m}$ of $e$'s and $o's$.
\item A sequence has ending $E2'$ if it can be expressed as
$$a_{1}, \ldots , a_{m}, e, e/0, \underbrace{o, \ldots , o}_{\clap{\parbox{2.5 cm}{\centering\footnotesize odd number of $o$'s}}}$$
for some nicely attainable sequence $a_{1}, \ldots , a_{m}$ of $e$'s and $o's$.
\item A sequence has ending $E3$ if it can be expressed as
$$a_{1}, \ldots , a_{m}, e, e/0, \underbrace{o, \ldots , o}_{\clap{\parbox{2.5 cm}{\centering\footnotesize non-negative even number of $o$'s}}}, e/0$$
for some nicely attainable sequence $a_{1}, \ldots , a_{m}$ of $e$'s and $o's$.
\end{itemize}

We adjust the claim only slightly, so that $\mathcal{B}$ is an odd BPS satisfying the conditions of the lemma, rather than an odd BPS of positive integers.  As in the proof above, $e$ and $o$ satisfy the claim.  We also verify that the BPS's $(e, e, 0)$ and $(e, o, 0)$ satisfy the claim:\newline

\begin{center}
\emph{$(e, e, 0)$ is associated with $(\varnothing, E3), (E1, \varnothing), (E2, E1), (E3, E2)$}
\end{center}

\begin{itemize}
\item $(e, e, 0)$ is associated with $(\varnothing, E3)$, using the sequence $e, 0, e$ or $e, e, 0$.
\item $(e, e, 0)$ is associated with $(E1, \varnothing)$, using the sequence $0, e, e$ or $e, 0, e$.
\item $(e, e, 0)$ is associated with $(E2, E1)$, using the sequence $0, e, e$.
\item $(e, e, 0)$ is associated with $(E3, E2)$, using the sequence $e, e, 0$.
\end{itemize}
Here we use the fact that $e, e/0, e/0, e$ is nicely attainable, and that combining two nicely attainable sequences gives another, by Lemma \ref{transfer_parity}.\newline

\begin{center}
\emph{$(e, o, 0)$ is associated with $(\varnothing, E2), (\varnothing, E2'), (E2, \varnothing), (E2', \varnothing)$}
\end{center}

\begin{itemize}
\item $(e, o, 0)$ is associated with $(\varnothing, E2)$, using the sequence $o, e, 0$.
\item $(e, o, 0)$ is associated with $(\varnothing, E2')$, using the sequence $e, 0, o$.
\item $(e, o, 0)$ is associated with $(E2, \varnothing)$, using the sequence $0, e, o$.
\item $(e, o, 0)$ is associated with $(E2', \varnothing)$, using the sequence $o, 0, e$.
\end{itemize}
Here we use the fact that $e, e/0, e/0, e$ is nicely attainable, and that combining two nicely attainable sequences gives another, by Lemma \ref{transfer_parity}.\newline

As in the proof above, if $\mathcal{B}_{1}, \mathcal{B}_{2}, \mathcal{B}_{3}$ each satisfy the claim, then $(\mathcal{B}_{1}, \mathcal{B}_{2}, \mathcal{B}_{3})$ satisfies the claim.  Since $e$ and $o$ satisfy the claim, and since $0$ does not satisfy the conditions of the lemma, we may assume $\mathcal{B}$ has depth at least one, so that $\mathcal{B} = (\mathcal{B}_{1}, \ldots , \mathcal{B}_{k})$.  We apply double induction as above:
\begin{itemize}
\item By induction on the length (the total number of $e$'s and $o$'s) of $\mathcal{B}$, we may assume that every odd BPS satisfying the conditions of the lemma with smaller length satisfies the claim.
\item By induction on the depth of $\mathcal{B}$, we may assume that every odd BPS satisfying the conditions of the lemma with the same length and smaller depth satisfies the claim.
\end{itemize}
If $k = 1$, then $\mathcal{B} = (\mathcal{B}_{1})$, where $\mathcal{B}_{1}$ has the same length as $\mathcal{B}$ and smaller depth.  By induction, $\mathcal{B}_{1}$ satisfies the claim, so $\mathcal{B}$ also satisfies the claim.  Therefore, we may assume $k\ge 3$.

If $\mathcal{B}$ has depth at least 2, then $\mathcal{B}_{1}, \mathcal{B}_{2}, (\mathcal{B}_{3}, \ldots , \mathcal{B}_{k})$ each satisfy the conditions of the lemma.  Then by induction, $\mathcal{B}_{1}, \mathcal{B}_{2}, (\mathcal{B}_{3}, \ldots , \mathcal{B}_{k})$ each satisfy the claim, so $(\mathcal{B}_{1}, \mathcal{B}_{2}, (\mathcal{B}_{3}, \ldots , \mathcal{B}_{k}))$ also satisfies the claim.  Since $\mathcal{B}$ contains all sequences that $(\mathcal{B}_{1}, \mathcal{B}_{2}, (\mathcal{B}_{3}, \ldots , \mathcal{B}_{k}))$ contains, $\mathcal{B}$ also satisfies the claim.  Therefore, we may assume $\mathcal{B}$ has depth 1.

If $\mathcal{B}$ has no 0's, then $\mathcal{B}_{1}, \mathcal{B}_{2}, (\mathcal{B}_{3}, \ldots , \mathcal{B}_{k})$ each satisfy the conditions of the lemma.  Then, as above, $\mathcal{B}$ satisfies the claim.  Therefore, we may assume $\mathcal{B}$ has a 0.  By the conditions of the lemma, $\mathcal{B}$ has an $e$ and no other 0's.

If $k\ge 5$, then rearrange the $\mathcal{B}_{i}$ so that the 0 and an $e$ are among $\mathcal{B}_{3}, \ldots , \mathcal{B}_{k}$.  Then $\mathcal{B}_{1}, \mathcal{B}_{2}, (\mathcal{B}_{3}, \ldots , \mathcal{B}_{k})$ each satisfy the conditions of the lemma.  Then, as above, $\mathcal{B}$ satisfies the claim.  Therefore, we may assume $k = 3$.  Then $\mathcal{B}$ is $(e, e, 0)$ or $(e, o, 0)$, so $\mathcal{B}$ is attainable, as desired.\newline

Therefore, the modified claim holds.  As in the proof above, this implies that $\mathcal{B}$ is associated with one of the ending pairs $(\varnothing, \varnothing), (\varnothing, E1), (\varnothing, E2)$.  Since the sequences associated with the endings $E1$ and $E2$ are attainable, even under the modified definition of $E2$, $\mathcal{B}$ is attainable, as desired.
\end{proof}


\section{Main Results}

\subsection{A class of graceful diameter-6 trees} 

\begin{theorem} \label{first_big_boy}
Let $T$ be a rooted diameter-6 tree, with central vertex and root $v$, with the following properties:
\begin{itemize}
\item The vertex $v$, and all vertices of distance 1 from $v$, have an odd number of children.
\item All leaves have distance 3 from $v$.
\end{itemize}
Then $T$ has a graceful labeling $f$ with $f(v) = 0$.
\end{theorem}

\begin{proof}
Removing the leaves of $T$ gives an odd radial rooted tree $S$.  Therefore, we can follow the general strategy outlined in section \ref{blockwise_section}:
\begin{itemize}
\item Choose a lexicographical order $\mathcal{L}$ on the vertices of $S$, so that the leaves of $T$ are $v_{1}, \ldots , v_{m}$, ordered according to $\mathcal{L}$.
\item Get a graceful tree $S'$ by attaching leaves to $v_{1}$, by Lemma \ref{auxiliary_radial_tree}.
\item Obtain $T$ from $S'$ by transfers $v_{i}\rightarrow v_{j}$.
\end{itemize}
The last step requires that the numbers of children of $v_{1}, \ldots , v_{m}$ in $T$ form an attainable sequence $n_{1}, \ldots , n_{m}$.  Therefore, it remains to prove that there exists a lexicographical order $\mathcal{L}$, such that the resulting sequence $n_{1}, \ldots , n_{m}$ is attainable.

To prove that $\mathcal{L}$ exists, consider the following correspondence between vertices of $T$ and odd BPS's of positive integers:
\begin{itemize}
\item Each vertex $u$ of distance 2 from $v$ corresponds to an odd BPS $n$ of depth 0, where $n$ is the number of children of $u$.
\item Each vertex $u$ of distance 1 from $v$ corresponds to an odd BPS $(n_{1}, \ldots, n_{k})$ of depth 1, where each $n_{i}$ 
is an odd BPS corresponding to a child of $u$.
\item The vertex $v$ corresponds to an odd BPS $(\mathcal{B}_{1}, \ldots , \mathcal{B}_{k})$ of depth 2, where each $\mathcal{B}_{i}$ is an odd BPS corresponding to a child of $v$.
\end{itemize}
For each vertex $u$ and its corresponding odd BPS $\mathcal{B}$, the sequences $\mathcal{B}$ contains correspond to the possible permutations of the leaves of $T$ that are descendants of $v$ under different lexicographical orders $\mathcal{L}$.

Consider the odd BPS corresponding to $v$, which is attainable by Lemma \ref{odd_BPS_depth_2}.  Therefore, there exists a lexicographical order $\mathcal{L}$ such that the resulting sequence $n_{1}, \ldots , n_{m}$ is attainable.  Therefore, our strategy is valid, and hence $T$ has a graceful labeling $f$ with $f(v) = 0$.
\end{proof}

\begin{theorem} \label{second_big_boy}
Let $T$ be a rooted diameter-6 tree, with central vertex and root $v$, with the following properties:
\begin{itemize}
\item The vertex $v$, and all vertices of distance 1 from $v$, have an odd number of children.
\item All leaves have distance 2 or 3 from $v$, such that
\begin{itemize}
\item No two leaves of distance 2 from $v$ have the same parent.
\item Each leaf of distance 2 from $v$ has a sibling with an even number of children.
\end{itemize}
\end{itemize}
Then $T$ has a graceful labeling $f$ with $f(v) = 0$.
\end{theorem}

\begin{proof}
As above, with the following adjustments:
\begin{itemize}
\item Remove only the leaves of $T$ with distance 3 from $v$ to get $S$.
\item The odd BPS corresponding to $v$ is attainable by Lemma \ref{odd_BPS_depth_2_with_zeros}.
\end{itemize}
\end{proof}




\subsection{Generalizations to trees of larger diameter}


We can extend the methods above to trees of larger diameter, but we need to add an additional condition on the number of vertices in the second-to-last level with an even number of children.

\begin{theorem} \label{third_big_boy}
Let $T$ be a rooted diameter-$2r$ tree, with central vertex and root $v$, with the following properties:
\begin{itemize}
\item The vertex $v$, and all vertices of distance at most $r - 2$ from $v$, have an odd number of children.
\item The number of vertices of distance $r - 1$ from $v$, with an even number of children, is not $3\pmod{4}$.
\item All leaves have distance $r$ from $v$.
\end{itemize}
Then $T$ has a graceful labeling $f$ with $f(v) = 0$.
\end{theorem}

\begin{proof}
The correspondence between odd BPS's and vertices of trees extends naturally to trees of larger diameter.  The odd BPS corresponding to $v$ is attainable by Lemma \ref{large_depth_attainable}.  Therefore, $T$ has a graceful labeling $f$ with $f(v) = 0$.
\end{proof}

\begin{theorem} \label{fourth_big_boy}
Let $T$ be a rooted diameter-$2r$ tree, with central vertex and root $v$, with the following properties:
\begin{itemize}
\item The vertex $v$, and all vertices of distance at most $r - 2$ from $v$, have an odd number of children.
\item The number of vertices of distance $r - 1$ from $v$, with an even number of children, is not $3\pmod{4}$.
\item All leaves have distance $r - 1$ or $r$ from $v$, such that
\begin{itemize}
\item No two leaves of distance $r - 1$ from $v$ have the same parent.
\item Each leaf of distance $r - 1$ from $v$ has a sibling with an even number of children.
\end{itemize}
\end{itemize}
Then $T$ has a graceful labeling $f$ with $f(v) = 0$.
\end{theorem}

\begin{proof}
As above, with the following adjustments:
\begin{itemize}
\item Remove only the leaves of $T$ with distance $r$ from $v$ to get $S$.
\item The odd BPS corresponding to $v$ is attainable by Lemma \ref{big_tree_last_lemma}.
\end{itemize}
\end{proof}

To see why we need the additional condition, consider the diameter-8 tree corresponding to the following BPS:
$$((0/5), (0/5, 0/5, 1/3), (0/5, 0/5, 2/3))$$
A careful consideration shows that the methods of Theorem \ref{first_big_boy} and \ref{second_big_boy} are not sufficient to handle such a tree.


\subsection{A back-and-forth sequence of transfers}

Returning to diameter-6 trees, we now consider trees such that all internal vertices have an odd number of children, which allows us to remove the restrictions on leaves of distance 2 from $v$
To prove that these trees are graceful, we use a back-and-forth sequence of transfers in the first level to manipulate the order of the vertices in the second level.

\begin{theorem} \label{rearrange_those_bad_boys}
Let $T$ be a rooted diameter-6 tree, with central vertex and root $v$, with the following properties:
\begin{itemize}
\item All internal vertices have an odd number of children.
\item All leaves have distance 2 or 3 from $v$.
\end{itemize}
Then $T$ has a graceful labeling $f$ with $f(v) = 0$.
\end{theorem}

\begin{proof}
Each branch of $T$ at $v$ has an odd number of vertices of distance 2 from $v$, and each is an internal vertex or a leaf.  We arrange the branches of $T$ as follows, according to the number of internal vertices and leaves of distance 2 from $v$:
\begin{enumerate}[(1)]
\item positive odd number of internal vertices, no leaves
\item positive odd number of internal vertices, positive even number of leaves
\item positive even number of internal vertices, positive odd number of leaves
\item no internal vertices, positive odd number of leaves
%
\end{enumerate}

Consider the gracefully labeled star $K_{1, n}$, with vertices $v_{0}, \ldots , v_{n}$, where $v_{0}$ is the central vertex and is labeled 0, and $v_{0}, \ldots , v_{n}$ is a closed convergent alternating sequence.  Perform a transfer $v_{0}\rightarrow v_{1}$ of the first type, leaving behind $v_{1}, \ldots , v_{m}$, where $m$ is the number of branches of $T$ at $v$.  We have the following correspondence between $v_{1}, \ldots , v_{m}$ and the branches of $T$ at $v$:
\begin{itemize}
\item $v_{1}, \ldots , v_{i_{1} - 1}$ correspond to branches of type (1).
\item $v_{i_{1}}, \ldots , v_{i_{2} - 1}$ correspond to branches of type (2).
\item $v_{i_{2}}, \ldots , v_{i_{3} - 1}$ correspond to branches of type (3).
\item $v_{i_{3}}, \ldots , v_{m}$ correspond to branches of type (4).
\end{itemize}
Then we perform the sequence of transfers
\begin{center}
\setlength{\tabcolsep}{1.35pt}
\begin{tabular}{lclclcl}
$v_{1}$ & $\rightarrow$ & $v_{2}$ & $\rightarrow\cdots\rightarrow$ & $v_{i_{3} - 2}$ & $\rightarrow$ & $v_{i_{3} - 1}$\\
$v_{i_{3} - 1}$ & $\rightarrow$ & $v_{i_{3} - 2}$ & $\rightarrow\cdots\rightarrow$ & $v_{i_{1} + 1}$ & $\rightarrow$ & $v_{i_{1}}$\\
$v_{i_{1}}$ & $\rightarrow$ & $v_{i_{1} + 1}$ & $\rightarrow\cdots\rightarrow$ & $v_{m - 1}$ & $\rightarrow$ & $v_{m}$
\end{tabular}
\end{center}

By Lemma \ref{transfer_parity}, this sequence of transfers leaves behind an odd number of leaves at each step, except for the turn-arounds at $v_{i_{3} - 1}, v_{i_{1}}$.  We can treat the sequence as leaving behind an odd number of leaves twice at each turn-around, rather than an even number of leaves once.  Then we leave the following numbers of leaves at each vertex:
\begin{center}
\setlength{\tabcolsep}{1.35pt}
\begin{tabular}{rclp{7.9cm}}
$v_{1},$ & $ \ldots ,$ & $v_{i_{1} - 1}$ & \qquad the full number of internal vertices\\
$v_{i_{1}},$ & $ \ldots ,$ & $v_{i_{2} - 1}$ & \qquad the full number of internal vertices\\
$v_{i_{2}},$ & $ \ldots ,$ & $v_{i_{3} - 1}$ & \qquad less than the full number of internal vertices\\
$v_{i_{3} - 1},$ & $ \ldots ,$ & $v_{i_{2}}$ & \qquad the remaining number of internal vertices\\
$v_{i_{2} - 1},$ & $ \ldots ,$ & $v_{i_{1}}$ & \qquad less than the full number of leaves\\
$v_{i_{1}},$ & $ \ldots ,$ & $v_{i_{2} - 1}$ & \qquad the remaining number of leaves\\
$v_{i_{2}},$ & $ \ldots ,$ & $v_{i_{3} - 1}$ & \qquad the full number of leaves\\
$v_{i_{3}},$ & $ \ldots ,$ & $v_{m}$ & \qquad the full number of leaves
\end{tabular}
\end{center}

By Lemma \ref{jake_the_snake}, each transfer leaves behind the first $l$ remaining leaves of the transferable set of leaves, so the order of the vertices in the second level corresponds to the list above.  Then we perform the following sequence of transfers, stopping at the last leaf left at $v_{i_{2}}$ during the second pass:
$$v_{m}\rightarrow v_{m + 1}\rightarrow v_{m + 2}\rightarrow\cdots$$
Then the vertices $v_{m + 1}, v_{m + 2}, \ldots$ correspond to the internal vertices in the second level of $T$, so we can obtain $T$ by this sequence of transfers.  Therefore, $T$ has a graceful labeling $f$ with $f(v) = 0$.
\end{proof}
%
%

\subsection{All even-caterpillar banana trees are graceful} \label{even-caterpillar_section}

Earlier, we introduced three classes of graceful 
trees:
\begin{itemize}
\item Banana trees (Hrn\u{c}iar \& Monoszova \cite{banana}, see Chen, L\"{u} \& Yeh \cite{firecrackers})
\item Generalized banana trees (Hrn\u{c}iar \& Monoszova \cite{banana})
\item Arbitrarily fixed generalized banana trees (Jesintha \& Sethuraman \cite{fixed-banana})
\end{itemize}
Here we introduce a new class of trees, which contains all three classes above.

%

\begin{definition}
Consider a rooted generalized banana tree, with apex as root, and with depth $d$.  Choose a non-negative integer $k < d$, and for each $l$ with $k\le l\le d - 2$, attach a positive even number of leaves to each vertex at level $l$.  The resulting tree is an \emph{even-caterpillar banana tree}.
\end{definition}
\noindent We make a few preliminary observations about this definition:
\begin{itemize}
\item We may add a different even number of leaves to each vertex, unlike in an arbitrarily fixed generalized banana tree.
\item By taking $k = d - 1$, we get a generalized banana tree.  Therefore, the class of even-caterpillar banana trees generalizes the class of generalized banana trees.
\end{itemize}
%
%

%
%

\begin{theorem}
All even-caterpillar banana trees are graceful.
\end{theorem}

\begin{proof}
Let $T$ be an even-caterpillar banana tree with $m$ branches at the apex, and fixed path length $h$.  We have two cases, based on the parity of $m$.\newline

\noindent\emph{Case 1: $m$ is odd.}\newline

We define the following auxiliary trees:
\begin{itemize}
\item Let $S$ be a rooted tree obtained from $m$ paths $P_{h}$ by identifying one leaf of each path all together.
\begin{itemize}
\item Let the identified vertex $v$ be the root of $S$.
\item Let $v_{1}, \ldots , v_{n + 1}$ be the vertices of $S$, ordered according to some lexicographical order $\mathcal{L}$ on the vertices of $S$, so that $v_{n - m + 2}, \ldots , v_{n + 1}$ are the leaves of $S$.
\end{itemize}
\item Let $S'$ be the tree obtained from $S$ by attaching leaves $\{u_{i}\}$ to $v_{n - m + 2}$, such that $S, S'$ have the same numbers of vertices and edges.
\end{itemize}
Since $m$ is odd, $S$ is an odd radial tree, so by Lemma \ref{auxiliary_radial_tree}, $S'$ has a graceful labeling $f$ with $f(v) = 0$, such that
$v_{n - m + 2}, \ldots , v_{n + 1}$ is an alternating sequence of vertices, and
the leaves $\{u_{i}\}$ form a transferable set of leaves.

Recall that $S'$ is obtained by a well-behaved sequence of transfers
$$v_{1}\rightarrow v_{2}\rightarrow\cdots\rightarrow v_{n - m + 1}\rightarrow v_{n - m + 2}$$
Therefore, we can extend this sequence of transfers as follows:
$$v_{n - m + 2}\rightarrow v_{n - m + 1}\rightarrow\cdots\rightarrow v_{i + 1}\rightarrow v_{i}\rightarrow v_{i + 1}\rightarrow\cdots\rightarrow v_{n - m + 2}$$
By Lemma \ref{transfer_parity}, we can leave behind 0 leaves at $v_{n - m + 2}$, an even number of leaves at each of $v_{i}, \ldots , v_{n - m + 1}$, and transfer the remaining leaves back to $v_{n - m + 2}$.  The resulting tree $S''$ has the following properties:
\begin{itemize}
\item By choosing $i$ such that $v_{i}$ is the first vertex in level $k$, this sequence of transfers accounts for all of the attached leaves, so that the resulting tree is identical to $T$ except for the leaves in the last level.
\item $v_{n - m + 2}, \ldots , v_{n + 1}$ is an alternating sequence of vertices, and the leaves at $v_{n - m + 2}$ form a transferable set of leaves with respect to this sequence.
\end{itemize}

We want to obtain $T$ from $S''$ by transfers, which requires that the numbers of children of $v_{n - m + 2}, \ldots , v_{n + 1}$ in $T$ form an attainable sequence $n_{1}, \ldots , n_{m}$.  Since we can carry out the process above for any $\mathcal{L}$, it suffices to prove that the BPS $(n_{1}, \ldots , n_{m})$ is attainable.  By Lemma \ref{BPS_depth_1}, every BPS of depth 1 is attainable, so $(n_{1}, \ldots , n_{m})$ is attainable, and hence $T$ is graceful.\newline

\noindent\emph{Case 2: The apex has even degree.}\newline

Let $T'$ be the even-caterpillar banana tree obtained from $T$ by deleting a branch $H$ at the apex.  Then the apex of $T'$ has odd degree, so $T'$ has a graceful labeling by the case above, with apex labeled 0.  By Lemma \ref{attach_caterpillar}, since $T$ can be formed by attaching the caterpillar $H$ to $T'$ at the apex, $T$ is graceful.
\end{proof}

Using this approach, we can add leaves in the same way to the lower levels of the trees in Theorems \ref{first_big_boy}, \ref{second_big_boy}, \ref{third_big_boy}, \ref{fourth_big_boy}, and \ref{rearrange_those_bad_boys}.


\section{Additional Results}

\subsection{A graceful class of spiders}

\begin{theorem}
Let $T$ be a spider with center $v$, and let the lengths of the legs of $T$ be $m_{1}, \ldots , m_{k}$, where $m_{1} \ge \cdots \ge m_{k}$.  If, for each $i$, we have
$$m_{i}\le\max\left(1,\;\log_{2}\left(\frac{n}{2i - 1}\right) + 1\right)$$
then $T$ has a graceful labeling $f$ with $f(v) = 0$.
\end{theorem}

\begin{proof}  We construct a labeling $f$.  Let $f(v) = 0$, and for each $i$ with $m_{i}\ge 2$, label the $i$th leg, from the vertex adjacent to the center outward, with labels
$$2^{m_{i} - 1}(2i - 1), \; 2^{m_{i} - 2}(2i - 1), \; 2i - 1$$
Assign the remaining labels to the vertices of the legs of length one.

No label is assigned to vertices in two different legs, so $f$ is injective and hence a labeling.  Moreover, for each $i$ with $m_{i}\ge 2$, we have
\begin{align*}
2^{m_{i} - 1}(2i - 1)&\le 2^{\log_{2}\left(\frac{n}{2i - 1}\right)}(2i - 1)\\
&\le \frac{n}{2i - 1}\cdot (2i - 1)\\
&\le n
\end{align*}
Therefore, all labels are at most $n$, so $V_{f} = \{0, \ldots , n\}$.  Finally, note that for each $i$ with $m_{i}\ge 2$, the induced edge labels along the $i$th leg are exactly the vertex labels along the $i$th leg.  The same is true for the legs of length 1, so $E_{f} = \{1, \ldots , n\}$, and hence $f$ is graceful.
\end{proof}

Since this theorem gives a graceful labeling with the center labeled 0, we can attach a path of arbitrary length to a spider $T$ satisfying the conditions of the theorem and obtain another graceful spider.

\subsection{Attaching many leaves to a vertex gives a graceful tree}

Earlier, we cited Kotzig's result that subdividing any edge of a tree sufficiently many times produces a tree with an $\alpha$-labeling (see Theorem \ref{subdivide_Kotzig}).  Here we prove a similar result, that attaching sufficiently many leaves to any fixed vertex of a tree produces a graceful tree.

To prove this result, we begin by making two reductions.  Let $T$ be a tree, let $v$ be a vertex of $T$, and consider the three statements below.  We show that (2) implies (1), and then that (3) implies (2).
\begin{enumerate}[(1)]
\item Attaching many leaves to $T$ at $v$ gives a graceful tree.
\item $T$ has a \emph{consistent range-relaxed graceful labeling} $f$ with $f(v) = 0$.
\item $T$ has a \emph{labeling function} $g$ with $g(v) = 0$.
\end{enumerate}

\subsubsection{Reduction to existence of labeling functions}

First, we must introduce relaxed graceful labelings of trees.  The phrase, ``relaxed graceful labeling," has been used to describe certain relaxations of the conditions for graceful graphs.  Van Bussel \cite{van2002relaxed} gives the following examples in his survey of the topic:
\begin{itemize}
\item \emph{Vertex-relaxed} graceful labelings, which allow repeated vertex labels.
\item \emph{Edge-relaxed} graceful labelings, which allow repeated edge labels.
\item \emph{Range-relaxed} graceful labelings, which allow labels greater than $n$.
\end{itemize}
Usually $\sigma$-labelings and $\rho$-labelings are not considered relaxed graceful labelings, even though they could be viewed as generalizations of graceful labelings.

Here we consider range-relaxed graceful labelings.  Recall that a labeling is an injective mapping of the vertices of a graph to the non-negative integers, so to define a range-relaxed graceful labeling, we need only require that the induced edge labels are distinct.  For simplicity, we restrict our attention to trees, rather than general graphs.

\begin{definition}
(see Van Bussel \cite{van2002relaxed}) Let $T$ be a tree, and let $f$ be a labeling of $T$.  The labeling $f$ is a \emph{range-relaxed graceful} labeling if the induced edge labels of $T$ under $f$ are all distinct.
\end{definition}

We now define a new type of relaxed graceful labeling, by requiring that the sets of vertex labels and edge labels to be the same, except for the vertex labeled 0.

\begin{definition}
Let $T$ be a tree, and let $f$ be a range-relaxed graceful labeling of $T$.  Let $V_{f}$ be the set of vertex labels under $f$, and let $E_{f}$ be the set of induced edge labels under $f$.  The labeling $f$ is a \emph{consistent range-relaxed graceful labeling} if $V_{f} = E_{f}\cup\{0\}$.
\end{definition}

The following lemma gives our motivation for introducing this new type of relaxed graceful labeling.  Specifically, given a tree with a consistent range-relaxed graceful labeling, we can attach leaves to the vertex labeled 0 and obtain a graceful tree.

\begin{lemma} \label{addleaves}
Let $T$ be a tree, and let $v$ be a vertex of $T$.  If $T$ has a consistent range-relaxed graceful labeling $f$ with $f(v) = 0$, then there exists $N$ such that, for all $k\ge N$, attaching $k$ leaves to $T$ at $v$ produces a graceful tree.

\end{lemma}

\begin{proof}
Suppose we attach enough leaves that the number of edges $n + k$ in the resulting tree $T'$ is greater than the maximum value of $f$.  Then we can assign the unused labels in $\{0, 1, \ldots , n + k\}$ to the attached leaves.  The attached edges then receive the same labels, so the resulting labeling is graceful.  Hence the desired $N$ exists.
\end{proof}



Therefore, in order to prove that attaching sufficiently many leaves to any fixed vertex of a tree produces a graceful tree, it suffices to prove that every tree has a consistent range-relaxed graceful labeling.  Before proving this, we make one further reduction, by introducing the concept of a labeling function, a generalization of the concept of a consistent range-relaxed graceful labeling to rational-valued functions.

\begin{definition}
Let $T$ be a tree with vertex set $V$, and let $g:V\rightarrow\mathbb{Q}_{\ge 0}$ be a function.  Let $V_{g}$ be the set of vertex labels under $g$, and let $E_{g}$ be the set of induced edge labels under $g$.  The function $g$ is a \emph{consistent range-relaxed graceful labeling function}, or \emph{labeling function} for short, if $V_{g} = E_{g}\cup\{0\}$.
\end{definition}

\begin{lemma} \label{labelinglemma}
Let $T$ be a tree, and let $v$ be a vertex of $T$.  Then $T$ has a consistent range-relaxed graceful labeling $f$ with $f(v) = 0$ if and only if $T$ has a labeling function $g$ with $g(v) = 0$.
\end{lemma}

\begin{proof}
First suppose $T$ has a labeling function $g$ with $g(v) = 0$.  Since there are only finitely many labels, there exists a positive integer $k$ such that $k\cdot g$ takes only non-negative integer values.  Then $k\cdot g$ is a consistent range-relaxed graceful labeling with $(k\cdot g)(v) = 0$.

Now suppose $T$ has a consistent range-relaxed graceful labeling $f$ with $f(v) = 0$.  Then $f$ itself is a labeling function of $T$ with $f(v) = 0$, which completes the proof.
\end{proof}

\subsubsection{Proof of existence of certain labeling functions}

In light of our two lemmas, it now suffices to consider labeling functions.  We begin by proving the following theorem.

\begin{theorem} \label{labelingtheorem}
Let $T$ be a tree, and let $v$ be a vertex of $T$.  If $v$ is a leaf, then $T$ has a labeling function $g$ with $g(v) = 0$.
\end{theorem}

\begin{proof}
Let $n_{v}$ be the unique neighbor of $v$.  Consider $T$ as a rooted tree with root $v$.  Given a subtree $H$ of $T$, call a vertex $u$ of $H$ \emph{full} if $H$ also contains all neighbors of $u$, that is, if $\deg_{H}u = \deg_{T}u$.  Then let $\mathcal{F}$ be the family of subtrees $H$ of $T$ with the following properties:
\begin{enumerate}[(1)]
\item $H$ contains $v, n_{v}$.
\item For any non-full vertex $u$ of $H$, the degree $\deg_{H}u$ is odd.
\item $H$ has a labeling function $g$ with $g(v) = 0$, such that if $v_{1}, v_{2}$ are vertices of $H$ with $g(v_{1}) = 2^{k}g(v_{2})$ for some positive integer $k$, then $v_{2}$ is a descendent of $v_{1}$, and $v_{1}$ is full.
\end{enumerate}

Consider the subtree of $T$ induced by $v, n_{v}$.  This subtree has a labeling function $g$ given by $g(v) = 0$, $g(n_{v}) = 1$, so $\mathcal{F}$ is nonempty.  Therefore, there exists a maximal tree $H_{\text{max}}\in\mathcal{F}$.

Suppose for contradiction that $H_{\text{max}}\ne T$.  Then there exists a non-full vertex $u$ of $H_{\text{max}}$.  (In particular, this implies $u\ne v$, so $g(u) > 0$.)  By the definition of $\mathcal{F}$, the degree $\deg_{H_{\text{max}}}u$ is odd.  We have two cases:\\

\begin{center}
\noindent\emph{Case 1: $\deg_{H_{\textnormal{max}}}u = \deg_{T}u - 1$.}\\
\end{center}

Let $w$ be the unique child of $u$ not in $H_{\text{max}}$, and let $H_{\text{max}}'$ be the subtree of $T$ induced by $w$ and the vertices of $H_{\text{max}}$.  We claim $H_{\text{max}}'\in\mathcal{F}$.  Clearly $H_{\text{max}}'$ satisfies (1).  For (2), the degrees of all vertices other than $u, w$ are the same in $H_{\text{max}}'$ as in $H_{\text{max}}$, so we need only check $u, w$, but $u$ is full in $H_{\text{max}}'$ and $w$ has degree 1 in $H_{\text{max}}'$.  Hence $H_{\text{max}}'$ also satisfies (2).  

It remains to show that $H_{\text{max}}'$ satisfies (3).  Let $g$ be a labeling function of $H_{\text{max}}$ with $g(v) = 0$ that satisfies the condition in (3), and extend $g$ to a function $g'$ on the vertex set of $H_{\text{max}}'$ by taking $g'(w) = g(u)/2$.  We claim $g'$ is a labeling function of $H_{\text{max}}'$ that satisfies the condition in (3).\\

\noindent\emph{Claim: $g'$ is injective.}\\

\noindent\emph{Proof of Claim.}  Suppose for contradiction that there exist distinct vertices $x, y$ of $H_{\text{max}}'$ with $g'(x) = g'(y)$.  Since $g$ is injective, $x, y$ cannot both be vertices of $H_{\text{max}}$, so we may assume that $x$ is $w$ and $y$ is a vertex of $H_{\text{max}}$.  Then $g'(w) = g(y)$, so $g(u) = 2g(y)$.  Since $g$ satisfies the condition in (3), the vertex $u$ is full in $H_{\text{max}}$, a contradiction.  Therefore, $g'$ is injective.\\

\noindent\emph{Claim: $g'$ is a labeling function of $H_{\text{max}}'$.}\\

\noindent\emph{Proof of Claim.}  We have $V_{g'} = V_{g}\cup\{g'(w)\}$ and $E_{g'} = E_{g}\cup\{|g'(u) - g'(w)|\}$.  Note that
$$|g'(u) - g'(w)| = |g(u) - g(u)/2| = g(u)/2 = g'(w)$$
Then since $V_{g} = E_{g}\cup\{0\}$, we also have $V_{g'} = E_{g'}\cup\{0\}$, so $g'$ is a labeling function of $H_{\text{max}}'$.\\

\noindent\emph{Claim: If $v_{1}, v_{2}$ are vertices of $H_{\text{max}}'$ with $g'(v_{1}) = 2^{k}g'(v_{2})$ for some positive integer $k$, then $v_{2}$ is a descendent of $v_{1}$, and $v_{1}$ is full in $H_{\text{max}}'$.}\\

\noindent\emph{Proof of Claim.}  Suppose $v_{1}, v_{2}$ are both vertices of $H_{\text{max}}$, so that $g(v_{1}) = 2^{k}g(v_{2})$ for some positive integer $k$.  Since $g$ satisfies the conditions in (3), $v_{2}$ is a descendent of $v_{1}$, and $v_{1}$ is full in $H_{\text{max}}$.  Then $v_{1}$ is also full in $H_{\text{max}}'$.  Therefore, $g'$ satisfies the condition in (3) for such $v_{1}, v_{2}$, so we may assume that one of $v_{1}, v_{2}$ is $w$.

Suppose for contradiction that $v_{1}$ is $w$.  Then $g'(w) = 2^{k}g(v_{2})$, so $g(u) = 2^{k + 1}g(v_{2})$.  Since $g$ satisfies the condition in (3), $u$ is full in $H_{\text{max}}$, a contradiction.  

Therefore, $v_{2}$ is $w$, so $g(v_{1}) = 2^{k}g'(w)$, and we have
$$g(v_{1}) = 2^{k - 1}g(u)$$
If $k > 1$, then since $g$ satisfies the condition in (3), $u$ is a descendent of $v_{1}$, and $v_{1}$ is full in $H_{\text{max}}$.  Then $v_{1}$ is also full in $H_{\text{max}}'$.  Therefore, $g'$ satisfies the condition in (3) for such $v_{1}, v_{2}$, so we may assume $k = 1$.

Then $g(v_{1}) = g(u)$, so $v_{1}$ is $u$, since $g$ is injective.  Then $v_{1}, v_{2}$ are $u, w$, and $w$ is a descendant of $u$, and $u$ is full in $H_{\text{max}}'$.  Therefore, $g'$ satisfies the condition in (3) for all $v_{1}, v_{2}$, as desired.\\

Therefore, $H_{\text{max}}'$ satisfies (1), (2), and (3), and hence $H_{\text{max}}'\in\mathcal{F}$, a contradiction to the maximality of $H_{\text{max}}$.  Therefore, in this case, we can conclude that $H_{\text{max}} = T$.\\

\begin{center}
\noindent\emph{Case 2: $\deg_{H_{\textnormal{max}}}u \le \deg_{T}u - 2$.}\\
\end{center}

Let $w_{1}, w_{2}$ be two children of $u$ not in $H_{\text{max}}$, and let $H_{\text{max}}'$ be the subtree of $T$ induced by $w_{1}, w_{2}$ and the vertices of $H_{\text{max}}$.  We claim $H_{\text{max}}'\in\mathcal{F}$.  Clearly $H_{\text{max}}'$ satisfies (1).  For (2), the degrees of all vertices other than $u, w_{1}, w_{2}$ are the same in $H_{\text{max}}'$ as in $H_{\text{max}}$, so we need only check $u, w_{1}, w_{2}$, but $u$ is full in $H_{\text{max}}'$ and $w_{1}, w_{2}$ have degree 1 in $H_{\text{max}}'$.  Hence $H_{\text{max}}'$ also satisfies (2).

It remains to show that $H_{\text{max}}'$ satisfies (3).  Let $g$ be a labeling function of $H_{\text{max}}$ with $g(v) = 0$ that satisfies the condition in (3), and extend $g$ to a function $g'$ on the vertex set of $H_{\text{max}}'$ by taking, for some rational $q$ with $0 < q < g(u)$,
\begin{align*}
g'(w_{1}) &= q\\
g'(w_{2}) &= g(u) - q
\end{align*}

Consider an equivalence relation on the positive rationals, such that $q_{1}, q_{2}$ are equivalent if and only if $q_{1}/q_{2} = 2^{k}$ for some integer $k$.  We choose $q$ according to the following claim.\\

\noindent\emph{Claim: There exists $q$ such that $g'(w_{1}), g'(w_{2})$ are not equivalent to each other or to any other nonzero label of $H_{\text{max}}'$ under $g'$.}\\

\noindent\emph{Proof of Claim.}  Take a small interval $(a, b)$ centered at $f(u)/2$ with $b < 2a$, so that no two rationals in $(a, b)$ are equivalent.  Then each label of $H_{\text{max}}$ under $g$ is equivalent to at most one rational in $(a, b)$.
\begin{itemize}
\item $g'(w_{1})$ is equivalent to a nonzero label of $H_{\text{max}}'$ other than $g'(w_{1}), g'(w_{2})$ for only finitely many rationals $q$ in $(a, b)$.
\item $g'(w_{2})$ is equivalent to a nonzero label of $H_{\text{max}}'$ other than $g'(w_{1}), g'(w_{2})$ for only finitely many rationals $q$ in $(a, b)$.
\item $g'(w_{1})$ is equivalent to $g'(w_{2})$ for only $q = g(u)/2$.
\end{itemize}
Therefore, since there are infinitely many rationals $q$ in $(a, b)$, there exists $q$ with the desired property.\\

Therefore, we can choose $q$ such that $g'(w_{1}), g'(w_{2})$ are not equivalent to each other or to any other nonzero label of $H_{\text{max}}'$ under $g'$.  We now prove that $g'$ satisfies the condition in (3).\\

\noindent\emph{Claim: $g'$ is injective.}\\

\noindent\emph{Proof of Claim.}  The labels $g'(w_{1}), g'(w_{2})$ are not equivalent, and hence not equal, to each other or to any other nonzero label of $H_{\text{max}}'$ under $g'$.  Then since $g$ is injective, $g'$ is also injective.\\

\noindent\emph{Claim: $g'$ is a labeling function of $H_{\text{max}}'$.}\\

\noindent\emph{Proof of Claim.}  We have $V_{g'} = V_{g}\cup\{g'(w_{1}), g'(w_{2})\}$ and $E_{g'} = E_{g}\cup\{|g'(u) - g'(w_{1})|, |g'(u) - g'(w_{2})|\}$.  Note that
\begin{align*}
|g'(u) - g'(w_{1})| &= g(u) - q = g'(w_{2})\\
|g'(u) - g'(w_{2})| &= g(u) - (g(u) - q) = g'(w_{1})
\end{align*}
Then since $V_{g} = E_{g}\cup\{0\}$, we also have $V_{g'} = E_{g'}\cup\{0\}$, so $g'$ is a labeling function of $H_{\text{max}}'$.\\

\noindent\emph{Claim: If $v_{1}, v_{2}$ are vertices of $H_{\text{max}}'$ with $g'(v_{1}) = 2^{k}g'(v_{2})$ for some positive integer $k$, then $v_{2}$ is a descendent of $v_{1}$, and $v_{1}$ is full in $H_{\text{max}}'$.}\\

\noindent\emph{Proof of Claim.}  By the hypotheses, $v_{1}, v_{2}$ are equivalent, so neither of $v_{1}, v_{2}$ is $w_{1}$ or $w_{2}$ by our choice of $q$.  Then since $g$ satisfies the condition in (3), $g'$ also satisfies the condition in (3).\\

Therefore, $H_{\text{max}}'$ satisfies (3), so $H_{\text{max}}'\in\mathcal{F}$, a contradiction to the maximality of $H_{\text{max}}$.  Therefore, in this case, we can conclude that $H_{\text{max}} = T$.\\

In both cases, we have $H_{\text{max}} = T$.  Therefore, $T\in\mathcal{F}$, so in particular, $T$ has a labeling function $g$ with $g(v) = 0$, as desired.
\end{proof}


\subsubsection{Application to original problem}

\begin{lemma} \label{rangerelaxed}
Let $T$ be a tree, and let $v$ be a vertex of $T$.  If $v$ is a leaf, then $T$ has a consistent range-relaxed graceful labeling $f$ with $f(v) = 0$.
\end{lemma}

\begin{proof}
By Lemma \ref{labelinglemma} and Theorem \ref{labelingtheorem}.
\end{proof}

\begin{theorem} \label{allrangerelaxed}
Let $T$ be a tree, and let $v$ be a vertex of $T$.  Then $T$ has a consistent range-relaxed graceful labeling $f$ with $f(v) = 0$.
\end{theorem}

\begin{proof}
We proceed by induction on the degree of $v$.
\begin{itemize}
\item If $\deg v = 0$, then $T$ is a single vertex, and the statement holds.
\item If $\deg v = 1$, then the statement follows from Lemma \ref{rangerelaxed}.
\end{itemize}
Suppose the statement holds for $\deg v = k$, and consider a tree $T$ with $\deg v = k + 1$.  Let $H_{2}$ be the subtree of $T$ corresponding to one branch of $T$ at $v$, and let $H_{1}$ be the subtree obtained by deleting from $T$ all vertices of $H_{2}$ other than $v$.  By induction, $H_{1}, H_{2}$ have consistent range-relaxed graceful labelings $f_{1}, f_{2}$ with $f_{1}(v) = 0$ and $f_{2}(v) = 0$.

Let $m$ be the maximum value of $f_{1}$.  Then define a function $f$ on the vertex set of $T$ as follows.
$$f(u) = \left\{
     \begin{array}{lr}
       0 & \text{if } u = v\\
       f_{1}(v) & \text{if $u$ is a vertex of $H_{1}$}\\
       (m + 1)f_{2}(v) & \text{if $u$ is a vertex of $H_{2}$}
     \end{array}
   \right.
$$
The multiplication by $m + 1$ ensures that $f$ is injective and hence a labeling.  Since $f_{1}, f_{2}$ are consistent range-relaxed graceful labelings, $f$ is also a consistent range-relaxed graceful labeling, as desired.  This completes the induction, so the statement holds for all $T$.
\end{proof}


\begin{corollary}
Let $T$ be a tree, and let $v$ be a vertex of $T$.  There exists $N$ such that, for all $k\ge N$, attaching $k$ leaves to $T$ at $v$ produces a graceful tree.
\end{corollary}

\begin{proof}
By Lemma \ref{addleaves} and Theorem \ref{allrangerelaxed}.
\end{proof}


\subsection{Trees with almost perfect matchings}

In section \ref{perfect_matching_prelim_section}, we used the $\Delta$-construction to prove the following result:
\begin{theorem}
{\normalfont (Broersma \& Hoede \cite{broersma1999another})} If $T$ has a perfect matching, and the contree of $T$ is graceful, then $T$ is graceful.
\end{theorem}
\noindent We now introduce the concept of an \emph{almost perfect matching}, extend the definition of a contree to trees with almost perfect matchings,
and use the $\Delta_{+1}$-construction to prove an analogous 
result.


\begin{definition}
We say that $M$ is an \emph{almost perfect} matching of $T$ if $M$ is incident with all but at most one of the vertices of $T$.
\end{definition}

\begin{definition}
Let $T$ be a tree, and let $M$ be an almost perfect matching of $T$.  The \emph{contree} of $T$ is the tree obtained from $T$ by contracting the edges in $M$.
\end{definition}

\begin{theorem}
If $T$ has an almost perfect matching, and the contree of $T$ is 0-rotatable, then $T$ is graceful.
\end{theorem}

\begin{proof}
Let $S$ be the contree of $T$, and let $n_{S}$ be the number of vertices of $S$.  The tree $T$ can be formed by the generalized $\Delta_{+1}$-construction, by attaching $n_{S} - 1$ copies of $P_{1}$, and one single vertex, all to each other, according to the structure of $S$.

Let $u$ be the vertex of $S$ corresponding to this single vertex.  Since $S$ is 0-rotatable, there is a graceful labeling of $S$ with $u$ labeled 0.  Therefore, by taking the complementary labeling, there is a graceful labeling of $S$ with $u$ labeled $n_{S} - 1$, as is required for the $\Delta_{+1}$-construction.

As in the proof of Theorem \ref{subdivide}, we must show that we can orient the copies of $P_{1}$ such that edges between different copies connect corresponding vertices.  To do this, we attach a leaf to $T$ at the vertex $u$, resulting in a tree with a perfect matching, and then apply Lemma \ref{subdivide_lemma}.  We orient the copies of $P_{1}$ according to the directions of the edges given by the lemma, as follows:
\begin{itemize}
\item If $u$ is the head of an edge in the matching, then for each copy of $P_{1}$, put the vertex corresponding to $v$ at the head of its edge.
\item If $u$ is the tail of an edge in the matching, then for each copy of $P_{1}$, put the vertex corresponding to $v$ at the tail of its edge.
\end{itemize}
This ensures that the copies of $P_{1}$ are oriented in the appropriate way, so the resulting tree is graceful by the generalized $\Delta_{+1}$-construction.
\end{proof}


\section{Conclusion}


In this thesis, we have introduced and developed three important new ideas:
\begin{enumerate}
\item Consider rearranging branches at all internal vertices. (3.7, 4.1)
\item All type-2 transfers are type-1 transfers in disguise, and hence can be removed from the discussion. (3.4)
\item Transfer backwards to manipulate the resulting labels. (3.5, 4.3, 4.4)
\end{enumerate}
Using these ideas, we have proved that several classes of diameter-6 trees are graceful.  More importantly, we have introduced new tools for proving that trees are graceful, built around the concept of a \emph{permutable sequence}.

If we restrict ourselves to the straightforward sequence of transfers in the first level, as in Lemma \ref{auxiliary_radial_tree}, then each tree corresponds to a \emph{blockwise permutable sequence (BPS)}.  However, if we use a more complicated sequence of transfers in the first level, as in Theorem \ref{rearrange_those_bad_boys}, then each tree corresponds to a more complicated permutable sequence.  These more complicated permutable sequence warrant further investigation.


To guide further research, we pose two conjectures.

\begin{conjecture}
Let $T$ be a rooted diameter-6 tree, with central vertex and root $v$, such that the vertex $v$, and all vertices of distance 1 from $v$, have an odd number of children.  Then $T$ has a graceful labeling $f$ with $f(v) = 0$.
\end{conjecture}

\noindent This conjecture has guided much of this thesis, and Theorems \ref{first_big_boy}, \ref{second_big_boy}, and \ref{rearrange_those_bad_boys} are special cases.  To get a sense of some obstructions to the conjecture, consider a diameter-6 tree corresponding to the following BPS:
$$((o, 0, 0), (o, 0, 0), (e))$$

\begin{conjecture}
Let $T$ be a rooted diameter-6 tree, with central vertex and root $v$, such that
\begin{itemize}
\item The vertex $v$ has an odd number of children.
\item Exactly four vertices of distance 1 from $v$ have an even number of children.
\end{itemize}
Then $T$ has a graceful labeling $f$ with $f(v) = 0$.
\end{conjecture}

\noindent This conjecture provides a relatively simple context to apply a more complicated sequence of transfers in the first level;  
the main challenge is to find a clean way to describe the resulting order of leaves.

\addcontentsline{toc}{section}{References}
\bibliographystyle{amsplain}
{\raggedright\bibliography{main}}

\end{document}